\documentclass{amsart}
\usepackage{amssymb,amsmath}
\usepackage{graphicx}
\theoremstyle{plain}
\newtheorem{thm}{Theorem}[section]

\newtheorem{cor}[thm]{Corollary}
\newtheorem{lem}[thm]{Lemma}
\newtheorem{prop}[thm]{Proposition}

\theoremstyle{definition}

\theoremstyle{remark}
\newtheorem{rem}{Remark}[section]
\newtheorem{ex}{Example}[section]

\numberwithin{equation}{section}

\newcommand{\ra}{\rightarrow}
\newcommand{\Ra}{\Rightarrow}

\newcommand{\Lra}{\Leftrightarrow}

\begin{document}

\title{Computing cutoff times of birth and death chains}

\author[G.-Y. Chen]{Guan-Yu Chen$^1$}

\author[L. Saloff-Coste]{Laurent Saloff-Coste$^2$}


\address{$^1$Department of Applied Mathematics, National Chiao Tung University, Hsinchu 300, Taiwan}
\email{gychen@math.nctu.edu.tw}

\address{$^2$Malott Hall, Department of Mathematics, Cornell University, Ithaca, NY 14853-4201}
\email{lsc@math.cornell.edu}



\begin{abstract}
Earlier work by Diaconis and Saloff-Coste gives a spectral criterion for a maximum separation cutoff to occur for birth and death chains. Ding, Lubetzky and Peres gave a related criterion for a maximum total variation cutoff to occur in the same setting. Here, we provide complementary results which allow us to compute the cutoff times and windows in a variety of examples.

\end{abstract}

\maketitle

\section{Introduction}\label{s-intro}

Let $\mathcal{X}$ be a finite set and $K$ be the transition matrix of a discrete time Markov chain on $\mathcal{X}$. For $t\in[0,\infty)$, set
\[
 H_t=e^{-t(I-K)}=e^{-t}\sum_{i=0}^\infty \frac{t^i}{i!}K^i.
\]
If $(X_m)_{m=0}^\infty$ is a Markov chain on $\mathcal{X}$ with transition matrix $K$ and $N_t$ is a Poisson process independent of $(X_m)_{m=0}^\infty$ with parameter $1$, then $H_t(x,\cdot)$ is the distribution of $X_{N_t}$ given $X_0=x$. It is well-known that if $K$ is irreducible with stationary distribution $\pi$, then
\[
 \lim_{t\ra\infty}H_t(x,y)=\pi(y),\quad\forall x,y\in\mathcal{X}.
\]
If $K$ is assumed further aperiodic, then
\[
 \lim_{m\ra\infty}K^m(x,y)=\pi(y),\quad\forall x,y\in\mathcal{X}.
\]
For simplicity, we use the triple $(\mathcal{X},K,\pi)$ to denote a discrete time irreducible Markov chain on $\mathcal{X}$ with transition matrix $K$ and stationary distribution $\pi$ and use $(\mathcal{X},H_t,\pi)$ to denote the associated continuous time chain introduced above.

In this paper, we consider the convergence of Markov chains in both total variation distance and separation. Let $\mu,\nu$ be two probabilities on $\mathcal{X}$. The total variation distance between $\mu,\nu$ and separation of $\mu$ w.r.t. $\nu$ are defined by
\[
 \|\mu-\nu\|_{\text{\tiny TV}}:=\max_{A\subset\mathcal{X}}\{\mu(A)-\nu(A)\},\quad
 \text{sep}(\mu,\nu):=\max_{x\in\mathcal{X}}\{1-\mu(x)/\nu(x)\}.
\]
With initial state $x$, the total variation distance and separation are defined by
\[
 d_{\text{\tiny TV}}(x,m):=\|K^m(x,\cdot)-\pi\|_{\text{\tiny TV}},\quad d_{\text{\tiny sep}}(x,m):=\text{sep}(K^m(x,\cdot),\pi).
\]
As these quantities are non-increasing in $m$, it is reasonable to consider the corresponding mixing time, which are defined by
\[
 T_{\text{\tiny TV}}(x,\epsilon):=\min\{m\ge 0|d_{\text{\tiny TV}}(x,m)\le\epsilon\}
\]
and
\[
 T_{\text{\tiny sep}}(x,\epsilon):=\min\{m\ge 0|d_{\text{\tiny sep}}(x,m)\le\epsilon\},
\]
for any $\epsilon\in(0,1)$. We define the maximum total variation distance and maximum separation by
\[
 d_{\text{\tiny TV}}(m):=\max_{x\in\mathcal{X}}d_{\text{\tiny TV}}(x,m),\quad d_{\text{\tiny sep}}(m):=\max_{x\in\mathcal{X}}d_{\text{\tiny sep}}(x,m).
\]
The corresponding mixing times are defined in a similar way and are denoted by $T_{\text{\tiny TV}}(\epsilon)$ and $T_{\text{\tiny sep}}(\epsilon)$. For the associated continuous time chains, we use $d_{\text{\tiny TV}}^{(c)}$, $d_{\text{\tiny sep}}^{(c)}$, $T_{\text{\tiny TV}}^{(c)}$ and $T_{\text{\tiny sep}}^{(c)}$. The inequalities,
\[
 d_{\text{\tiny TV}}(m)\le d_{\text{\tiny sep}}(m)\le 1-(1-2d_{\text{\tiny TV}}(m))^2,
\]
provide comparisons between the maximum total variation distance and maximum separation. As a consequence, one has
\[
 T_{\text{\tiny TV}}(\epsilon)\le T_{\text{\tiny sep}}(\epsilon)\le 2T_{\text{\tiny TV}}(\epsilon/4),\quad \forall \epsilon\in(0,1).
\]
Those results also apply for the continuous time chain and we refer the reader to \cite{AF} for detailed discussions and to \cite{LPW08} for various techniques in estimating the mixing times.

A birth and death chain on $\{0,1,...,n\}$ with transition rates $p_i,q_i,r_i$ is a Markov chain with transition matrix $K$ satisfying
\[
 K(i,i+1)=p_i,\quad K(i,i-1)=q_i,\quad K(i,i)=r_i,\quad\forall 0\le i\le n,
\]
where $p_i+q_i+r_i=1$ and $p_n=q_0=0$. Conventionally, $p_i,q_i,r_i$ are called the birth, death and holding rates at $i$. In the above setting, it is easy to see that $K$ is irreducible if and only if $p_iq_{i+1}>0$ for $0\le i<n$ and the unique stationary distribution $\pi$ satisfies $\pi(i)=c(p_0\cdots p_{i-1})/(q_1\cdots q_i)$, where $c$ is a normalizing constant such that $\sum_i\pi(i)=1$. Ding {\it et al.} proved in \cite{DLP10} that, over all initial states, separation is maximized when the chain starts at $0$ or $n$ and Diaconis and Saloff-Coste provided a formula for maximum separation in \cite{DS06}. As a consequence, the mixing time for maximum separation (and then for the maximum total variation distance) is comparable with the sum of reciprocals of non-zero eigenvalues of $I-K$. In \cite{CSal13-2}, Chen and Saloff-Coste showed that both mixing times are of the same order as the maximum expected hitting time to the median of $\pi$ over all initial distributions concentrated on the boundary points.

The cutoff phenomenon was first observed by Aldous and Diaconis in 1980s. For a formal definition, if $d$ is the total variation distance or separation either in the maximum case or with a specified initial state, a family of irreducible Markov chains $(\mathcal{X}_n,K_n,\pi_n)_{n=1}^\infty$ is said to present a cutoff in $d$, or a $d$-cutoff, if there is a sequence of positive integers $(t_n)_{n=1}^\infty$ such that
\[
 \forall \epsilon\in(0,1),\quad\lim_{n\ra\infty}\frac{T_{n,d}(\epsilon)}{t_n}=1,
\]
where $T_{n,d}$ is the mixing time in $d$ of the $n$th chain. A family that presents a cutoff in $d$ is said to have a $(t_n,b_n)$ cutoff in $d$ or a $(t_n,b_n)$ $d$-cutoff if $t_n>0,b_n>0$, $b_n/t_n\ra 0$ and
\[
 \forall \epsilon\in(0,1),\quad\limsup_{n\ra\infty}\frac{|T_{n,d}(\epsilon)-t_n|}{b_n}<\infty.
\]
In either case, the sequence $(t_n)_{n=1}^\infty$ is called a cutoff time and, in the latter case, the sequence $(b_n)_{n=1}^\infty$ is called the window with respect to $(t_n)_{n=1}^\infty$. The definition of cutoffs for families of continuous time chains is similar and we refer the reader to \cite{D96cutoff,CSal08} for an introduction and a detailed discussion of cutoffs.

Return to birth and death chains. To avoid the confusion of the total variation distances (resp. separation) in the maximum case and with a specified initial states, we use $\mathcal{F}$ and $\mathcal{F}_c$ for families of birth and death chains without starting states specified and write $\mathcal{F}^L,\mathcal{F}_c^L$ and $\mathcal{F}^R,\mathcal{F}_c^R$ respectively for families of chains started at the left and right boundary states. Diaconis and Saloff-Coste obtained in \cite{DS06} a spectral criterion for the existence of the separation cutoff and we cite part of their results in the following.

\begin{thm}\cite[Theorems 5.1-6.1]{DS06}\label{t-sep}
For $n=1,2,...$, let $K_n$ be the transition matrix of an irreducible birth and death chain on $\{0,1,...,n\}$ and $\lambda_{n,1},...,\lambda_{n,n}$ be the non-zero eigenvalues of $I-K_n$. Set
\[
 t_n=\sum_{i=1}^n\frac{1}{\lambda_{n,i}},\quad \lambda_n=\min_{1\le i\le n}\lambda_{n,i},\quad\sigma_n^2=\sum_{i=1}^n\frac{1}{\lambda_{n,i}^2},\quad \rho_n^2=\sum_{i=1}^n\frac{1-\lambda_{n,i}}{\lambda_{n,i}^2}.
\]
Let $\mathcal{F}$ be the family $(K_n)_{n=1}^\infty$ and $\mathcal{F}_c$ be the family of associated continuous time chains.
\begin{itemize}
\item[(1)] $\mathcal{F}_c^L$  has a separation cutoff if and only if $t_n\lambda_n\ra\infty$.

\item[(2)] Suppose $K_n(i,i+1)+K_n(i+1,i)\le 1$ for all $i,n$. Then, $\mathcal{F}^L$  has a separation cutoff if and only if $t_n\lambda_n\ra\infty$.
\end{itemize}

Furthermore, if $t_n\lambda_n\ra\infty$, then $\mathcal{F}_c^L$ has a $(t_n,\sigma_n)$ separation cutoff and, under the assumption of \textnormal{(2)}, $\mathcal{F}^L$ have a $(t_n,\max\{\rho_n,1\})$ separation cutoff.
\end{thm}

\begin{rem}\label{r-t-sep}
In Theorem \ref{t-sep}, the $(t_n,\max\{\rho_n,1\})$ separation cutoff of $\mathcal{F}^L$ is not discussed in \cite{DS06} but is an implicit result of the techniques therein. We give a proof of this fact in the appendix for completion. In the proof that there is a $(t_n,\max\{\rho_n,1\})$ separation cutoff, we show that
\[
 \mathcal{F}^L \text{ has a cutoff} \quad\Lra\quad \rho_n=o(t_n)\quad\Lra\quad
 \max\{\rho_n,1/\lambda_n\}=o(t_n).
\]
\end{rem}

\begin{rem}\label{r-sep2}
For any irreducible birth and death chain, it was proved in \cite{DLP10} that the maximum separation of the associated continuous time chain is attained when the initial state is any of the boundary states. This is also true for the discrete time case if the transition matrix $K$ satisfies $\min_iK(i,i)\ge 1/2$. As a result, if $\mathcal{F},\mathcal{F}_c$ and $t_n,\lambda_n$ are as in Theorem \ref{t-sep}, then
\begin{itemize}
\item[(1)] $\mathcal{F}_c$ has a maximum separation cutoff if and only if $t_n\lambda_n\ra\infty$.

\item[(2)] Assuming that $\inf_{i,n}K_n(i,i)\ge 1/2$, $\mathcal{F}$ has a maximum separation cutoff if and only if $t_n\lambda_n\ra\infty$.
\end{itemize}
\end{rem}


For cutoffs in the maximum total variation, Ding, Lubetzky and Peres provide the following criterion in \cite{DLP10}.

\begin{thm}\cite[Corollary 2 and Theorem 3]{DLP10}\label{t-tv}
Let $\mathcal{F},\mathcal{F}_c,\lambda_n$ be as in Theorem \ref{t-sep} and let $T_{n,\textnormal{\tiny TV}},T_{n,\textnormal{\tiny TV}}^{(c)}$ be the maximum total variation mixing time of the $n$th chains.
\begin{itemize}
\item[(1)] $\mathcal{F}_c$ has a maximum total variation cutoff if and only if $T_{n,\textnormal{\tiny TV}}^{(c)}(\epsilon)\lambda_n\ra\infty$ for some $\epsilon\in(0,1)$.

\item[(2)] Assume that $\inf_{i,n}K_n(i,i)>0$. Then, $\mathcal{F}$ has a maximum total variation cutoff if and only if $T_{n,\textnormal{\tiny TV}}(\epsilon)\lambda_n\ra\infty$ for some $\epsilon\in(0,1)$.
\end{itemize}
\end{thm}

\begin{rem}\label{r-tv}
For any birth and death chain started at the left or right boundary state, the total variation distance can be different and the biased random walk with constant birth and death rates is a typical example. Further, the maximum total variation distance over all initial states is not necessarily attained at boundary states and a birth and death chain with valley stationary distribution, a distribution which is decreasing on $\{0,...,M\}$ and increasing on $\{M,...,n\}$ for some $0<M<n$, could illustrate this observation. This is very different from the case of separation and we refer the readers to Sections \ref{s-bs} and \ref{s-comp} for more discussions.
\end{rem}

To state our main results, we need the following notation. For $n\in\mathbb{N}$, let $\mathcal{X}_n=\{0,1,...,n\}$ and $(X_m^{(n)})_{m=0}^\infty$ be an irreducible birth and death chain on $\mathcal{X}_n$ with transition matrix $K_n$ and stationary distribution $\pi_n$. Let $N_t$ be a Poisson process independent of $(X_m^{(n)})$ with parameter $1$. For $i\in\mathcal{X}_n$, set
\begin{equation}\label{eq-tauc}
 \tau^{(n)}_i=\inf\{m\ge 0|X_m^{(n)}=i\},\quad\widetilde{\tau}^{(n)}_i=\inf\{t\ge 0|X_{N_t}^{(n)}=i\}.
\end{equation}
For $j\in\mathcal{X}_n$, let $\mathbb{E}_j$ and $\text{Var}_j$ denote the conditional expectation and variance given $X^{(n)}_0=j$.

\begin{rem}
It follows from the definition of $\tau_i^{(n)},\widetilde{\tau}_i^{(n)}$ that $\mathbb{E}_j\tau^{(n)}_i=\mathbb{E}_j\widetilde{\tau}^{(n)}_i$ for all $i,j\in\mathcal{X}_n$. See \cite{AF} for more information of the hitting times $\tau_i^{(n)},\widetilde{\tau}_i^{(n)}$.
\end{rem}

\begin{thm}\label{t-sep2}
Let $\mathcal{F},\mathcal{F}_c,\lambda_n$ be as in Theorem \ref{t-sep} and $\tau^{(n)}_i,\widetilde{\tau}^{(n)}_i$ be the hitting times in \textnormal{(\ref{eq-tauc})}. For $n\ge 1$, let $M_n\in\{0,1,...,n\}$ and set
\[
 s_n=\mathbb{E}_0\widetilde{\tau}^{(n)}_{M_n}+\mathbb{E}_n\widetilde{\tau}^{(n)}_{M_n}
 =\mathbb{E}_0\tau^{(n)}_{M_n}+\mathbb{E}_n\tau^{(n)}_{M_n}
\]
and
\[
 b_n^2=\textnormal{Var}_0\widetilde{\tau}^{(n)}_{M_n}
 +\textnormal{Var}_n\widetilde{\tau}^{(n)}_{M_n},\quad
 c_n^2=\textnormal{Var}_0\tau^{(n)}_{M_n}
 +\textnormal{Var}_n\tau^{(n)}_{M_n}.
\]
Suppose that
\begin{equation}\label{eq-pin}
 \inf_{n\ge 1}\pi_n([0,M_n])>0,\quad\inf_{n\ge 1}\pi_n([M_n,n])>0.
\end{equation}
Then, the following properties hold.
\begin{itemize}
\item[(1)] $\mathcal{F}_c^L$ has a separation cutoff if and only if $s_n\lambda_n\ra\infty$ if and only if $s_n/b_n\ra\infty$. Furthermore, if $s_n/b_n\ra\infty$, then $\mathcal{F}_c^L$ has a $(s_n,b_n)$ separation cutoff.

\item[(2)] Assume that $K_n(i,i+1)+K_n(i+1,i)\le 1$ for all $i,n$. Then, $\mathcal{F}^L$ has a separation cutoff if and only if $s_n\lambda_n\ra\infty$ if and only if $s_n/c_n\ra\infty$. Furthermore, if $s_n/c_n\ra\infty$, then $\mathcal{F}^L$ has a $(s_n,\max\{c_n,1/\lambda_n\})$ separation cutoff.

\end{itemize}
\end{thm}

\begin{rem}\label{r-sep}
Let $\sigma_n,\rho_n$ be the constants in Theorem \ref{t-sep}. Let $M_n,M_n'\in\{0,1,...,n\}$ and $b_n,c_n,b_n',c_n'$ be the constants in Theorem \ref{t-sep2} defined accordingly. Suppose $M_n,M_n'$ satisfy (\ref{eq-pin}). Then,
\[
 b_n\asymp b_n'\asymp \sigma_n,\quad \max\{c_n,1/\lambda_n\}\asymp\max\{c_n',1/\lambda_n\}\asymp\max\{\rho_n,1/\lambda_n\},
\]
where $u_n\asymp v_n$ means that both sequences, $u_n/v_n$ and $v_n/u_n$, are bounded. See Corollary \ref{c-sep} for a proof. Comparing Theorems \ref{t-sep} and \ref{t-sep2}, one can see that the cutoff window for $\mathcal{F}_c^L$ is unchanged up to some universal multiples but the cutoff window for $\mathcal{F}^L$ can have a bigger order in Theorem \ref{t-sep2} due to the change of the cutoff time.
\end{rem}

In total variation, we have the following result.

\begin{thm}\label{t-tv2}
Let $\mathcal{F},\mathcal{F}_c,\lambda_n$ be as in Theorem \ref{t-sep} and $\tau^{(n)}_i,\widetilde{\tau}^{(n)}_i$ be the hitting times in \textnormal{(\ref{eq-tauc})}. Let $M_n\in\{0,1,...,n\}$ and set
\[
 \theta_n=\max\left\{\mathbb{E}_0\tau^{(n)}_{M_n},\mathbb{E}_n\tau^{(n)}_{M_n}\right\}
 =\max\left\{\mathbb{E}_0\widetilde{\tau}^{(n)}_{M_n},\mathbb{E}_n\widetilde{\tau}^{(n)}_{M_n}\right\}
\]
and
\[
 \alpha_n^2=\max\left\{\textnormal{Var}_0\widetilde{\tau}^{(n)}_{M_n},
 \textnormal{Var}_n\widetilde{\tau}^{(n)}_{M_n}\right\}
\]
and
\[
 \beta_n^2=\max\left\{\textnormal{Var}_0\tau^{(n)}_{M_n},
 \textnormal{Var}_n\tau^{(n)}_{M_n}\right\}.
\]
Suppose
\begin{equation}\label{eq-pin2}
 \inf_{n\ge 1}\pi_n([0,M_n])>0,\quad\inf_{n\ge 1}\pi_n([M_n,n])>0.
\end{equation}
In the maximum total variation distance:
\begin{itemize}
\item[(1)] $\mathcal{F}_c$ has a cutoff if and only if $\theta_n\lambda_n\ra\infty$ if and only if $\theta_n/\alpha_n\ra\infty$. Furthermore, if $\mathcal{F}_c$ has a cutoff, then $\mathcal{F}_c$ has a $(\theta_n,\alpha_n)$ cutoff.

\item[(2)] Assume that $\inf_{i,n}K_n(i,i)>0$. Then, $\mathcal{F}$ has a cutoff if and only if $\theta_n\lambda_n\ra\infty$ if and only if $\theta_n/\beta_n\ra\infty$. Furthermore, if $\mathcal{F}$ has a cutoff, then $\mathcal{F}$ has a $(\theta_n,\beta_n)$ cutoff.
\end{itemize}
\end{thm}

\begin{rem}\label{r-wcomp2}
In Theorem \ref{t-tv2}, if $\delta=\inf_{i,n}K_n(i,i)$, then $\delta\alpha_n^2\le \beta_n^2\le\alpha_n^2$. See Remark \ref{r-wcomp} for details.
\end{rem}

\begin{rem}\label{r-exp1}
Let $\mathcal{F}=(\mathcal{X}_n,K_n,\pi_n)_{n=1}^\infty$ be a family of irreducible birth and death chains with $\mathcal{X}_n=\{0,1,...,n\}$. For $a\in(0,1)$, set $M_n(a)$ be a state in $\mathcal{X}_n$ satisfying
\[
 \pi_n([0,M_n(a)])\ge a,\quad\pi_n([M_n(a),n])\ge 1-a.
\]
By Theorem \ref{t-sep} and Remark \ref{r-sep2}, if $\mathcal{F}_c$ has a cutoff in maximum separation, then
\begin{equation}\label{eq-limhit1}
 \lim_{n\ra\infty}\frac{\mathbb{E}_0\widetilde{\tau}^{(n)}_{M_n(a)}
 +\mathbb{E}_n\widetilde{\tau}^{(n)}_{M_n(a)}}
 {\mathbb{E}_0\widetilde{\tau}^{(n)}_{M_n(b)}+\mathbb{E}_n\widetilde{\tau}^{(n)}_{M_n(b)}}=1,
 \quad\forall 0<a<b<1.
\end{equation}
From Theorem \ref{t-tv2}, if $\mathcal{F}_c$ has a cutoff in the maximum total variation, then
\begin{equation}\label{eq-limhit2}
 \lim_{n\ra\infty}\frac{\max\{\mathbb{E}_0\widetilde{\tau}^{(n)}_{M_n(a)},
 \mathbb{E}_n\widetilde{\tau}^{(n)}_{M_n(a)}\}}
 {\max\{\mathbb{E}_0\widetilde{\tau}^{(n)}_{M_n(b)},\mathbb{E}_n\widetilde{\tau}^{(n)}_{M_n(b)}\}}=1,
 \quad\forall 0<a<b<1.
\end{equation}
But, the converse of these statements are not necessarily true. For example, let
\[
 K_n(i,i+1)=K_n(i+1,i)=1/2,\quad\forall 0<i<n,\quad K_n(n,n)=1/2,
\]
and
\[
 K_n(0,1)=K_n(1,0)=\xi_n,\quad K_n(0,0)=1-\xi_n,\quad K_n(1,1)=1/2-\xi_n,
\]
where $\xi_n\in(0,1/2)$. Note that $K_n$ can be regarded as the transition matrix of a simple random walk on $\mathcal{X}_n$ with specific transitions at the boundary states and a bottleneck between $0$ and $1$ when $\xi_n$ is small. It is clear that the stationary distribution satisfies $\pi_n(i)=1/(n+1)$ for all $0\le i\le n$. After some computations, one has, for $n$ large enough,
\[
 M_n(a)\asymp n\asymp (n-M_n(a)).
\]
This implies
\[
 \mathbb{E}_0\widetilde{\tau}^{(n)}_{M_n(a)}=\frac{1}{\xi_n}+M_n(a)(M_n(a)+1)-2\asymp \frac{1}{\xi_n}+n^2
\]
and
\[
 \mathbb{E}_n\widetilde{\tau}^{(n)}_{M_n(a)}=(n-M_n(a))[n-M_n(a)+1]\asymp n^2.
\]
Let $p_{n,i},q_{n,i},r_{n,i}$ and $\lambda_n$ be the transition rates and the spectral gap of $K_n$. By Theorem 1.2 in \cite{CSal13-2}, we have
\begin{equation}\label{eq-gap}
 \frac{1}{\lambda_n}\asymp\max\left\{\max_{j:j<M_n}\sum_{k=j}^{M_n-1}\frac{\pi_n([0,j])}{\pi_n(k)p_{n,k}},
 \max_{j:j>M_n}\sum_{k=M_n+1}^j\frac{\pi_n([j,n])}{\pi_n(k)q_{n,k}}\right\},
\end{equation}
where $M_n=\lfloor n/2\rfloor$. This implies
\[
 \frac{1}{\lambda_n}\asymp \frac{1}{\xi_n}+n^2.
\]
As a consequence of Theorems \ref{t-sep2} and \ref{t-tv2}, $\mathcal{F}_c$ has neither a maximum separation cutoff nor a maximum total variation cutoff. Let $s_n$ and $\theta_n$ be the constants in Theorems \ref{t-sep2} and \ref{t-tv2}. If $n^2\xi_n\ra 0$, then
\[
 s_n\sim \theta_n\sim\mathbb{E}_0\widetilde{\tau}^{(n)}_{M_n(a)}\sim\frac{1}{\xi_n},\quad\forall a\in(0,1).
\]
The above example illustrates that (\ref{eq-limhit1}) and (\ref{eq-limhit2}) are necessary but not sufficient for the existence of the corresponding cutoffs.
\end{rem}

The following theorem describes one of the main applications of Theorems \ref{t-sep2}-\ref{t-tv2}.

\begin{thm}\label{t-randomm}
Consider a family $\mathcal{F}=(\mathcal{X}_n,K_n,\pi_n)_{n=1}^\infty$ of irreducible birth and death chains with $\mathcal{X}_n=\{0,1,...,n\}$. For $n\ge 1$, let $(\Omega_n,\mathbb{P}^{(n)})$ be a probability space and $C_{n,1},...,C_{n,n}:\Omega_n\ra(0,1)$ be independent and identically distributed random variables. For $\omega_n\in\Omega_n$ and $0\le i\le n$, let $(\mathcal{X}_n,L_n^{(\omega_n)},\pi_n)$ be a Markov chain given by
\[
 \begin{cases}
 L^{(\omega_n)}_n(i,i+1)=K_n(i,i+1)C_{n,i+1}(\omega_n),\\
 L^{(\omega_n)}_n(i+1,i)=K_n(i+1,i)C_{n,i+1}(\omega_n),\\
 L^{(\omega_n)}_n(i,i)=1-L^{(\omega_n)}_n(i,i+1)-L^{(\omega_n)}_n(i,i-1),
 \end{cases}
\]
and, for $\omega=(\omega_1,\omega_2,...)\in\prod_{n=1}^\infty\Omega_n$, let $\mathcal{F}^{(\omega)}=(\mathcal{X}_n,L_n^{(\omega_n)},\pi_n)_{n=1}^\infty$. Let $\mathcal{F}_c,\mathcal{F}_c^{(\omega)}$ be the continuous time families associated with $\mathcal{F},\mathcal{F}^{(\omega)}$. For $n\ge 1$, set $\mu_n=\mathbb{E}(1/C_{n,1})$, $\nu_n^2=\textnormal{Var}(1/C_{n,1})$ and let $\theta_n,\alpha_n,\beta_n$ be the constants in Theorem \ref{t-tv2}.
\begin{itemize}
\item[(1)] If $\mathcal{F}_c$ has a maximum total variation cutoff and $\nu_n\alpha_n=o(\mu_n\theta_n)$, then there is a sequence $E_n\subset\Omega_n$ such that $\mathbb{P}^{(n)}(E_n)\ra 1$ and, for any $\omega\in\prod_{n=1}^\infty E_n$, $\mathcal{F}^{(\omega)}_c$ has a maximum total variation cutoff with cutoff time $\mu_n\theta_n$.

\item[(2)] Assuming $\inf_{n,i}K_n(i,i)>0$ and replacing $\alpha_n$ by $\beta_n$, the statement in (1) also holds for the families $\mathcal{F},\mathcal{F}^{(\omega)}$.
\end{itemize}
\end{thm}

\begin{rem}
In Theorem \ref{t-randomm}, $L_n$ can be regarded as a random birth and death chain obtained by applying i.i.d. random slowdowns on $K_n$ without changing the stationary distribution.
\end{rem}

\begin{rem}
Theorem \ref{t-randomm} also holds in maximum separation.
\end{rem}

The remaining of this article is organized in the following way. Sections \ref{s-sep} and \ref{s-tv} contain the proofs of Theorems \ref{t-sep2} and \ref{t-tv2} respectively. The proof of Theorem \ref{t-randomm} is given in Section \ref{s-rbdc}. We also introduce another randomization of simple random walks on paths and discuss its cutoff and mixing time. In Section \ref{s-bs}, we consider families of chains started at one boundary states and provide criteria for the existence of a total variation cutoff and formulas for the cutoff time. We discuss the distinction between maximum total variation cutoffs and cutoffs from a boundary state and illustrate this with several examples in Section \ref{s-comp}. The main results of Section \ref{s-bs} are proved in Section \ref{s-proof}. In Section \ref{s-ex}, we apply the developed theory to compute the cutoff time of some classical examples. Some useful lemmas and auxiliary results are gathered in the appendix.


\section{Cutoff in separation}\label{s-sep}

This section is dedicated to the proof of Theorem \ref{t-sep2} and we need the following two lemmas. The first lemma concerns the mean and variance of hitting times and the second lemma provides a comparison of spectral gaps.

\begin{lem}\label{l-meanvar}
Let $K$ be the transition matrix of an irreducible birth and death chain on $\{0,1,...,n\}$. For $1\le i\le n$, let $\beta^{(i)}_1,...,\beta^{(i)}_i$ be the eigenvalues of the submatrix of $I-K$ indexed by $\{0,...,i-1\}$ and set
\begin{equation}\label{eq-taui}
 \tau_i=\min\{m\ge 0|X_m=i\},\quad\widetilde{\tau}_i=\inf\{t\ge 0|X_{N_t}=i\},
\end{equation}
where $(X_m)_{m=0}^\infty$ is a Markov chain with transition matrix $K$ and $N_t$ is a Poisson process independent of $X_m$ with parameter $1$. Then, $\beta^{(i)}_j\in(0,2)$ for all $1\le j\le i$ and
\begin{equation}\label{eq-mean}
 \mathbb{E}_0\tau_i=\mathbb{E}_0\widetilde{\tau}_i=\sum_{j=1}^i\frac{1}{\beta^{(i)}_j},
\end{equation}
and
\begin{equation}\label{eq-var}
 \textnormal{Var}_0(\tau_i)=\sum_{j=1}^i\frac{1-\beta^{(i)}_j}{\left(\beta^{(i)}_j\right)^2},\quad
\begin{aligned}
 \textnormal{Var}_0(\widetilde{\tau}_i)=\sum_{j=1}^i\frac{1}{\left(\beta^{(i)}_j\right)^2}.
\end{aligned}
\end{equation}
\end{lem}
\begin{proof}
Let $\widetilde{K}$ be the submatrix of $K$ indexed by $\{0,1,...,i-1\}$. Let $\beta$ be an eigenvalue of $\widetilde{K}$ and $x=(x_0,...,x_{i-1})$ be a left eigenvector associated with $\beta$. That is,
\[
 \begin{cases}\beta x_j=K(j-1,j)x_{j-1}+K(j,j)x_j+K(j+1,j)x_{j+1},\quad\forall 0<j<i-1,\\
 \beta x_0=K(0,0)x_0+K(1,0)x_1,\\
 \beta x_{i-1}=K(i-2,i-1)x_{i-2}+K(i-1,i-1)x_{i-1}.\end{cases}
\]
By the irreducibility of $K$, if $x_{i-1}=0$, then $x_j=0$ for all $0\le j<i$. This implies $x_{i-1}\ne 0$ and then
\[
 |\beta|\sum_{j=0}^{i-1}|x_j|\le\sum_{j=0}^{i-1}|x_j|-K(i-1,i)|x_{i-1}|<\sum_{j=0}^{i-1}|x_j|.
\]
Since $x$ is an eigenvector of $K$, $\sum_j|x_j|>0$ and thus $|\beta|<1$. This proves that $\beta^{(i)}_j\in(0,2)$ for all $1\le j\le i$. For (\ref{eq-mean}) and (\ref{eq-var}), note that the distribution of $\widetilde{\tau}_i$ was given by Brown and Shao in \cite{BS87} and the technique therein also applies for $\tau_i$. This leads to the desired identities, where we refer the reader to their work for details.
\end{proof}

\begin{rem}\label{r-var}
In Lemma \ref{l-meanvar}, the first equality of (\ref{eq-var}) implies
\[
 \sum_{j=1}^i\frac{1}{(\beta^{(i)}_j)^2}\ge \sum_{j=1}^i\frac{1}{\beta^{(i)}_j},\quad\forall j\ge 1.
\]
\end{rem}

\begin{lem}\label{l-gap}
Let $K$ be the transition matrix of an irreducible birth and death chain on $\{0,1,...,n\}$ with stationary distribution $\pi$. For $0\le i\le n$, let $L_i$ be the sub-matrix of $K$ obtained by removing the row and column of $K$ indexed by state $i$. Let $\lambda_1<\cdots<\lambda_n$ be the non-zero eigenvalues of $I-K$ and $\lambda^{(i)}_1\le\cdots\le\lambda^{(i)}_n$ be the eigenvalues of $I-L_i$. Then,
\[
 \lambda^{(i)}_j\le \lambda_j\le\lambda^{(i)}_{j+1}\le\lambda_{j+1},\quad\forall 1\le j<n,
\]
and
\[
 \left(\frac{\min\{\pi([0,i]),\pi([i,n])\}}{4}\right)\lambda_1\le\lambda^{(i)}_1\le\lambda_1.
\]
In particular, if $M$ is a median of $\pi$, i.e. $\pi([0,M])\ge 1/2$ and $\pi([M,n])\ge 1/2$, then $\lambda_1/8\le \lambda^{(M)}_1\le\lambda_1$.
\end{lem}

The proof of Lemma \ref{l-gap} is based on a weighted Hardy inequality obtained in \cite{CSal13-2} and is discussed in the appendix. In what follows, for any two sequences of positive reals $a_n,b_n$, we write $a_n=o(b_n)$ if $a_n/b_n\ra 0$ and write $a_n=O(b_n)$ if $a_n/b_n$ is bounded. In the case that $a_n=O(b_n)$ and $b_n=O(a_n)$, we write $a_n\asymp b_n$ instead.

\begin{proof}[Proof of Theorem \ref{t-sep2}]
Let $\lambda_{n,i},\lambda_n,t_n,\sigma_n,\rho_n$ be constants in Theorem \ref{t-sep}. Note that, for $n\ge 2$,
\[
 \max\{\rho_n^2,1/\lambda_n^2\}\le \sigma_n^2=\sum_{i=1}^n\frac{1}{\lambda_{n,i}^2}\le \frac{t_n}{\lambda_n}.
\]
This implies
\begin{equation}\label{eq-comp}
 \sqrt{t_n\lambda_n}\le\frac{t_n}{\sigma_n}\le\frac{t_n}{\max\{\rho_n,1/\lambda_n\}}\le t_n\lambda_n.
\end{equation}
As a consequence, we have
\begin{equation}\label{eq-equi}
 t_n\lambda_n\ra\infty\quad\Lra\quad \sigma_n=o(t_n)\quad\Lra\quad \max\{\rho_n,1/\lambda_n\}=o(t_n).
\end{equation}

Next, let $s_n,b_n,c_n$ be constants in Theorem \ref{t-sep2}. Observe that
\[
 1/\lambda_n\le\max\{\rho_n,1/\lambda_n\}\le \sigma_n.
\]
Set $a_n=\min\{\pi_n([0,M_n]),\pi_n([M_n,n])\}$. By Lemmas \ref{l-meanvar} and \ref{l-gap}, one has
\[
 t_n\le s_n\le t_n+\frac{4}{a_n\lambda_n}\le t_n+\frac{4\sigma_n}{a_n}
\]
and
\[
 \sigma_n^2\le b_n^2\le \sigma_n^2+\left(\frac{4}{a_n\lambda_n}\right)^2
 \le\frac{17\sigma_n^2}{a_n^2}.\notag
\]
According to the assumption of (\ref{eq-pin}), we have $a_n\asymp 1$ and this implies
\[
 t_n\lambda_n\ra\infty\quad\Lra\quad s_n\lambda_n\ra\infty
\]
and
\begin{equation}\label{eq-equi2}
 |t_n-s_n|=O(\sigma_n),\quad |t_n-s_n|=O(\max\{\rho_n,1/\lambda_n\}),\quad
 b_n\asymp \sigma_n.
\end{equation}
As a consequence of (\ref{eq-equi}) and (\ref{eq-equi2}), we obtain
\begin{equation}\label{eq-equi3}
 t_n\lambda_n\ra\infty\quad\Lra\quad b_n=o(s_n)\quad\Lra\quad \max\{c_n,1/\lambda_n\}=o(s_n).
\end{equation}
The first equivalence of (\ref{eq-equi3}) proves the criterion for cutoff in (1). For (2), if $\mathcal{F}^L$ has a separation cutoff, then Theorem \ref{t-sep} implies $t_n\lambda_n\ra\infty$. By the last identity in (\ref{eq-equi3}), we obtain $c_n=o(s_n)$. To see the inverse direction, observe that the mapping $u\mapsto (1-u)/u^2$ is decreasing on $(0,2]$ and $\lambda_{n,i}\in(0,2)$ for all $1\le i\le n$. In the same reasoning as before, Lemmas \ref{l-meanvar} and \ref{l-gap} yield
\begin{equation}\label{eq-equi4}
 \rho_n^2\le c_n^2\le\rho_n^2+\frac{1-a_n\lambda_n/4}{(a_n\lambda_n/4)^2}
 +\frac{\lambda_{n,n}-1}{\lambda_{n,n}^2}\le \rho_n^2+\frac{17}{a_n^2\lambda_n^2}.
\end{equation}
By the first inequality of (\ref{eq-equi4}), if $c_n=o(s_n)$, then $\rho_n=o(s_n)$. Accompanied with the facts,
\[
 s_n=t_n+\frac{4}{a_n\lambda_n}\le\left(1+\frac{4}{a_n}\right)t_n,\quad a_n\asymp 1,
\]
we obtain $\rho_n=o(t_n)$. By Remark \ref{r-t-sep}, $\mathcal{F}^L$ has a separation cutoff.

To see a window, we recall Corollary 2.5(v) of \cite{CSal08}, which says that if a family has a $(t_n,\sigma_n)$ cutoff and
\[
 b_n=o(t_n)\,\, (\text{or }b_n=o(s_n)),\quad |t_n-s_n|=O(b_n),\quad \sigma_n=O(b_n),
\]
then the family has a $(s_n,b_n)$ cutoff. By Theorem \ref{t-sep}, the desired cutoff for $\mathcal{F}^L_c$ is given by the first and third identities in (\ref{eq-equi2}), while the desired cutoff for $\mathcal{F}^L$ is provided by the second identity in (\ref{eq-equi2}), the third identity in (\ref{eq-equi3}) and the following observations
\[
 \max\{\rho_n,1/\lambda_n\}\asymp \max\{c_n,1/\lambda_n\},\quad \max\{\rho_n,1\}=O(\max\{c_n,1/\lambda_n\}),
\]
which are implies by (\ref{eq-equi4}) and the fact $\lambda_n\le 2$.
\end{proof}

In the following corollary, we summarize some useful comparison between the variances of hitting times and the windows of cutoffs obtained in the proof of Theorem \ref{t-sep2}.

\begin{cor}\label{c-sep}
Let $K$ be the transition matrix of an irreducible birth and death chain on $\{0,1,...,n\}$ with stationary distribution  $\pi$ and $\tau_i,\widetilde{\tau}_i$ be the hitting times in \textnormal{(\ref{eq-taui})}. Suppose $\lambda_1,...,\lambda_n$ be non-zero eigenvalues of $I-K$ and set
\[
 t=\sum_{i=1}^n\frac{1}{\lambda_i},\quad\sigma^2=\sum_{i=1}^n\frac{1}{\lambda_i^2},\quad \rho^2=\sigma^2-t,
 \quad \lambda=\min_{1\le i\le n}\lambda_i.
\]
Then, for $0\le i\le n$,
\[
 t\le\mathbb{E}_0\widetilde{\tau}_i+\mathbb{E}_n\widetilde{\tau}_i=\mathbb{E}_0\tau_i+\mathbb{E}_n\tau_i
 \le t+\frac{4}{a(i)\lambda}
\]
and
\[
 \sigma^2\le\textnormal{Var}_0\widetilde{\tau}_i+\textnormal{Var}_n\widetilde{\tau}_i\le\frac{17\sigma^2}
 {a(i)^2},\quad \rho^2\le \textnormal{Var}_0\tau_i+\textnormal{Var}_n\tau_i\le\rho^2+\frac{17}{a(i)^2\lambda^2},
\]
where $a(i)=\min\{\pi([0,i]),\pi([i,n])\}$.
\end{cor}

To determine a cutoff time and a window using Theorem \ref{t-sep2}, one needs to compute the mean and variance of the hitting time to some state given that the chain starts at one boundary state. Explicit formulas on both terms are available using the Markov property and we summarize them in Lemma \ref{l-meanvar2}.

The next proposition discusses the cutoff times obtained in Theorem \ref{t-sep2} and provides a universal lower bound on the corresponding windows using the transition rates and the stationary distribution.

\begin{prop}\label{p-hittime}
Let $K$ be the transition matrix of a birth and death chain on $\{0,1,...,n\}$ with transition rates $p_i,q_i,r_i$. Let $\tau_i,\widetilde{\tau}_i$ be the hitting times in \textnormal{(\ref{eq-taui})} and set
\[
 s(i)=\mathbb{E}_0\widetilde{\tau}_i+\mathbb{E}_n\widetilde{\tau}_i,\quad b(i)^2=\textnormal{Var}_0(\widetilde{\tau}_i)
 +\textnormal{Var}_n(\widetilde{\tau}_i).
\]
Suppose $K$ is irreducible with stationary distribution $\pi$ and spectral gap $\lambda$. Let $M\in\{0,1,...,n\}$ be a state satisfying $\pi([0,M])\ge 1/2$ and $\pi([M,n])\ge 1/2$. Then, for $0\le i\le j\le M$,
\begin{equation}\label{eq-hittime1}
 s(i)-s(j)=\sum_{\ell=i}^{j-1}\frac{1-2\pi([0,\ell])}{p_\ell\pi(\ell)}\ge 0,
\end{equation}
and, for $0\le i\le n$,
\begin{equation}\label{eq-hittime2}
 b(i)\ge\frac{1}{\lambda}\ge\frac{1}{2}\max_{0\le j\le M\le k\le n}\max\left\{\sum_{\ell=j}^{M-1}\frac{\pi([0,j])}{p_\ell\pi(\ell)},
 \sum_{\ell=M+1}^k\frac{\pi([k,n])}{q_\ell\pi(\ell)}\right\}.
\end{equation}
\end{prop}
\begin{proof}
(\ref{eq-hittime1}) is given by Lemma \ref{l-meanvar2} and the first inequality of (\ref{eq-hittime2}) is obvious from Lemmas \ref{l-meanvar}-\ref{l-gap}, while the second inequality of (\ref{eq-hittime2}) is cited from Theorem A.1 of \cite{CSal13-2}.
\end{proof}

\begin{rem}
Let $s_n,t_n$ be the constants in Theorems \ref{t-sep}-\ref{t-sep2}. By Corollary \ref{c-sep}, one has $s_n-t_n\ge 0$ and, by (\ref{eq-hittime1}), the difference $s_n-t_n$ is minimized when $M_n$ satisfies
\[
 \pi_n([0,M_n])\ge 1/2,\quad \pi_n([M_n,n])\ge 1/2.
\]
\end{rem}


\section{Cutoff in total variation}\label{s-tv}

This section is dedicated to the proof of Theorem \ref{t-tv2}. Throughout the rest of this article, we will write $\mathbb{P}_i$ to denote the probability given the initial state $i$. First, recall two useful bounds on the total variation.

\begin{lem}\cite[Proposition 3.8 and Equation (3.5)]{CSal13-2}\label{l-tv}
Consider a continuous time birth and death chain on $\{0,1,...,n\}$ with stationary distribution $\pi$. For $0\le i\le n$, let $\widetilde{\tau}_i$ be the first hitting time to state $i$ and $d^{(c)}_{\textnormal{\tiny TV}}(i,t)$ be the total variation distance at time $t$ with initial state $i$. Then, for $0\le i\le n$ and $0\le j\le k\le n$,
\[
 d^{(c)}_{\textnormal{\tiny TV}}(i,t)\le \mathbb{P}_i(\max\{\widetilde{\tau}_j,\widetilde{\tau}_k\}>t)+1-\pi([j,k])
\]
and
\[
 d^{(c)}_{\textnormal{\tiny TV}}(0,t)\ge\mathbb{P}_0(\widetilde{\tau}_i>t)-\pi([0,i-1]).
\]
\end{lem}

Based on the above lemma, we may bound the maximum total variation mixing time using the expected hitting times.

\begin{thm}\label{t-tvmix}
Let $\pi,\widetilde{\tau}_i$ be as in Lemma \ref{l-tv} and set
\[
 \theta(i)=\max\{\mathbb{E}_0\widetilde{\tau}_i,\mathbb{E}_n\widetilde{\tau}_i\},\quad \alpha(i)^2=\max\{\textnormal{Var}_0\widetilde{\tau}_i,\textnormal{Var}_n\widetilde{\tau}_i\}.
\]
The maximum total variation mixing time satisfies
\[
 T^{(c)}_{\textnormal{\tiny TV}}(\epsilon_1)\le \theta(j)+\mathbb{E}_j\widetilde{\tau}_k+\mathbb{E}_k\widetilde{\tau}_j
 +\sqrt{\left(\tfrac{2}{\delta}-1\right)}\max\{\alpha(j),\alpha(k)\}
\]
and
\[
 T^{(c)}_{\textnormal{\tiny TV}}(\epsilon_2)\ge\theta(j)-\mathbb{E}_k\widetilde{\tau}_j
 -\sqrt{\left(\tfrac{1}{\delta}-1\right)}\max\{\alpha(j),\alpha(k)\},
\]
for any $0\le j\le k\le n$ and $\delta\in(0,1)$, where $\epsilon_1=1-\pi([j,k])+\delta$ and $\epsilon_2=\min\{\pi([j,n]),\pi([0,k])\}-\delta$.
\end{thm}

\begin{proof}
We first consider the upper bound. Set $\epsilon_1=1-\pi([j,k])+\delta$. By Lemma \ref{l-tv}, if $i\le j$, then
\[
 d^{(c)}_{\text{\tiny TV}}(i,t)\le\mathbb{P}_0(\widetilde{\tau}_k>t)+1-\pi([j,k]).
\]
As a result of the one-sided Chebyshev inequality, this implies
\[
 T^{(c)}_{\text{\tiny TV}}(i,\epsilon_1)\le \mathbb{E}_0\widetilde{\tau}_k+\sqrt{\left(\tfrac{1}{\delta}-1\right)}\alpha(k).
\]
Similarly, if $i\ge k$, then
\[
 T^{(c)}_{\text{\tiny TV}}(i,\epsilon_1)\le \mathbb{E}_n\widetilde{\tau}_j+\sqrt{\left(\tfrac{1}{\delta}-1\right)}\alpha(j).
\]
Note that, in the case $j<i<k$,
\[
 \mathbb{P}_i(\max\{\widetilde{\tau}_j,\widetilde{\tau}_k\}>t)
 \le\mathbb{P}_i(\widetilde{\tau}_k>t)+\mathbb{P}_i(\widetilde{\tau}_j>t)
 \le\mathbb{P}_j(\widetilde{\tau}_k>t)+\mathbb{P}_k(\widetilde{\tau}_j>t).
\]
This implies
\[
 T^{(c)}_{\text{\tiny TV}}(i,\epsilon_1)\le \mathbb{E}_j\widetilde{\tau}_k+\mathbb{E}_k\widetilde{\tau}_j
 +\sqrt{\left(\tfrac{2}{\delta}-1\right)}\max\{\alpha(j),\alpha(k)\}.
\]
Combining all above gives the desired upper bound.

For the lower bound, set $\epsilon_2=\min\{\pi([j,n]),\pi([0,k])\}-\delta$. By the second inequality of Lemma \ref{l-tv}, one has
\[
 d^{(c)}_{\text{\tiny TV}}(0,t)\ge \pi([j,n])-\mathbb{P}_0(\widetilde{\tau}_j\le t).
\]
Setting $t=\mathbb{E}_0\widetilde{\tau}_j-\sqrt{(1/\delta-1)}\alpha(j)$ in the above inequality derives
\[
 d^{(c)}_{\text{\tiny TV}}(0,t)\ge \pi([j,n])-\delta\ge\epsilon_2.
\]
This implies
\[
 T^{(c)}_{\text{\tiny TV}}(\epsilon_2)\ge T^{(c)}_{\text{\tiny TV}}(0,\epsilon_2)\ge \mathbb{E}_0\widetilde{\tau}_j-\sqrt{(\tfrac{1}{\delta}-1)}\alpha(j).
\]
Similarly, for $k\ge j$, we have
\[
 T^{(c)}_{\text{\tiny TV}}(\epsilon_2)\ge \mathbb{E}_n\widetilde{\tau}_k-\sqrt{(\tfrac{1}{\delta}-1)}\alpha(k)=
 \mathbb{E}_n\widetilde{\tau}_j-\mathbb{E}_k\widetilde{\tau}_j-\sqrt{(\tfrac{1}{\delta}-1)}\alpha(k).
\]
Both inequalities combine to the desired lower bound.
\end{proof}

\begin{proof}[Proof of Theorem \ref{t-tv2}(Continuous time case)] It has been shown in \cite{DLP10} that separation is maximized when the chain started at any of the boundary states and the maximum total variation cutoff is equivalent to the maximum separation cutoff. It is clear that the constants, $s_n$ and $b_n$, in Theorem \ref{t-sep2} are respectively of the same order as the constants, $\theta_n$ and $\alpha_n$, in Theorem \ref{t-tv2}. As a consequence of Theorem \ref{t-sep2}, $\mathcal{F}_c$ has a cutoff in the maximum total variation if and only if $\theta_n\lambda_n\ra\infty$ if and only if $\theta_n/\alpha_n\ra\infty$.

To see a cutoff time and a window, we assume in the following that $\theta_n/\alpha_n\ra\infty$. Set
\[
 \epsilon_0=\inf_n\min\{\pi_n([0,M_n]),\pi_n([M_n,n])\}.
\]
For $\epsilon\in(0,\epsilon_0)$, we may choose $x_n,y_n$ such that
\[
 \pi_n([0,x_n])\ge \frac{\epsilon}{3},\quad\pi_n([x_n,n])\ge 1-\frac{\epsilon}{3},\quad\pi_n([0,y_n])\ge 1-\frac{\epsilon}{3},\quad\pi_n([y_n,n])\ge\frac{\epsilon}{3}.
\]
Clearly, $x_n\le y_n$. Replacing $j,k,\delta$ with $x_n,y_n,\epsilon/3$ in Theorem \ref{t-tvmix} yields
\[
 T^{(c)}_{n,\textnormal{\tiny TV}}(\epsilon)\le\theta_n(x_n)+\mathbb{E}_{x_n}\widetilde{\tau}^{(n)}_{y_n}
 +\mathbb{E}_{y_n}\widetilde{\tau}^{(n)}_{x_n}+\sqrt{\frac{6}{\epsilon}}
 \max\{\alpha_n(x_n),\alpha_n(y_n)\},
\]
where
\[
 \theta_n(j):=\max\{\mathbb{E}_0\widetilde{\tau}^{(n)}_j,\mathbb{E}_n\widetilde{\tau}^{(n)}_j\},\quad \alpha_n^2(j)=\max\{\textnormal{Var}_0\widetilde{\tau}^{(n)}_j,\textnormal{Var}_n\widetilde{\tau}^{(n)}_j\}.
\]
In the above notations, $\theta_n=\theta_n(M_n)$ and $\alpha_n=\alpha_n(M_n)$. Since $x_n\le M_n\le y_n$, one has
\[
 \mathbb{E}_{n}\widetilde{\tau}^{(n)}_{x_n}=\mathbb{E}_{n}\widetilde{\tau}^{(n)}_{M_n}
 +\mathbb{E}_{M_n}\widetilde{\tau}^{(n)}_{x_n},\quad\mathbb{E}_{0}\widetilde{\tau}^{(n)}_{M_n}=
 \mathbb{E}_{0}\widetilde{\tau}^{(n)}_{x_n}+\mathbb{E}_{x_n}\widetilde{\tau}^{(n)}_{M_n}.
\]
Note that, for any positive reals $a,b,c,d$,
\[
 |\max\{a+b,c\}-\max\{a,c+d\}|\le\max\{b,d\}.
\]
This implies
\[
 |\theta_n(x_n)-\theta_n|\le\mathbb{E}\widetilde{\tau}^{(n)}_{M_n}
 +\mathbb{E}_{M_n}\widetilde{\tau}^{(n)}_{x_n}\le\mathbb{E}_{x_n}\widetilde{\tau}^{(n)}_{y_n}
 +\mathbb{E}_{y_n}\widetilde{\tau}^{(n)}_{x_n}.
\]
According to the definition of $x_n,y_n,M_n$, Corollary \ref{c-sep} implies
 \[
 \quad\alpha_n(x_n)\asymp \alpha_n\asymp \alpha_n(y_n).
\]
Let $p_{n,\ell},q_{n,\ell}$ be the birth and death rates of the $n$th chain. The replacement of $j,M,k$ with $x_n,M_n,y_n$ in (\ref{eq-hittime2}) yields that, for any $0\le i\le n$,
\begin{align}
 \alpha_n(i)&\ge\frac{1}{2\sqrt{2}}\max\left\{\sum_{\ell=x_n}^{M_n-1}\frac{\pi_n([0,x_n])}{p_{n,\ell}\pi_n(\ell)},
 \sum_{\ell=M_n+1}^{y_n}\frac{\pi_n([y_n,n])}{q_{n,\ell}\pi_n(\ell)}\right\}\notag\\
 &\ge\frac{\epsilon}{12\sqrt{2}}\sum_{\ell=x_n}^{y_n-1}\frac{1}{p_{n,\ell}\pi_n(\ell)}
 =\frac{\epsilon}{12\sqrt{2}}\sum_{\ell=x_n+1}^{y_n}\frac{1}{q_{n,\ell}\pi_n(\ell)}\notag\\
 &\ge\frac{\epsilon}{12\sqrt{2}}\max\{\mathbb{E}_{x_n}\widetilde{\tau}^{(n)}_{y_n},
 \mathbb{E}_{y_n}\widetilde{\tau}^{(n)}_{x_n}\},\notag
\end{align}
where the second inequality uses the fact $q_{n,\ell}\pi_n(\ell)=p_{n,\ell-1}\pi(\ell-1)$ and the last inequality applies the first identity in Lemma \ref{l-meanvar2}. As a consequence, we may conclude from the above discussions that
\[
 T^{(c)}_{n,\textnormal{\tiny TV}}(\epsilon)-\theta_n\le\left(\frac{48\sqrt{2}}{\epsilon}+\sqrt{\frac{6}{\epsilon}}\right)
 \max\{\alpha_n(x_n),\alpha_n(y_n)\}\asymp \alpha_n,
\]
for all $\epsilon\in(0,\epsilon_0)$. In a similar statement, one can show, by the second part of Theorem \ref{t-tvmix}, that
\[
 \theta_n-T^{(c)}_{n,\textnormal{\tiny TV}}(1-\epsilon)\le\left(\frac{36\sqrt{2}}{\epsilon}+\sqrt{\frac{3}{\epsilon}}\right)
 \max\{\alpha_n(x_n),\alpha_n(y_n)\}
 =O(\alpha_n),
\]
for all $\epsilon\in(0,\epsilon_0)$. This proves the $(\theta_n,\alpha_n)$ cutoff for $\mathcal{F}_c$.
\end{proof}

\begin{proof}[Proof of Theorem \ref{t-tv2}(Discrete time case)]
We will use the result in the continuous time case and \cite{CSal13-1} to deal with the discrete time case. Set
\[
 \delta=\inf_{n,i}K_n(i,i),\quad K_n^{(\delta)}=(K_n-\delta I)/(1-\delta).
\]
In the assumption for discrete time case, we have $\delta\in(0,1)$. Let $\mathcal{X}_n=\{0,1,...,n\}$, $\mathcal{F}^{(\delta)}=(\mathcal{X}_n,K_n^{(\delta)},\pi_n)_{n=1}^\infty$ and $\mathcal{F}^{(\delta)}_c$ be the family of continuous time chains associated with $\mathcal{F}^{(\delta)}$. It was proved in \cite{CSal13-1} (See Theorems 3.1 and 3.3) that, in the maximum total variation,
\begin{equation}\label{eq-cut1}
 \mathcal{F}\text{ has a cutoff}\quad\Lra\quad \mathcal{F}^{(\delta)}_c\text{ has a cutoff}
\end{equation}
and
\begin{equation}\label{eq-cut2}
 \mathcal{F}\text{ has a $(t_n,b_n)$ cutoff}\quad\Lra\quad \mathcal{F}^{(\delta)}_c\text{ has a $((1-\delta)t_n,b_n)$ cutoff}.
\end{equation}
Let $\widetilde{\tau}_i^{(n,\delta)}$ be the hitting time to state $i$ of the continuous time chain associated with $K_n^{(\delta)}$ and $\mathbb{E}_i,\textnormal{Var}_i$ be the conditional expectation and variance given the initial state $i$. Set
\[
 \theta_n^{(\delta)}=\max\left\{\mathbb{E}_0\widetilde{\tau}_{M_n}^{(n,\delta)},
 \mathbb{E}_n\widetilde{\tau}_{M_n}^{(n,\delta)}\right\},\quad
 \beta_n^{(\delta)}=\max\left\{\textnormal{Var}_0\widetilde{\tau}_{M_n}^{(n,\delta)},
 \textnormal{Var}_n\widetilde{\tau}_{M_n}^{(n,\delta)}\right\}.
\]
For $\mathcal{F}^{(\delta)}_c$, it has been proved in the continuous time case that
\[
 \mathcal{F}^{(\delta)}_c\text{ has a cutoff}\quad\Lra\quad \theta_n^{(\delta)}\lambda_n^{(\delta)}\ra\infty \quad\Lra\quad\theta_n^{(\delta)}/\beta_n^{(\delta)}\ra\infty,
\]
where $\lambda_n^{(\delta)}$ is the smallest non-zero eigenvalue of $I-K_n^{(\delta)}$. Furthermore, if it holds true that $\theta_n^{(\delta)}/\beta_n^{(\delta)}\ra\infty$, then $\mathcal{F}_c^{(\delta)}$ has a $(\theta_n^{(\delta)},\beta_n^{(\delta)})$ cutoff. As a result of (\ref{eq-cut1}) and (\ref{eq-cut2}), we have
\[
 \mathcal{F}\text{ has a cutoff}\quad\Lra\quad \theta_n^{(\delta)}/\beta_n^{(\delta)}\ra\infty,
\]
and, further, if the right side holds, then $\mathcal{F}$ has a $(\theta_n^{(\delta)}/(1-\delta),\beta_n^{(\delta)})$ cutoff.

Let $\lambda_n,\theta_n,\beta_n$ be the constants in Theorem \ref{t-tv2}. Clearly, $\lambda_n=(1-\delta)\lambda_n^{(\delta)}$. To finish the proof, it suffices to show that
\begin{equation}\label{eq-cut3}
 \theta_n^{(\delta)}=(1-\delta)\theta_n,\quad \beta_n^{(\delta)}\asymp\beta_n.
\end{equation}
Let $p_{n,i},q_{n,i},r_{n,i}$ be the transition rates of $K_n$ and $p_{n,i}^{(\delta)},q_{n,i}^{(\delta)},r_{n,i}^{(\delta)}$ be the transition rates of $K_n^{(\delta)}$. It is clear that
\[
 p_{n,i}^{(\delta)}=p_{n,i}/(1-\delta),\quad q_{n,i}^{(\delta)}=q_{n,i}/(1-\delta),\quad r_{n,i}^{(\delta)}=(r_{n,i}-\delta)/(1-\delta).
\]
The first equality of (\ref{eq-cut3}) is an immediate result of the first identity of Lemma \ref{l-meanvar2}. To see the second part of (\ref{eq-cut3}), let $\lambda_{n,1},...,\lambda_{n,n}$ be eigenvalues of the submatrix of $I-K_n$ obtained by removing the $M_n$-th row and column. Clearly, $\lambda_{n,1}/(1-\delta),...,\lambda_{n,n}/(1-\delta)$ are eigenvalues of the submatrix of $I-K_n^{(\delta)}$ obtained by removing the $M_n$-th row and column. As a consequence of Lemma \ref{l-meanvar}, we have
\[
 \beta_n^2\asymp\sum_{i=1}^n\frac{1-\lambda_{n,i}}{\lambda_{n,i}^2},\quad \left(\beta_n^{(\delta)}\right)^2\asymp\sum_{i=1}^n\frac{1}{\lambda_{n,i}^2}.
\]
Note that the application of Remark \ref{r-var} on the chain $(\mathcal{X}_n,K_n^{(\delta)},\pi_n)$ says
\[
 (1-\delta)\sum_{i=1}^n\frac{1}{\lambda_{n,i}^2}\ge\sum_{i=1}^n\frac{1}{\lambda_{n,i}}.
\]
This implies $\beta_n\asymp\beta_n^{(\delta)}$.
\end{proof}


\section{A randomization of birth and death chains}\label{s-rbdc}

This section gives two nontrivial examples as applications of theorems in the introduction. The first example is stated in Theorem \ref{t-randomm} and we discuss its proof in the following.

\begin{proof}[Proof of Theorem \ref{t-randomm}]
The proofs for $\mathcal{F}_c$ and $\mathcal{F}$ are similar and we consider only the continuous time case. Let $M_n,\theta_n,\alpha_n$ be as in Theorem \ref{t-tv2}. For convenience, we let $(p_{n,i},q_{n,i},r_{n,i})$ be the transition rates of $K_n$. For $n\ge 1$, set
\[
 \theta_{n,1}=\sum_{i=0}^{M_n-1}\frac{\pi_n([0,i])}{\pi_n(i)p_{n,i}},\quad \theta_{n,2}=\sum_{i=M_n+1}^{n}\frac{\pi_n([i,n])}{\pi_n(i)q_{n,i}}
\]
and
\[
 \alpha_{n,1}^2=\sum_{i=0}^{M_n-1}\sum_{j=i}^{M_n-1}\frac{\pi_n([0,i])^2}{\pi_n(i)p_{n,i}\pi_n(j)p_{n,j}},\,
 \alpha_{n,1}^2=\sum_{i=M_n+1}^{n}\sum_{j=M_n+1}^i\frac{\pi_n([i,n])^2}{\pi_n(i)q_{n,i}\pi_n(j)q_{n,j}}.
\]
It is clear from Lemma \ref{l-meanvar2} that
\[
 \theta_n=\max\{\theta_{n,1},\theta_{n,2}\},\quad \alpha_n=\max\{\alpha_{n,1},\alpha_{n,2}\}.
\]
Without loss of generality, we may assume that $\theta_n=\theta_{n,1}$. For $n\ge 1$, let $U_{n,1},V_{n,1}$ be positive random variables defined by
\[
 U_{n,1}=\sum_{i=0}^{M_n-1}\frac{\pi_n([0,i])}{\pi_n(i)p_{n,i}C_{n,i+1}},\,
 V_{n,1}^2=\sum_{i=0}^{M_n-1}\sum_{j=i}^{M_n-1}\frac{\pi_n([0,i])^2}{\pi_n(i)p_{n,i}C_{n,i+1}\pi_n(j)p_{n,j}C_{n,j+1}}.
\]
By the independency of $C_{n,i}$, one may compute
\[
 \mathbb{E}U_{n,1}=\mu_n\theta_{n,1}=\mu_n\theta_n,\quad \text{Var}(U_{n,1})=\nu_n^2\alpha_{n,1}^2\le\nu_n^2\alpha_n^2
\]
and
\begin{align}
 \mathbb{E}V_{n,1}^2&=\sum_{0\le i<j\le M_n-1}\frac{\pi_n([0,i])^2}{\pi_n(i)p_{n,i}\pi_n(j)p_{n,j}}\mu_n^2
 +\sum_{i=0}^{M_n-1}\frac{\pi_n([0,i])^2}{\pi_n(i)^2p_{n,i}^2}(\mu_n^2+\nu_n^2)\notag\\
 &\le (\mu_n^2+\nu_n^2)\alpha_{n,1}^2\le[(\mu_n+\nu_n)\alpha_{n,1}]^2.\notag
\end{align}
The estimation for $\mathbb{E}V_n^2$ implies
\[
 \mathbb{E}V_{n,1}\le\sqrt{\mathbb{E}V_{n,1}^2}\le(\mu_n+\nu_n)\alpha_{n,1}\le(\mu_n+\nu_n)\alpha_n.
\]
Set $a_n=\sqrt{(\mu_n\theta_n)/(\nu_n\alpha_n)}$, $b_n=\sqrt{(\mu_n\theta_n)/[(\mu_n+\nu_n)\alpha_n]}$ and
\[
 E_{n,1}=\{\omega_n\in\Omega_n:|U_{n,1}(\omega_n)-\mu_n\theta_n|<a_n\nu_n\alpha_n,\,
 V_{n,1}(\omega_n)<b_n(\mu_n+\nu_n)\alpha_n\}.
\]
Since $\mathcal{F}_c$ has a maximum total variation cutoff, Theorem \ref{t-tv2} implies $\alpha_n=o(\theta_n)$. In the assumption of $(\nu_n\alpha_n)=o(\mu_n\theta_n)$, it is easy to see that, for $\omega_n\in E_{n,1}$,
\[
 U_{n,1}(\omega_n)\sim\mu_n\theta_n,\quad V_{n,1}(\omega_n)=o(\mu_n\theta_n).
\]
By the Chebyshev and Markov inequalities, the fact that $a_n,b_n\ra\infty$ yields $\mathbb{P}^{(n)}(E_{n,1})\ra 1$.

In the same way, we set
\[
 U_{n,2}=\sum_{i=M_n+1}^{n}\frac{\pi_n([i,n])}{\pi_n(i)q_{n,i}C_{n,i}},\,
 V_{n,2}^2=\sum_{i=M_n+1}^{n}\sum_{j=M_n+1}^{i}\frac{\pi_n([i,n])^2}{\pi_n(i)q_{n,i}C_{n,i}\pi_n(j)q_{n,j}C_{n,j}},
\]
and
\[
 E_{n,2}=\{\omega_n\in\Omega_n:U_{n,2}(\omega_n)<\mu_n\theta_n+a_n\nu_n\alpha_n,\,
 V_{n,2}(\omega_n)<b_n(\mu_n+\nu_n)\alpha_n\}.
\]
A similar reasoning as before yields that $\mathbb{P}^{(n)}(E_{n,2})\ra 1$ and, for $\omega_n\in E_{n,2}$,
\[
 U_{n,2}(\omega_n)\le \mu_n\theta_n(1+o(1)),\quad V_{n,2}(\omega_n)=o(\mu_n\theta_n).
\]
As consequence, if we set $E_n=E_{n,1}\cap E_{n,2}$, then $\mathbb{P}^{(n)}(E_n)\ra\infty$ and, for $\omega_n\in E_n$,
\[
 \max\{U_{n,1},U_{n,2}\}\sim \mu_n\theta_n,\quad \max\{V_{n,1},V_{n,2}\}=o(\mu_n\theta_n).
\]
The maximum total variation cutoff for $\mathcal{F}_c^{(\omega)}$ and the cutoff time $\mu_n\theta_n$ are immediate from Theorem \ref{t-tv2}.
\end{proof}

\begin{rem}
From the proof given above, one can derive a variation of Theorem \ref{t-randomm}. Namely, under the assumption of $\nu_n\alpha_n=o(\mu_n\theta_n)$, if $\mathcal{F}_c$ has no maximum total variation cutoff (resp. maximum separation cutoff), then there is a sequence $E_n\subset\Omega_n$ satisfying $\mathbb{P}^{(n)}(E_n)\ra 1$ such that $\mathcal{F}_c^{(\omega)}$ has no maximum total variation cutoff (resp. maximum separation cutoff) for $\omega\in\prod_{n=1}^\infty E_n$. Note that, the requirement $\nu_n\alpha_n=o(\mu_n\theta_n)$ and the assumption of no cutoff will imply the existence of a subsequence, say $i_n$, such that $\nu_{i_n}=o(\mu_{i_n})$. As a result of the Chebyshev inequality, $1/C_{i_n,1}-\mathbb{E}(1/C_{i_n,1})$ converges in probability to $0$. This turns $\mathcal{F}_c^{(\omega)}$ into a lazy version of $\mathcal{F}_c$ with high probability.
\end{rem}

Note that the hypothesis of $\nu_n\alpha_n=o(\mu_n\theta_n)$ requires the existence of a second moment of $1/C_{n,1}$. Next, we give an example where $1/C_{n,1}$ does not have a finite first moment.

\begin{thm}\label{t-randomsrw}
For $n\ge 1$, let $C_{n,1},...,C_{n,n}$ be i.i.d. uniform random variables over $(0,1)$ defined on $(\Omega_n,\mathbb{P}^{(n)})$. For $\omega=(\omega_1,\omega_2,...)\in\prod_n\Omega_n$, let $\mathcal{F}^{(\omega)}=(\mathcal{X}_n,K_n^{(\omega_n)},\pi_n)_{n=1}^\infty$ be a family of birth and death chains with $\mathcal{X}_n=\{0,1,...,n\}$ and
\[
 \begin{cases}K_n^{(\omega_n)}(i,i+1)=K(i+1,i)=C_{n,i+1}/2,\quad\forall 0\le i<n,\\ K_n^{(\omega_n)}(i,i)=1-K_n^{(\omega_n)}(i,i+1)-K_n^{(\omega_n)}(i,i-1),\quad\forall i.\end{cases}
\]
Let $\mathcal{F}_c^{(\omega)}$ be the family of continuous time chains associated with $\mathcal{F}^{(\omega)}$ and, for $\omega_n\in\Omega_n$, let $T_{n,\textnormal{\tiny TV}}^{c}(\omega_n,\cdot)$ be the maximum total variation mixing time for $(\mathcal{X}_n,K_n^{(\omega_n)},\pi_n)$. Then, there is a sequence $E_n\subset\Omega_n$ satisfying $\mathbb{P}^{(n)}(E_n)\ra 1$ such that, for any $\omega=(\omega_1,\omega_2,...)\in\prod_{n=1}^\infty E_n$, the family $\mathcal{F}_c^{(\omega)}$ has no maximum total variation cutoff and $T_{n,\textnormal{\tiny TV}}^c(\omega_n,\epsilon)\asymp n^2\log n$ for $\epsilon\in(0,1/10)$.
\end{thm}

\begin{proof}
Let $M_n\in\mathcal{X}_n$ and $U_{n,1},U_{n,2}$ be as in the proof of Theorem \ref{t-randomm}. For $n\ge 1$, set
\[
 \overline{\Omega}_n=\left\{C_{n,i}>\frac{1}{n\log n},\forall 1\le i\le n\right\},\quad \overline{\mathbb{P}}^{(n)}(\cdot)=\mathbb{P}^{(n)}(\cdot|\overline{\Omega}_n),
\]
where $\overline{\mathbb{P}}^{(n)}$ is the conditional probability of $\mathbb{P}^{(n)}$ given $\overline{\Omega}_n$. Clearly, $\mathbb{P}^{(n)}(\overline{\Omega}_n)=(1-1/n\log n)^n\ra 1$ and, in $\overline{\mathbb{P}}^{(n)}$, $C_{n,1},...,C_{n,n}$ are i.i.d. uniformly distributed over $(1/n\log n,1)$. Let $\overline{\mathbb{E}}$ and $\overline{\text{Var}}$ be the expectation and variance taken in $\overline{\mathbb{P}}^{(n)}$. It is an easy exercise to compute
\[
 \overline{\mathbb{E}}(1/C_{n,1})=\frac{\log n+\log\log n}{1-1/n\log n}\sim \log n
\]
and
\[
 \overline{\text{Var}}(1/C_{n,1})=n\log n-(\overline{\mathbb{E}}(1/C_{n,1}))^2\sim n\log n,
\]
This implies that, if $M_n\ra\infty$ and $n-M_n\ra\infty$, then
\[
 \overline{\mathbb{E}}U_{n,1}\sim M_n^2\log n,\quad\overline{\mathbb{E}}U_{n,2}\sim(n-M_n)^2\log n,
\]
and
\[
 \overline{\text{Var}}(U_{n,1})\sim M_n^2n\log n,\quad \overline{\text{Var}}(U_{n,2})\sim(n-M_n)^2n\log n.
\]
For $a\in (0,1)$, if $M_n=\lfloor an\rfloor$, we write $U_{n,i}^{(a)}$ for $U_{n,i}$. As a result of the above computation, we obtain
\[
 \overline{\mathbb{E}}U_{n,1}^{(a)}\sim a^2n^2\log n,\quad \overline{\mathbb{E}}U_{n,2}^{(a)}\sim(1-a)^2n^2\log n,
\]
and
\[
 \overline{\text{Var}}(U_{n,1}^{(a)})\sim a^2n^3\log n,\quad \overline{\text{Var}}(U_{n,2}^{(a)})\sim(1-a)^2n^3\log n.
\]
For $n\ge 1$, let
\[
 E_n=\left\{\omega_n\in A_n:|U_{n,1}^{(a)}-a^2n^2\log n|<n^{3/2}\log n,\text{ for }a=1/4,1/2\right\}.
\]
It is easy to show that $\overline{\mathbb{P}}^{(n)}(E_n)\ra 1$ and, hence, $\mathbb{P}^{(n)}(E_n)\ge \mathbb{P}^{(n)}(A_n)\overline{\mathbb{P}}^{(n)}(E_n)\ra 1$. Furthermore, for $\omega_n\in E_n$,
\[
 \max\{U_{n,1}^{(1/2)}(\omega_n),U_{n,2}^{(1/2)}(\omega_n)\}\sim\frac{n^2\log n}{4}
\]
and
\[
 \max\{U_{n,1}^{(1/4)}(\omega_n),U_{n,2}^{(1/4)}(\omega_n)\}\sim\frac{9 n^2\log n}{16}.
\]
By Remark \ref{r-exp1}, $\mathcal{F}_c^{(\omega)}$ has no maximum total variation cutoff for $\omega\in\prod_nE_n$. The order of the mixing time is given by Theorems 3.1 and 3.9 of \cite{CSal13-2}.
\end{proof}

\begin{rem}
We refer the reader to \cite{DW13} for another randomization of birth and death chains, which is different from the one considered in Theorem \ref{t-randomsrw}.
\end{rem}


\section{Chains started at boundary states}\label{s-bs}

For continuous time birth and death chains, \cite{DLP10} shows that separation reaches its maximum when the initial state is any of the boundary states. This is not true in the case of total variation and it is easy to construct counterexamples. In this section, we discuss the total variation cutoff for families of birth and death chains started at a boundary state. As before, we use $\mathcal{F}$ and $\mathcal{F}_c$ for families of birth and death chains without starting states specified and write $\mathcal{F}^L,\mathcal{F}_c^L$ and $\mathcal{F}^R,\mathcal{F}_c^R$ respectively for families of chains started at the left and right boundary states.

The following theorem displays a list of equivalent conditions for the total variation cutoff. It is worthwhile to note that some of these conditions are very similar to the conditions in Theorem \ref{t-tv2}.

\begin{thm}\label{t-tv3}
Let $\mathcal{F}=(\mathcal{X}_n,K_n,\pi_n)_{n=1}^\infty$ be a family of irreducible birth and death chains with $\mathcal{X}_n=\{0,1,...,n\}$ and $\mathcal{F}_c$ be the family of associated continuous time chains in $\mathcal{F}$. For $n\ge 1$, let $\widetilde{\tau}^{(n)}_i$ be the first hitting time to state $i$ of the $n$th chain in $\mathcal{F}_c$ and, for $a\in(0,1)$, let $M_n(a)$ be a state in $\mathcal{X}_n$ satisfying
\[
 \pi_n([0,M_n(a)])\ge a,\quad \pi_n([M_n(a),n])\ge 1-a,
\]
and let $\lambda_n(a)$ be the smallest eigenvalue of the submatrix of $I-K_n$ indexed by states $0,...,M_n(a)-1$. Set
\[
 u_n(a)=\mathbb{E}_0\widetilde{\tau}^{(n)}_{M_n(a)},\quad v_n^2(a)=\textnormal{Var}_0\widetilde{\tau}^{(n)}_{M_n(a)}.
\]
Assume that $\pi_n(0)\ra 0$. Then, the following are equivalent.
\begin{itemize}
\item[(1)] $\mathcal{F}_c^L$ has a total variation cutoff.

\item[(2)] $u_n(a)/v_n(a)\ra\infty$ for all $a\in(0,1)$.

\item[(3)] $u_n(a)\lambda_n(a)\ra\infty$ for all $a\in(0,1)$.

\item[(4)] There are $a\in(0,1)$ and a positive sequence $(t_n)_{n=1}^\infty$ satisfying
\[
 t_n=O(u_n(c)),\quad  \forall c\in(0,1)
\]
and
\[
 \lim_{n\ra\infty}\mathbb{P}_0\left(\widetilde{\tau}^{(n)}_{M_n(a)}>(1-\epsilon)t_n\right)=1,\quad\forall \epsilon\in(0,1),
\]
and, for any $b\in(a,1)$, there is $\alpha_b\in(0,1)$ such that
\[
 \limsup_{n\ra\infty}\mathbb{P}_0\left(\widetilde{\tau}^{(n)}_{M_n(b)}>(1+\epsilon)t_n\right)\le\alpha_b,\quad\forall
 \epsilon>0,
\]
where $\mathbb{P}_i$ denotes the probability given the initial state $i$.
\end{itemize}

Furthermore, if \textnormal{(2)} or \textnormal{(3)} holds, then $\mathcal{F}_c^L$ has a cutoff with cutoff time $(u_n(a))_{n=1}^\infty$ for any $a\in(0,1)$. If \textnormal{(4)} holds, then $\mathcal{F}_c^L$ has a cutoff with cutoff time $(t_n)_{n=1}^\infty$.
\end{thm}

The discrete time version of the previous theorem can be stated as follows.

\begin{thm}\label{t-tv4}
Let $\mathcal{F},M_n(a),\lambda_n(a)$ be as in Theorem \ref{t-tv3}. For $n\ge 1$, let $\tau^{(n)}_i$ be the first hitting time to state $i$ of the $n$th chain in $\mathcal{F}$ and, for $a\in(0,1)$, set
\[
 u_n(a)=\mathbb{E}_0\tau^{(n)}_{M_n(a)},\quad w_n^2(a)=\textnormal{Var}_0\tau^{(n)}_{M_n(a)}.
\]
Assume that $\pi_n(0)\ra 0$, $\inf_{i,n}K_n(i,i)>0$ and $u_n(a)\ra\infty$ for some $a\in(0,1)$. Then, the conclusion in Theorem \ref{t-tv3} remains true for the family $\mathcal{F}^L$.
\end{thm}

\begin{rem}
One can see from the proof of Theorem \ref{t-tv3} that the condition $\pi_n(0)\ra 0$ is necessary for the existence of cutoff of $\mathcal{F}_c^L$. This is also true for the family $\mathcal{F}^L$ in Theorem \ref{t-tv4}.
\end{rem}

\begin{rem}\label{r-pin->0}
Let $\mathcal{F},\mathcal{F}_c$ be as in Theorem \ref{t-tv3} and $(p_{n,i},q_{n,i},r_{n,i})$ be the transition rates of the $n$th chains in $\mathcal{F}$. Let $M_n\in\mathcal{X}_n$ be a sequence of states satisfying (\ref{eq-pin2}), that is,
\[
 \inf_{n\ge 1}\pi_n([0,M_n])>0,\quad\inf_{n\ge 1}\pi_n([M_n,n])>0,
\]
and $x_n\in\{0,n\}$ be a boundary state fulfilling the following equation
\[
 \max\{\mathbb{E}_0\widetilde{\tau}^{(n)}_{M_n},\mathbb{E}_n\widetilde{\tau}^{(n)}_{M_n}\}
 =\mathbb{E}_{x_n}\widetilde{\tau}^{(n)}_{M_n}.
\]
By Lemma \ref{l-meanvar2} and Theorem A.1 of \cite{CSal13-2}, if $x_n=0$, then
\[
 \mathbb{E}_{x_n}\widetilde{\tau}^{(n)}_{M_n}=\sum_{i=0}^{M_n-1}\frac{\pi_n([0,i])}{\pi_n(i)p_{n,i}}
 \le\sum_{i=0}^{M_n-1}\frac{1}{\pi_n(i)p_{n,i}}
\]
and
\begin{align}
 \frac{1}{\lambda_n}&\ge\min\{\pi_n([0,M_n]),\pi_n([M_n,n])\}\times
 \max_{j:j<M_n}\sum_{i=j}^{M_n-1}
 \frac{\pi_n([0,j])}{\pi_n(i)p_{n,i}}\notag\\
 &\ge\min\{\pi_n([0,M_n]),\pi_n([M_n,n])\}\times\pi_n(0)\sum_{i=0}^{M_n-1}
 \frac{1}{\pi_n(i)p_{n,i}}\notag
\end{align}
This implies
\[
 \mathbb{E}_{x_n}\widetilde{\tau}^{(n)}_{M_n}\lambda_n\le\frac{1}{\min\{\pi_n([0,M_n]),\pi_n([M_n,n])\}\pi_n(0)}.
\]
In a similar way, this inequality also holds in the case $x_n=n$. As a consequence of Theorem \ref{t-tv2}, if $\mathcal{F}_c$ has a maximum total variation cutoff, then $\pi_n(x_n)\ra 0$. The above discussion also holds for $\mathcal{F}$ with the assumption $\inf_{n,i}K_n(i,i)>0$.
\end{rem}

\begin{rem}\label{r-cutoffitme}
Let $\mathcal{F}_c^L$ and $\mathcal{F}^L$ be the families in Theorems \ref{t-tv3} and \ref{t-tv4}. If $\mathcal{F}_c^L$ (resp. $\mathcal{F}^L$) has a total variation cutoff with cutoff time $t_n$, then
\[
 t_n\sim\mathbb{E}_0\widetilde{\tau}^{(n)}_{M_n},\quad(\text{resp. } t_n\sim\mathbb{E}_0\tau^{(n)}_{M_n},)
\]
where $M_n\in\mathcal{X}_n$ is any sequence satisfying
\begin{equation}\label{eq-mn2}
 \inf_{n\ge 1}\pi_n([0,M_n])>0,\quad \inf_{n\ge 1}\pi_n([M_n,n])>0.
\end{equation}
In particular, if $\mathcal{F}_c^L$ (resp. $\mathcal{F}^L$) has a total variation cutoff with bounded cutoff time, then
\[
 \mathbb{E}_0\widetilde{\tau}^{(n)}_{M_n}=O(1),\quad (\text{resp. }\mathbb{E}_0\tau^{(n)}_{M_n}=O(1),)
\]
for any sequence $M_n\in\mathcal{X}_n$ satisfying (\ref{eq-mn2}).
\end{rem}

\begin{rem}\label{r-cutoff}
Let $\mathcal{F}_c^L$ be the family in Theorems \ref{t-tv3}. If $\mathcal{F}_c^L$ has a total variation cutoff, then $u_n(a)\sim u_n(b)$ for all $a,b\in(0,1)$, or equivalently
\[
 \mathbb{E}_{M_n(a)}\widetilde{\tau}^{(n)}_{M_n(b)}=o\left(\mathbb{E}_0\widetilde{\tau}^{(n)}_{M_n(c)}\right),\quad\forall
 a,b,c\in(0,1).
\]
This is also true for $\mathcal{F}^L$ with the assumption in Theorem \ref{t-tv4}. But, the converse is not necessarily true. For an illustration, recall the example in Remark \ref{r-exp1}. It has been proved that
\[
 \mathbb{E}_0\widetilde{\tau}^{(n)}_{M_n(a)}\asymp \frac{1}{\lambda_n}\asymp \frac{1}{\xi_n}+n^2,\quad
 \forall a\in(0,1).
\]
By Lemma \ref{l-meanvar2}, one may compute
\[
 \text{Var}_0\widetilde{\tau}^{(n)}_1=\frac{1}{\xi_n^2}
\]
and
\[
 \text{Var}_1\widetilde{\tau}^{(n)}_{M_n(a)}\ge\sum_{i=1}^{M_n(a)-1}\frac{1}{K_n(i,i+1)\pi_n(i)}
 \sum_{\ell=1}^{i}\pi_n(\ell)\mathbb{E}_\ell\widetilde{\tau}^{(n)}_{i+1}\asymp n^4.
\]
Along with the fact $\text{Var}_0\widetilde{\tau}^{(n)}_i\le(\mathbb{E}_0\widetilde{\tau}^{(n)}_i)^2$, we may conclude from the above computations that $\text{Var}_0\widetilde{\tau}^{(n)}_{M_n(a)}\asymp \xi_n^{-2}+n^4$ for all $a\in(0,1)$. By Theorem \ref{t-tv3}, this implies that the family $\mathcal{F}_c^L$ has no total variation cutoff. It has been shown in Remark \ref{r-exp1} that if $n^2\xi_n\ra 0$, then $\mathbb{E}_0\widetilde{\tau}^{(n)}_{M_n(a)}\sim\xi_n^{-1}$ for all $a\in(0,1)$.
\end{rem}

\begin{rem}\label{r-wcomp}
Let $v_n(a)$ and $w_n(a)$ be the constants in Theorems \ref{t-tv3} and \ref{t-tv4}. It is remarkable that if $\delta=\inf_{i,n}K_n(i,i)>0$, then $\delta v_n^2(a)\le w_n^2(a)\le v_n^2(a)$ for all $a\in(0,1)$. To see this, we let $\beta^{(n)}_1,...,\beta^{(n)}_{M_n}$ be the eigenvalues of the submatrix of $I-K_n$ indexed by $0,...,M_n(a)-1$. By Lemma \ref{l-meanvar}, $\beta^{(n)}_i>0$ for all $i$ and
\[
 v_n^2(a)=\sum_{i=1}^{M_n(a)}\frac{1}{(\beta^{(n)}_i)^2},\quad w_n^2(a)=\sum_{i=1}^{M_n(a)}\frac{1-\beta^{(n)}_i}{(\beta^{(n)}_i)^2}.
\]
Clearly, $w_n^2(a)\le v_n^2(a)$. For the lower bound of $w_n^2(a)$, set $K_n^{(\delta)}=(K_n-\delta I)/(1-\delta)$. Note that $K_n^{(\delta)}$ is also a stochastic matrix and the submatrix of $I-K_n^{(\delta)}$ indexed by $0,...,M_n(a)-1$ has eigenvalues $\beta^{(n)}_1/(1-\delta),...,\beta^{(n)}_{M_n(a)}/(1-\delta)$. By Remark \ref{r-var}, we have
\[
 (1-\delta)\sum_{i=1}^{M_n(a)}\frac{1}{(\beta^{(n)}_i)^2}\ge\sum_{i=1}^{M_n(a)}\frac{1}{\beta^{(n)}_i}
\]
and this implies $w_n^2(a)\ge \delta v_n^2(a)$.
\end{rem}

\begin{rem}
Note that, in Theorems \ref{t-tv3} and \ref{t-tv4}, if one chooses $\mathbb{E}_0\widetilde{\tau}^{(n)}_{M_n(a)}$ and $\mathbb{E}_0\tau^{(n)}_{M_n(a)}$ as the cutoff times, the square roots of $\text{Var}_0\widetilde{\tau}^{(n)}_{M_n(a)}$ and $\text{Var}_0\tau^{(n)}_{M_n(a)}$ are no longer suitable for the respective cutoff windows. This is very different from the conclusion in Theorem \ref{t-tv2} and we refer the reader to Example \ref{ex-cutoff3} for an illustration of this observation.
\end{rem}

The next corollary provides a way of selecting cutoff windows.

\begin{cor}\label{c-tv}
Let $\mathcal{F}_c,u_n(a),v_n(a)$ be as in Theorem \ref{t-tv3}. If $\mathcal{F}_c^L$ has a total variation cutoff and $b_n>0$ is a sequence satisfying
\[
 b_n=o(u_n(a)),\quad v_n(a)=O(b_n),\quad\forall a\in(0,1),
\]
then $\mathcal{F}_c^L$ has a $(u_n(a),b_n)$ total variation cutoff. The above statement is also true for $\mathcal{F}^L$ under the assumption of $\inf_{n,i}K_n(i,i)>0$ and $\inf_nb_n>0$ and the replacement of $v_n(a)$ by $w_n(a)$ in Theorem \ref{t-tv4}.
\end{cor}

\begin{ex}\label{ex-cutoff3}
Let $\mathcal{F}=(\mathcal{X}_n,K_n,\pi_n)_{n=1}^\infty$ be a family of birth and death chains for which $\mathcal{X}_n=\{0,1,...,n\}$, $\pi_n(i)=2^{-n}\binom{n}{i}$ and
\[
 \begin{cases}
 K_n(i,i+1)=1-\frac{i}{n},\quad K_n(i+1,i)=\frac{i+1}{n}\quad\text{for }i\ne M_n,\\
 K_n(M_n,M_n+1)=c_n\left(1-\frac{M_n}{n}\right),\quad  K_n(M_n,M_n)=(1-c_n)\left(1-\frac{M_n}{n}\right),\\
 K_n(M_n+1,M_n)=\frac{c_n(M_n+1)}{n},\quad
 K_n(M_n+1,M_n+1)=\frac{(1-c_n)(M_n+1)}{n},
 \end{cases}
\]
where $c_n\in(0,1)$ and $M_n\in\mathcal{X}_n$ is a state satisfying $\pi_n([0,M_n])\ge 1/4$ and $\pi_n([M_n,n])\ge 3/4$. Let $\mathcal{F}_c$ be the family associated with $\mathcal{F}$ and $\widetilde{\tau}^{(n)}_i$ be the first hitting time to state $i$ of the $n$th chain in $\mathcal{F}_c$. We will also use $M_n(a)$ with $a\in(0,1)$ to denote a state satisfying $\pi_n([0,M_n(a)])\ge a$ and $\pi_n([M_n(a),n])\ge 1-a$. When $c_n=1$, $(\mathcal{X}_n,K_n,\pi_n)$ is the Ehrenfest chain on $\{0,1,...,n\}$. The spectral information of the Ehrenfest chain is well-studied and it is easy to derive by Lemma \ref{l-gap} that
\[
 \mathbb{E}_0\widetilde{\tau}^{(n)}_{\lfloor n/2\rfloor}=\frac{1}{4}n\log n+O(n),\quad
 \text{Var}_0\widetilde{\tau}^{(n)}_{\lfloor n/2\rfloor}\asymp n^2.
\]
One may use Stirling's formula to show that, for $0<a<b<1$,
\[
 \left|\frac{n}{2}-M_n(a)\right|\asymp\sqrt{n},\quad \pi_n(i)\asymp\frac{1}{\sqrt{n}}\quad\text{uniformly for }M_n(a)\le i\le M_n(b).
\]
By Lemmas \ref{l-meanvar2}, \ref{l-gap} and \ref{l-hit2}, this implies that, for $a\in(0,1)$,
\begin{equation}\label{eq-ehren}
 \mathbb{E}_0\widetilde{\tau}^{(n)}_{M_n(a)}=\frac{1}{4}n\log n+O(n),\quad
 \text{Var}_0\widetilde{\tau}^{(n)}_{M_n(a)}\asymp n^2.
\end{equation}
When $c_n$ is small, $(\mathcal{X}_n,K_n,\pi_n)$ is the modification of the Ehrenfest chain with bottleneck between states $M_n$ and $M_n+1$. In the following, we will discuss the total variation cutoff and the cutoff window of $\mathcal{F}_c^L$ when $c_n$ is small.

First, we consider the total variation cutoff of $\mathcal{F}_c^L$. By Lemma \ref{l-meanvar2} and (\ref{eq-ehren}), one can show without difficulty that, for $a\in(0,1/2)$,
\begin{equation}\label{eq-ehren2}
 \mathbb{E}_0\widetilde{\tau}^{(n)}_{M_n(a)}=\frac{1}{4}n\log n+O(n),\quad
 \text{Var}_0\widetilde{\tau}^{(n)}_{M_n(a)}\asymp n^2,
\end{equation}
and, for $a\in(1/2,1)$,
\begin{equation}\label{eq-ehren3}
 \mathbb{E}_0\widetilde{\tau}^{(n)}_{M_n(a)}=\frac{1}{4}n\log n+O(n)+\frac{1+o(1)}{2c_n\pi_n(M_n)},\quad
 \text{Var}_0\widetilde{\tau}^{(n)}_{M_n(a)}\asymp n^2+\frac{n}{c_n^2},
\end{equation}
where $\pi_n(M_n)\asymp 1/\sqrt{n}$. By Theorem \ref{t-tv3}, $\mathcal{F}_c^L$ has a total variation cutoff if and only if $c_n\sqrt{n}\log n\ra\infty$.

Next, we discuss the cutoff window of $\mathcal{F}_c^L$. Assume that $c_n\sqrt{n}\log n\ra\infty$. By Corollary \ref{c-tv} and Equations (\ref{eq-ehren2}) and (\ref{eq-ehren3}), $\mathcal{F}_c^L$ has a $(\frac{1}{4}n\log n,\max\{\sqrt{n}/c_n,n\})$ total variation cutoff. We will prove that the window is optimal when $c_n\sqrt{n}\ra 0$. Suppose $c_n\sqrt{n}\ra 0$ and set
\[
 s_n=\mathbb{E}_0\widetilde{\tau}^{(n)}_{M_n},\quad t_n=\mathbb{E}_0\widetilde{\tau}^{(n)}_{M_n+1},\quad a_n^2=\text{Var}^{(n)}_0\widetilde{\tau}^{(n)}_{M_n},\quad b_n^2=\text{Var}^{(n)}_0\widetilde{\tau}^{(n)}_{M_n+1}.
\]
Let $T_{n,\text{\tiny TV}}^c(0,\epsilon)$ be the total variation mixing time of the $n$th chain in $\mathcal{F}_c^L$ and recall (\ref{eq-mixing}) in the following
\[
 T_{n,\text{\tiny TV}}^{(c)}(0,\epsilon)\begin{cases}\le\mathbb{E}_0\widetilde{\tau}^{(n)}_i
 +\sqrt{(\tfrac{1-\delta}{\delta})\text{Var}_0(\widetilde{\tau}^{(n)}_i)}&\text{for } \epsilon=\delta+\pi_n([i+1,n])\\
 \ge\mathbb{E}_0\widetilde{\tau}^{(n)}_i-\sqrt{(\tfrac{\delta}{1-\delta})\text{Var}_0(\widetilde{\tau}^{(n)}_i)}&
 \text{for }\epsilon=\delta-\pi_n([0,i-1])\end{cases}.
\]
In the first inequality, the replacement of $i=M_n$ and $\delta=1/8$ implies
\[
 T_{n,\text{\tiny TV}}^c(0,7/8)\le s_n+3a_n.
\]
In the second inequality, the replacement of $i=M_n+1$ and $\delta=3/8$ gives
\[
 T_{n,\text{\tiny TV}}^c(0,1/8)\ge t_n-\frac{4}{5}b_n.
\]
These two inequalities yield
\[
 T_{n,\text{\tiny TV}}^c(0,1/8)-T_{n,\text{\tiny TV}}^c(0,7/8)\ge \mathbb{E}_{M_n}\widetilde{\tau}^{(n)}_{M_n+1}-3a_n-\frac{4}{5}b_n.
\]
Under the assumption that $c_n\sqrt{n}\ra 0$, one may compute using Lemma \ref{l-meanvar2} that
\[
 a_n\asymp n,\quad b_n\sim \mathbb{E}_{M_n}\widetilde{\tau}^{(n)}_{M_n+1}\asymp\frac{\sqrt{n}}{c_n}=\frac{n}{c_n\sqrt{n}}.
\]
Consequently, when $c_n\sqrt{n}\ra 0$, the cutoff window can be $\text{Var}_0\widetilde{\tau}^{(n)}_{M_n(a)}$ for any $a\in (1/4,1)$ but not for $a\in(0,1/4)$. Similar observation also happens in $\mathcal{F}_c^R$.

We would like to point out an interesting observation arising from the bottleneck effect in this example. Compared with the case $c_n=1$ for all $n$, when $c_n$ is of order bigger than $1/\sqrt{n}$, $\mathcal{F}_c^L$ has a cutoff with the same cutoff time and window. When $c_n$ is of order between $1/\sqrt{n}$ and $1/\sqrt{n}\log n$, $\mathcal{F}_c^L$ has a cutoff with the same cutoff time but different (larger) cutoff window. When $c_n$ is of order smaller than $1/\sqrt{n}\log n$, the cutoff of $\mathcal{F}_c^L$ disappears.
\end{ex}

The proofs of Theorems \ref{t-tv3} and \ref{t-tv4} and Corollary \ref{c-tv} are complicated and are given in Section \ref{s-proof}.

\section{Comparison of total variation cutoffs}\label{s-comp}

In this section, we make a comparison of cutoffs introduced in Sections \ref{s-tv} and \ref{s-bs}. To avoid confusion, we use $\mathcal{F},\mathcal{F}_c$ to denote families of birth and death chains without initial states specified and let $\mathcal{F}^L,\mathcal{F}_c^L$ and $\mathcal{F}^R,\mathcal{F}_c^R$ be families of chains started at respectively left and right boundary states. The following theorem is an immediate corollary of Theorems \ref{t-tv3} and \ref{t-tv4} and the proof is given in the end of this section.

\begin{thm}\label{t-tv5}
Let $\mathcal{F}=(\mathcal{X}_n,K_n,\pi_n)_{n=1}^\infty$ be a family of irreducible birth and death chains with $\mathcal{X}_n=\{0,...,n\}$ and $\mathcal{F}_c$ be the family of continuous time chains associated with $\mathcal{F}$. For any sequence $S=(x_n)_{n=1}^\infty$ with $x_n\in\mathcal{X}_n$, let $\mathcal{F}^S,\mathcal{F}_c^S$ be the families of chains in $\mathcal{F},\mathcal{F}_c$ for which the $n$th chain started at $x_n$.
\begin{itemize}
\item[(1)] If $\mathcal{F}_c^L$ and $\mathcal{F}_c^R$ have a total variation cutoff with cutoff time $r_n$ and $s_n$, then $\mathcal{F}_c$ has a maximum total variation cutoff with cutoff time $t_n$, where $t_n=\max\{r_n,s_n\}$.

\item[(2)] Let $M_n\in\mathcal{X}_n$ be a sequence of states satisfying
\[
 \inf_{n\ge 1}\pi_n([0,M_n])>0,\quad\inf_{n\ge 1}\pi_n([M_n,n])>0
\]
and let $S=(x_n)_{n=1}^\infty$, where $x_n\in\{0,n\}$ is a state such that
\[
 \max\left\{\mathbb{E}_{0}\widetilde{\tau}^{(n)}_{M_n},\mathbb{E}_{n}\widetilde{\tau}^{(n)}_{M_n}\right\}
 =\mathbb{E}_{x_n}\widetilde{\tau}^{(n)}_{M_n}
\]
and $\widetilde{\tau}^{(n)}_i$ is the first hitting time to state $i$ of the $n$th chain in $\mathcal{F}_c$. If $\mathcal{F}_c$ has a maximum total variation cutoff with cutoff time $t_n$, then $\mathcal{F}_c^S$ has a total variation cutoff with cutoff time $t_n$. In particular, $\mathcal{F}_c^S$ has a $(\mathbb{E}_{x_n}\widetilde{\tau}^{(n)}_{M_n},b_n)$ total variation cutoff with $b_n^2=\max\{\textnormal{Var}_0\widetilde{\tau}^{(n)}_{M_n},\textnormal{Var}_n\widetilde{\tau}^{(n)}_{M_n}\}$.
\end{itemize}
The above statements also apply for $\mathcal{F}$ under the assumption $\inf_{n,i}K_n(i,i)>0$.
\end{thm}

\begin{rem}
Let $\mathcal{F}_c,\widetilde{\tau}^{(n)}_i,M_n(a)$ be as in Theorem \ref{t-tv3}. By Theorem \ref{t-tv5}(2) and Remark \ref{r-cutoff}, if $\mathcal{F}_c$ has a maximum total variation cutoff, then
\[
 \mathbb{E}_{M_n(a)}\widetilde{\tau}^{(n)}_{M_n(b)}
 =o\left(\max\left\{\mathbb{E}_{0}\widetilde{\tau}^{(n)}_{M_n(c)},
 \mathbb{E}_{0}\widetilde{\tau}^{(n)}_{M_n(c)}\right\}\right),\quad\forall a,b,c\in(0,1).
\]
\end{rem}

The following example gives counterexamples to the converse of (1) and (2) in Theorem \ref{t-tv5}.

\begin{ex}\label{ex-cutoff2}
Consider the family $\mathcal{F}=(\mathcal{X}_n,K_n,\pi_n)_{n=1}^\infty$, where $\mathcal{X}_n=\{0,1,...,n\}$ and
\[
 \begin{cases}K_n(i,i+1)=1-\frac{i}{2n},\quad\forall 0\le i<n,\,i\ne i_n,\\
 K_n(i+1,i)=\frac{i+1}{2n},\quad\forall 0\le i<n-1,\,i\ne i_n,\quad K_n(n,n-1)=1,\\
 K_n(i_n,i_n+1)=c_n(1-\frac{i_n}{2n}),\quad K_n(i_n+1,i_n)=c_n\frac{i_n+1}{2n},\\
 K_n(i_n,i_n)=(1-c_n)(1-\frac{i_n}{2n}),\quad K_n(i_n+1,i_n+1)=(1-c_n)\frac{i_n+1}{2n},\end{cases}
\]
with $0\le i_n<n$ and $c_n\in[0,1]$, and
\[
 \pi_n(i)=2^{1-2n}\binom{2n}{i},\quad\forall 0\le i<n,\quad \pi_n(n)=2^{-2n}\binom{2n}{n}.
\]
As before, we use $M_n(a)$ to denote a state in $\mathcal{X}_n$ satisfying $\pi_n([0,M_n(a)])\ge a$ and $\pi_n([M_n(a),n])\ge 1-a$ and let $\widetilde{\tau}^{(n)}_i$ be the first hitting time to state $i$ of the continuous time chain associated with $(\mathcal{X}_n,K_n,\pi_n)$. Let $0<\lambda_{n,1}<\lambda_{n,2}<\cdots<\lambda_{n,n}$ be eigenvalues of $I-K_n$. It follows immediately from the central limit theorem that
\begin{equation}\label{eq-mna}
 n-M_n(a)\asymp \sqrt{n},\quad\forall a\in(0,1).
\end{equation}
In what follows, we discuss the total variation cutoffs of $\mathcal{F}_c$, $\mathcal{F}_c^L$ and $\mathcal{F}_c^R$ with specific $c_n$ and $i_n$.

First, assume that $c_n=1$ for all $n$. In this setting, the chain $(\mathcal{X}_n,K_n,\pi_n)$ is exactly the collapsed chain of the Ehrenfest model on $\{0,1,...,2n\}$ obtained by combining states $\{i,2n-i\}$ into a new state for $0\le i<n$. The spectral information of the Ehrenfest model is well-studied and this implies
\[
 \lambda_{n,i}=\frac{2i}{n},\quad\forall 1\le i\le n.
\]
By Theorem \ref{t-sep}, $\mathcal{F}_c$ has a maximum separation cutoff with cutoff time $\frac{1}{2}n\log n$ and, thus, has a maximum total variation cutoff. A simple computation with the Stirling formula gives
\[
 \pi_n(i)\asymp \frac{1}{\sqrt{n}},\quad\text{uniformly for }M_n(a)\le i\le n.
\]
By Lemma \ref{l-meanvar2}, this implies that, for $a\in(0,1)$,
\[
 \mathbb{E}_n\widetilde{\tau}^{(n)}_{M_n(a)}\asymp n,\quad \text{Var}_n\widetilde{\tau}_{M_n(a)}\asymp n^2,
\]
and, by Theorem \ref{t-sep2}, we have $\mathbb{E}_0\widetilde{\tau}^{(n)}_{M_n(a)}\sim \frac{1}{2}n\log n$ for any $a\in(0,1)$. As a consequence of Theorems \ref{t-tv3} and \ref{t-tv5}(2), $\mathcal{F}_c^R$ has no total variation cutoff, but $\mathcal{F}_c^L$ has with cutoff time $\frac{1}{2}n\log n$. Furthermore, by Theorem \ref{t-tv5}(1), the total variation cutoff time for $\mathcal{F}_c$ can be $\frac{1}{2}n\log n$. This gives a counterexample to the converse of Theorem \ref{t-tv5}(1).

Next, we consider the case $n-i_n=o(\sqrt{n})$ and $c_n$ is small. The second assumption means that a bottleneck arises between states $i_n$ and $i_n+1$. Under the first assumption, (\ref{eq-mna}) implies that, for $a\in(0,1)$, both $\mathbb{E}_0\widetilde{\tau}^{(n)}_{M_n(a)}$ and $\text{Var}_0\widetilde{\tau}^{(n)}_{M_n(a)}$ remain the same as in the case $c_n=1$. This implies that $\mathcal{F}_c^L$ has a total variation cutoff with cutoff time $\frac{1}{2}n\log n$. For the cutoff of $\mathcal{F}_c^R$, one may compute using the formula in Lemma \ref{l-meanvar2} that, for any $a\in(0,1)$,
\[
 \mathbb{E}_n\widetilde{\tau}^{(n)}_{M_n(a)}\asymp n+\frac{n-i_n}{c_n},\quad \text{Var}_n\widetilde{\tau}_{M_n(a)}\asymp
 \left(n+\frac{n-i_n}{c_n}\right)^2.
\]
Consequently, Theorem \ref{t-tv3} implies that $\mathcal{F}_c^R$ has no cutoff in total variation. Moreover, Theorem \ref{t-tv2} implies that if $(n-i_n)/c_n=o(n\log n)$, then $\mathcal{F}_c$ has a maximum total variation cutoff. If $n\log n=O((n-i_n)/c_n)$, then $\mathcal{F}_c$ has no maximum total variation cutoff, which gives a counterexample to the converse of Theorem \ref{t-tv5}(2).
\end{ex}

The next theorem provides more information on the comparison of cutoffs and should be regarded as a complement to Theorem \ref{t-tv5}.

\begin{thm}\label{t-comp}
Let $\mathcal{F}=\{(\mathcal{X}_n,K_n,\pi_n)_{n=1}^\infty$ be a family of birth and death chains with $\mathcal{X}_n=\{0,1,...,n\}$ and $\mathcal{F}_c$ be the family of continuous time chains associated with $\mathcal{F}$. Suppose that, in total variation, $\mathcal{F}_c^L$ has a cutoff with cutoff time $t_n$ but no subsequence of $\mathcal{F}_c^R$ has a cutoff. Let $M_n$ be a state in $\mathcal{X}_n$ and set
\[
 R=\limsup_{n\ra\infty}\frac{\mathbb{E}_n\widetilde{\tau}^{(n)}_{M_n}}{t_n},\quad\forall a\in(0,1).
\]
Then, the following are equivalent.
\begin{itemize}
\item[(1)] $\mathcal{F}_c$ has a maximum total variation cutoff. In particular, $t_n$ is a cutoff time.

\item[(2)] $R=0$ for some sequence $(M_n)_{n=1}^\infty$ satisfying
\begin{equation}\label{eq-mn}
 \inf_{n\ge 1}\pi_n([0,M_n])>0,\quad\inf_{n\ge 1}\pi_n([M_n,n])>0.
\end{equation}

\item[(3)] $R=0$ for any sequence $(M_n)_{n=1}^\infty$ satisfying \textnormal{(\ref{eq-mn})}.
\end{itemize}

The above statement also holds for $\mathcal{F}$ provided $\inf_{n,i}K_n(i,i)>0$.
\end{thm}

\begin{proof}
We first consider the continuous time case. Since $\mathcal{F}_c^L$ has a total variation cutoff with cutoff time $t_n$, Theorem \ref{t-tv3} implies
\begin{equation}\label{eq-tn}
 \mathbb{E}_0\widetilde{\tau}^{(n)}_{M_n(a)}\sim t_n,\quad \text{Var}_0\widetilde{\tau}^{(n)}_{M_n(a)}=o(t_n^2),\quad\forall a\in(0,1).
\end{equation}
Under the assumption of (\ref{eq-mn}), one may choose $0<a<b<1$ such that $M_n(a)\le M_n\le M_n(b)$. By (\ref{eq-tn}), this implies
\begin{equation}\label{eq-tn2}
 \mathbb{E}_0\widetilde{\tau}^{(n)}_{M_n}\sim t_n,\quad \text{Var}_0\widetilde{\tau}^{(n)}_{M_n}=o(t_n^2).
\end{equation}

(3)$\Ra$(2) is obvious. Now, we prove (2)$\Ra$(1) and assume that (2) holds.
Note that $R=0$ is equivalent to $\mathbb{E}_n\widetilde{\tau}^{(n)}_{M_n}=o(t_n)$. This implies $\text{Var}_n\widetilde{\tau}^{(n)}_{M_n}=o(t_n^2)$ using the fact $\text{Var}_n\widetilde{\tau}^{(n)}_i\le(\mathbb{E}_n\widetilde{\tau}^{(n)}_i)^2$. Along with (\ref{eq-tn2}), we may conclude
\begin{equation}\label{eq-tn3}
 \sqrt{\max\left\{\text{Var}_0\widetilde{\tau}^{(n)}_{M_n},\text{Var}_n\widetilde{\tau}^{(n)}_{M_n}\right\}}
 =o\left(\max\left\{\mathbb{E}_0\widetilde{\tau}^{(n)}_{M_n},\mathbb{E}_n\widetilde{\tau}^{(n)}_{M_n}\right\}\right).
\end{equation}
By Theorem \ref{t-tv2}, $\mathcal{F}_c$ has a maximum total variation cutoff with cutoff time $t_n$.

For (1)$\Ra$(3), we prove the equivalent implication by assuming that $R>0$ for some sequence $(M_n)_{n=1}^\infty$ satisfying (\ref{eq-mn}). Since $R>0$, we may choose a subsequence $(k_n)_{n=1}^\infty$ such that
\begin{equation}\label{eq-tn4}
 t_{k_n}=O\left(\mathbb{E}_{k_n}\widetilde{\tau}^{(k_n)}_{M_{k_n}}\right).
\end{equation}
As the subfamily of $\mathcal{F}_c^R$ indexed by $(k_n)_{n=1}^\infty$ is assumed to have no total variation cutoff, we may refine, by Theorem \ref{t-tv3}, the selection of $k_n$ such that
\begin{equation}\label{eq-tn5}
 \sqrt{\text{Var}_{k_n}\widetilde{\tau}^{(k_n)}_{M_{k_n}}}\asymp\mathbb{E}_{k_n}\widetilde{\tau}^{(k_n)}_{M_{k_n}}.
\end{equation}
Combining (\ref{eq-tn2}) with the above discussion leads to
\[
 \sqrt{\max\left\{\text{Var}_0\widetilde{\tau}^{(k_n)}_{M_{k_n}},\text{Var}_{k_n}\widetilde{\tau}^{(k_n)}_{M_{k_n}}\right\}}
 \asymp\max\left\{\mathbb{E}_0\widetilde{\tau}^{(k_n)}_{M_{k_n}},\mathbb{E}_{k_n}\widetilde{\tau}^{(k_n)}_{M_{k_n}}\right\}.
\]
By Theorem \ref{t-tv2}, the subfamily of $\mathcal{F}_c$ indexed by $(k_n)$ has no maximum total variation cutoff.

Next, we consider the discrete time case. (3)$\Ra$(2) is clear. For (2)$\Ra$(1), assume that $R=0$ for some sequence $M_n$ satisfying (\ref{eq-mn}). Observe that
\begin{equation}\label{eq->n}
 \mathbb{E}_0\tau^{(n)}_{M_n}+\mathbb{E}_n\tau^{(n)}_{M_n}\ge n.
\end{equation}
By Remark \ref{r-cutoffitme}, (\ref{eq->n}) implies $t_n\ra\infty$. Using Theorem \ref{t-tv4}, one may derive a discrete time version of (\ref{eq-tn}), (\ref{eq-tn2}) and (\ref{eq-tn3}). As a consequence of Theorem \ref{t-tv2}, $\mathcal{F}$ has a maximum total variation cutoff with cutoff time $t_n$.

For (1)$\Ra$(3), we assume the inverse of (3) that $R>0$ for some sequence $M_n$ satisfying (\ref{eq-mn}). Consider the following two cases.

{\bf Case 1:} $t_{k_n}\ra\infty$ for some subsequence $k_n$.

{\bf Case 2:} $t_{k_n}=O(1)$ for some subsequence $k_n$.

The proof of Case 1 is the same as the continuous time case. For Case 2, since the subfamily of $\mathcal{F}^L$ indexed by $(k_n)$ has a cutoff with cutoff time $t_{k_n}$, Remark \ref{r-cutoffitme} implies that
\[
 \mathbb{E}_0\tau^{(k_n)}_{M_{k_n}}=O(1),\quad \text{Var}_0\tau^{(k_n)}_{M_{k_n}}=O(1).
\]
By (\ref{eq->n}), we have $\mathbb{E}_{k_n}\tau^{(k_n)}_{M_{k_n}}\ra\infty$ and, by Theorem \ref{t-tv4}, we obtain a discrete version of (\ref{eq-tn4}) and then (\ref{eq-tn5}). Consequently, Theorem \ref{t-tv2} implies that $\mathcal{F}$ has no maximum total variation cutoff.
\end{proof}

The next theorem is a special version of Theorem \ref{t-tv5} which identifies two different cutoffs discussed in this section.
\begin{thm}\label{t-tv6}
Let $\mathcal{F}=(\mathcal{X}_n,K_n,\pi_n)_{n=1}^\infty$ be a family of irreducible birth and death chains with $\mathcal{X}_n=\{0,...,n\}$ and $\mathcal{F}_c$ be the families of continuous time chains associated with $\mathcal{F}$. Assume that $K_n(i,j)=K_n(n-i,n-j)$ for all $i,j\in\mathcal{X}_n$ and $n\ge 1$.
\begin{itemize}
\item[(1)] $\mathcal{F}^L_c$ has a total variation cutoff with cutoff time $t_n$ if and only if $\mathcal{F}_c$ has a maximum total variation cutoff with cutoff time $t_n$.

\item[(2)] Under the assumption that $\inf_{n,i}K_n(i,i)>0$, $\mathcal{F}^L$ has a total variation cutoff with cutoff time $t_n$ if and only if $\mathcal{F}$ has a maximum total variation cutoff with cutoff time $t_n$.
\end{itemize}
\end{thm}

\begin{proof}[Proof of Theorem \ref{t-tv5}(Continuous time case)]
As before, we use $\widetilde{\tau}^{(n)}_i$ to denote the first hitting time to state $i$ of the $n$th chain in $\mathcal{F}_c$ and use the notation $M_n(a)$ with $a\in(0,1)$ to denote a state in $\mathcal{X}_n$ satisfying $\pi_n([0,M_n(a)])\ge a$ and $\pi_n([M_n(a),n])\ge 1-a$.

For (1), assume that $\mathcal{F}_c^L,\mathcal{F}_c^R$ have total variation cutoffs with cutoff times $r_n,s_n$. By Theorem \ref{t-tv3}, we have
\[
 \sqrt{\text{Var}_0\widetilde{\tau}^{(n)}_{M_{n}(1/2)}}
 =o\left(\mathbb{E}_0\widetilde{\tau}^{(n)}_{M_{n}(1/2)}\right),
 \quad \mathbb{E}_0\widetilde{\tau}^{(n)}_{M_{n}(1/2)}\sim r_{n},
\]
and
\[
 \sqrt{\text{Var}_{n}\widetilde{\tau}^{(n)}_{M_{n}(1/2)}}
 =o\left(\mathbb{E}_{n}\widetilde{\tau}^{(n)}_{M_{n}(1/2)}\right),
 \quad \mathbb{E}_{n}\widetilde{\tau}^{(n)}_{M_{n}(1/2)}\sim s_{n}.
\]
Clearly, this implies
\[
 \sqrt{\max\left\{\text{Var}_0\widetilde{\tau}^{(n)}_{M_{n}(1/2)},\text{Var}_{n}\widetilde{\tau}^{(n)}_{M_{n}(1/2)}\right\}}
 =o\left(\max\left\{\mathbb{E}_0\widetilde{\tau}^{(n)}_{M_{n}(1/2)},\mathbb{E}_{n}\widetilde{\tau}^{(n)}_{M_{n}(1/2)}\right\}\right)
\]
and
\[
 \max\left\{\mathbb{E}_0\widetilde{\tau}^{(n)}_{M_{n}(1/2)},\mathbb{E}_{n}\widetilde{\tau}^{(n)}_{M_{n}(1/2)}\right\}
 \sim \max\{r_{n},s_n\}=t_n.
\]
By Theorem \ref{t-tv2}, $\mathcal{F}_c$ has a maximum total variation cutoff with cutoff time $t_n$.

For (2), let $\widehat{\mathcal{F}}=(\mathcal{X}_n,\widehat{K}_n,\widehat{\pi}_n)_{n=1}^\infty$ be a family given by
\[
 \widehat{K}_n=K_n,\quad\widehat{\pi}_n=\pi_n\quad\text{if }x_n=0,
\]
and
\[
 \widehat{K}_n(i,j)=K_n(n-i,n-j),\quad\widehat{\pi}_n(i)=\pi_n(n-i),\quad\forall i,j\in\mathcal{X}_n\quad\text{if }x_n=n.
\]
Let $\widehat{\mathcal{F}}_c$ be the family of continuous time chains associated with $\widehat{\mathcal{F}}$. Suppose that $\mathcal{F}_c$ has a maximum total variation cutoff with cutoff time $t_n$. It is obvious that $\widehat{\mathcal{F}}_c$ also has a maximum total variation cutoff with cutoff time $t_n$ and, to show that $\mathcal{F}_c^S$ has a total variation cutoff with cutoff time $t_n$, it is equivalent to prove that $\widehat{\mathcal{F}}_c^L$ has a total variation cutoff with cutoff time $t_n$.

Let $\widehat{\tau}^{(n)}_i$ be the first hitting time to state $i$ of the continuous time chain associated with $(\mathcal{X}_n,\widehat{K}_n,\widehat{\pi}_n)$ and set $\widehat{M}_n$ be a state defined by
\[
 \widehat{M}_n=\begin{cases}M_n&\text{if }x_n=0\\n-M_n&\text{if }x_n=n\end{cases}.
\]
We use $\widehat{M}_n(a)$ to denote a state such that
\[
 \widehat{\pi}_n([0,\widehat{M}_n(a)])\ge a,\quad \widehat{\pi}_n([\widehat{M}_n(a),n])\ge 1-a.
\]
By Theorem \ref{t-tv2}, the total variation cutoff of $\widehat{\mathcal{F}}_c$ with cutoff time $t_n$ implies
\[
 t_n\sim\max\left\{\mathbb{E}_0\widehat{\tau}^{(n)}_{\widehat{M}_n},
 \mathbb{E}_n\widehat{\tau}^{(n)}_{\widehat{M}_n}\right\}=\mathbb{E}_0\widehat{\tau}^{(n)}_{\widehat{M}_n}
\]
and, for any $a\in(0,1)$,
\begin{equation}\label{eq-cutoff}
 \sqrt{\max\left\{\text{Var}_0\widehat{\tau}^{(n)}_{\widehat{M}_n(a)},
 \text{Var}_{n}\widehat{\tau}^{(n)}_{\widehat{M}_{n}(a)}\right\}}=o(t_n)
 =o\left(\mathbb{E}_0\widehat{\tau}^{(n)}_{\widehat{M}_n}\right).
\end{equation}
As a result of Lemma \ref{l-hit2} and (\ref{eq-cutoff}), we have, for $0<b<a<1$,
\[
 \mathbb{E}_{\widehat{M}_n(b)}\widehat{\tau}^{(n)}_{\widehat{M}_n(a)}
 =O\left(\sqrt{\text{Var}_{\widehat{M}_n(b)}\widehat{\tau}^{(n)}_{\widehat{M}_n(a)}}\right)
 =o\left(\mathbb{E}_0\widehat{\tau}^{(n)}_{\widehat{M}_n}\right),
\]
which leads to
\[
 \mathbb{E}_0\widehat{\tau}^{(n)}_{\widehat{M}_n(a)}\sim \mathbb{E}_0\widehat{\tau}^{(n)}_{\widehat{M}_n},\quad\forall a\in(0,1).
\]
Applying the last identity to (\ref{eq-cutoff}) yields
\[
 \sqrt{\text{Var}_0\widehat{\tau}^{(n)}_{\widehat{M}_n(a)}}
 =o\left(\mathbb{E}_{0}\widehat{\tau}^{(n)}_{\widehat{M}_n(a)}\right),\quad\forall a\in(0,1).
\]
By Theorem \ref{t-tv3}, $\widehat{\mathcal{F}}_c^L$ has a total variation cutoff with cutoff time $t_n$. The precise description of the cutoff time and window is given by Theorem \ref{t-tv2}, Corollary \ref{c-tv} and Remark \ref{r-sep}.
\end{proof}

\begin{proof}[Proof of Theorem \ref{t-tv5}(Discrete time case)]
We use $\tau^{(n)}_i$ to denote the first hitting time to state $i$ of the $n$th chain in $\mathcal{F}$ and $M_n(a)$ for a state in $\mathcal{X}_n$ satisfying $\pi_n([0,M_n(a)])\ge a$ and $\pi_n([M_n(a),n])\ge 1-a$.

For (1), assume that $\mathcal{F}^L,\mathcal{F}^R$ have cutoffs with respective cutoff times $r_n,s_n$. Given an increasing sequence $\mathcal{K}=(k_n)_{n=1}^\infty$ in $\{1,2,...\}$, let $\mathcal{F}(\mathcal{K})$ be the family of chains in $\mathcal{F}$ indexed by the sequence $\mathcal{K}$. By Proposition 2.1 in \cite{CSal10}, to prove $\mathcal{F}$ has a maximum total variation cutoff, it suffices to show that, for any increasing sequence of positive integers, there is a subsequence, say $\mathcal{K}$, such that $\mathcal{F}(\mathcal{K})$ has a maximum total variation cutoff. Note that, by Remark \ref{r-cutoffitme}, $r_n+s_n$ must tend to infinity. This implies that $\mathcal{K}$ can be chosen to satisfy one of the following cases.

{\bf Case 1:} $r_{k_n}\ra\infty$ and $s_{k_n}\ra\infty$.

{\bf Case 2:} $r_{k_n}\ra\infty$ and $s_{k_n}=O(1)$.

{\bf Case 3:} $r_{k_n}=O(1)$ and $s_{k_n}\ra\infty$.

The proof for Case 1 is the same as the continuous time case. The proofs of Case 2 and Case 3 are similar and we discuss Case 2, here. By Theorem \ref{t-tv4} and Remark \ref{r-cutoffitme}, the cutoffs of $\mathcal{F}^L,\mathcal{F}^R$ imply that, for $a\in(0,1)$,
\[
 \mathbb{E}_0\tau^{(k_n)}_{M_{k_n}(a)}\sim r_{k_n},\quad \sqrt{\text{Var}_0\tau^{(k_n)}_{M_{k_n}(a)}}=o(r_{k_n}),
\]
and
\[
 \sqrt{\text{Var}_{k_n}\tau^{(k_n)}_{M_{k_n}(a)}}
 \le\mathbb{E}_{k_n}\tau^{(k_n)}_{M_{k_n}(a)}=O(1).
\]
This implies, for $a\in(0,1)$,
\[
 \sqrt{\max\left\{\text{Var}_0\tau^{(k_n)}_{M_{k_n}(a)},
 \text{Var}_{k_n}\tau^{(k_n)}_{M_{k_n}(a)}\right\}}
 =o\left(\max\left\{\mathbb{E}_0\tau^{(k_n)}_{M_{k_n}(a)},
 \mathbb{E}_{k_n}\tau^{(k_n)}_{M_{k_n}(a)}\right\}\right)
\]
and
\[
 \max\left\{\mathbb{E}_0\tau^{(k_n)}_{M_{k_n}(a)},
 \mathbb{E}_{k_n}\tau^{(k_n)}_{M_{k_n}(a)}\right\}\sim \max\{r_{k_n},s_{k_n}\}=t_{k_n}.
\]
By Theorem \ref{t-tv2}, $\mathcal{F}(\mathcal{K})$ has a maximum total variation cutoff with cutoff time $t_{k_n}$.

For (2), based on the following observation
\[
 n\le \mathbb{E}_0\tau^{(n)}_i+\mathbb{E}_n\tau^{(n)}_i,\quad\forall 0\le i\le n,
\]
we have $\mathbb{E}_{x_n}\tau^{(n)}_{M_n}\ra \infty$. The remaining proof is similar to the continuous time case and is skipped.
\end{proof}

\section{Proof of Theorems \ref{t-tv3}, \ref{t-tv4} and Corollary \ref{c-tv}}\label{s-proof}

This section is dedicated to the proof of Theorems \ref{t-tv3} and \ref{t-tv4} and we need the following lemmas.

\begin{lem}\label{l-hit2}
Let $(\mathcal{X},K,\pi)$ be an irreducible birth and death chain on $\{0,1,...,n\}$ and $\tau_i,\widetilde{\tau}_i$ be the first hitting times to state $i$ of the discrete time chain and the associated continuous time chain. Let $\lambda_i$ be the smallest eigenvalue of the submatrix of $I-K$ indexed by $0,...,i-1$. Then, for $i<j$,
\[
 \frac{\pi([0,i])}{2\pi([0,j-1])}(\mathbb{E}_i\widetilde{\tau_j})^2\le
 \textnormal{Var}_i(\widetilde{\tau}_j)\le\frac{2}{\lambda_j}\mathbb{E}_i\widetilde{\tau}_j
\]
and
\[
 \frac{\delta\pi([0,i])}{2\pi([0,j-1])}(\mathbb{E}_i\tau_j)^2\le
 \textnormal{Var}_i(\tau_j)\le\frac{2}{\lambda_j}\mathbb{E}_i\tau_j,
\]
where $\delta=\min_iK(i,i)$. In particular,
\[
 \mathbb{E}_i\tau_j=\mathbb{E}_i\widetilde{\tau}_j\le\frac{4\pi([0,j-1])}{\pi([0,i])\lambda_j}.
\]
\end{lem}

\begin{lem}\label{l-hit1}
Let $K$ be the transition matrix of an irreducible birth and death chain on $\{0,1,...,n\}$ and $\widetilde{\tau}_i$ be the first hitting time to state $i$ for the continuous time chain associated with $K$. For $0<i\le n$ and $a\in(0,1)$,
\[
 \mathbb{P}_0(\widetilde{\tau}_i>a\mathbb{E}_0\widetilde{\tau}_i)\ge \min\left\{e^{-\sqrt{a}},\frac{(1-a)^2}{\sqrt{a}+(1-a)^2}\right\}.
\]
\end{lem}

\begin{lem}\label{l-hit3}
Let $K$ be the transition matrix of an irreducible birth and death chain on $\mathcal{X}=\{0,1,...,n\}$ with transition rates $p_i,q_i,r_i$ and stationary distribution $\pi$. Let $\tau_i,\widetilde{\tau}_i$ be as in Lemma \ref{l-hit2}. Then, for $i<j<k$,
\[
 \mathbb{E}_j\min\{\tau_i,\tau_k\}=\mathbb{E}_j\min\{\widetilde{\tau}_i,\widetilde{\tau}_k\}=A/B,
\]
where
\[
 A=\sum_{\begin{subarray}{c}i+1\le \ell_1\le j\\ j\le\ell_2\le k-1 \end{subarray}}\frac{\pi([\ell_1,\ell_2])}
 {\pi(\ell_1)q_{\ell_1}\pi(\ell_2)p_{\ell_2}},\quad B=\sum_{\ell=i}^{k-1}\frac{1}{\pi(\ell)p_{\ell}}.
\]
\end{lem}

\begin{lem}\label{l-unimodal}
Let $(\mathcal{X},K,\pi)$ be an irreducible birth and death chain on $\{0,1,...,n\}$ and $H_t=e^{-t(I-K)}$. Then,
\begin{itemize}
\item[(1)] $H_t(0,i)/\pi(i)\ge H_t(0,i+1)/\pi(i+1)$ for $0\le i<n$ and $t\ge 0$,

\item[(2)] Assume that $\min_iK(i,i)\ge 1/2$. Then, $K^m(0,i)/\pi(i)\ge K^m(0,i+1)/\pi(i+1)$ for $0\le i<n$ and $m\ge 0$.
\end{itemize}
\end{lem}

We relegate the proofs of Lemmas \ref{l-hit2}, \ref{l-hit1} and \ref{l-hit3} to the appendix and refer the reader to Lemma 4.1 in \cite{DLP10} for a proof of Lemma \ref{l-unimodal}.

\begin{proof}[Proof of Theorem \ref{t-tv3}]
We first prove the equivalence for cutoffs. Note that $\pi_n(0)\ra 0$ is necessary for the total variation cutoff since
\[
 \liminf_{n\ra\infty}d_{n,\text{\tiny TV}}^{(c)}(0,t)\le \liminf_{n\ra\infty}d_{n,\text{\tiny TV}}^{(c)}(0,0)=1-\limsup_{n\ra\infty}\pi_n(0).
\]
Under the assumption that $\pi_n(0)\ra 0$, it is easy to see that, for any $a\in(0,1)$, $M_n(a)\ge 1$ if $n$ is large enough. For $a\in(0,1)$ and $n\ge 1$ such that $M_n(a)\ge 1$, we let
\[
 \lambda_{n,1}(a)<\cdots<\lambda_{n,M_n(a)}(a)
\]
be the eigenvalues of the submatrix of $I-K_n$ indexed by $0,1,...,M_n(a)-1$. Clearly, $\lambda_n(a)=\lambda_{n,1}(a)$ and, by Lemma \ref{l-meanvar},
\[
 u_n(a)=\sum_{i=1}^{M_n(a)}\frac{1}{\lambda_{n,i}(a)},\quad v_n^2(a)=\sum_{i=1}^{M_n(a)}\frac{1}{\lambda^2_{n,i}(a)}.
\]
As in the proof of (\ref{eq-comp}), we have
\[
 \sqrt{u_n(a)\lambda_n(a)}\le\frac{u_n(a)}{v_n(a)}\le u_n(a)\lambda_n(a).
\]
This implies the equivalence of (2) and (3).

To prove the remaining equivalences, we let $d_{n,\text{\tiny TV}}^{(c)}$ be the total variation distance of the $n$th chains. By Lemma \ref{l-tv}, one has
\begin{equation}\label{eq-tv}
 d_{n,\text{\tiny TV}}^{(c)}(0,t)\begin{cases}\le\mathbb{P}_0(\widetilde{\tau}^{(n)}_i>t)+\pi_n([i+1,n]),\\
 \ge\mathbb{P}_0(\widetilde{\tau}^{(n)}_i>t)-\pi_n([0,i-1]).\end{cases}
\end{equation}
As a result of the one-sided Chebyshev inequality, this implies
\begin{equation}\label{eq-mixing}
 T_{n,\text{\tiny TV}}^{(c)}(0,\epsilon)\begin{cases}\le\mathbb{E}_0\widetilde{\tau}^{(n)}_i
 +\sqrt{(\tfrac{1-\delta}{\delta})\text{Var}_0(\widetilde{\tau}^{(n)}_i)}&\text{for } \epsilon=\delta+\pi_n([i+1,n]),\\
 \ge\mathbb{E}_0\widetilde{\tau}^{(n)}_i-\sqrt{(\tfrac{\delta}{1-\delta})\text{Var}_0(\widetilde{\tau}^{(n)}_i)}&
 \text{for }\epsilon=\delta-\pi_n([0,i-1]),\end{cases}
\end{equation}
where $\delta\in(0,1)$.

Now, we prove (2)$\Ra$(1) and assume that (2) holds. By the last inequality of Lemma \ref{l-hit2}, we have, for $0<\delta<\epsilon<1$,
\begin{equation}\label{eq-unvn}
 0\le u_n(\epsilon)-u_n(\delta)\le\frac{4\epsilon}{\delta\lambda_n(\epsilon)}\le\frac{4\epsilon v_n(\epsilon)}{\delta}=o(u_n(\epsilon)).
\end{equation}
Fix $\epsilon\in(0,1)$ and let $0<\epsilon_1<\epsilon<\epsilon_2<1$. By (\ref{eq-mixing}), the replacement of $i=M_n(\epsilon_2)$, $\delta=\epsilon_2-\epsilon$ in the first inequality and the replacement of $i=M_n(\epsilon_1)$, $\delta=1-\epsilon+\epsilon_1$ in the second inequality yield
\[
 \begin{cases}T_{n,\text{\tiny TV}}^{(c)}(0,1-\epsilon)\le u_n(\epsilon_2)+\sqrt{(\frac{1}{\epsilon_2-\epsilon}-1)}v_n(\epsilon_2)=(1+o(1))u_n(\epsilon_2),\\
 T_{n,\text{\tiny TV}}^{(c)}(0,1-\epsilon)\ge u_n(\epsilon_1)-\sqrt{(\frac{1}{\epsilon-\epsilon_1}-1)}v_n(\epsilon_1)=(1+o(1))u_n(\epsilon_1).\end{cases}
\]
As a result of (\ref{eq-unvn}), we obtain that $T_{n,\text{\tiny TV}}^{(c)}(0,\epsilon)=(1+o(1))u_n(\eta)$ for any $\epsilon,\eta\in(0,1)$,
which proves (1).

Next, we prove (4)$\Ra$(3). Assume that $(t_n)_{n=0}^\infty$ is a positive sequence satisfying $t_n=O(u_n(c))$ for all $c\in(0,1)$ and $a\in(0,1)$ is a constant such that
\begin{equation}\label{eq-lim1}
 \lim_{n\ra\infty}\mathbb{P}_0\left(\widetilde{\tau}^{(n)}_{M_n(b)}>(1-\epsilon)t_n\right)=1,\quad\forall b\in(a,1),
\end{equation}
and, for any $b\in(a,1)$, there corresponds a constant $\alpha_b\in(0,1)$ such that
\begin{equation}\label{eq-lim2} \limsup_{n\ra\infty}\mathbb{P}_0\left(\widetilde{\tau}^{(n)}_{M_n(b)}>(1+\epsilon)t_n\right)\le\alpha_b,
\end{equation}
for all $\epsilon\in(0,1)$. Note that $\lambda_n(a_2)\le\lambda_n(a_1)$ for $0<a_1<a_2<1$. To prove (3), it suffices to show that $t_n\lambda_n(b)\ra\infty$ for all $b\in(a,1)$. Now, we fix $b\in(a,1)$. Since $\pi_n(0)\ra 0$, it is clear that $M_n(b)\ge 1$ for $n$ large enough. By \cite{BS87}, if $M_n(b)\ge 1$, we may write $\widetilde{\tau}^{(n)}_{M_n(b)}=T_n(b)+S_n(b)$, where $T_n(b)$ and $S_n(b)$ are independent, $T_n(b)$ is an exponential random variable with parameter $\lambda_n(b)$ and $S_n(b)$ is a sum of independent exponential random variables with parameters $\lambda_{n,2}(b),...,\lambda_{n,M_n(b)}(b)$. Note that
\begin{align}
 \mathbb{P}_0\left(\widetilde{\tau}^{(n)}_{M_n(b)}>(1-\epsilon)t_n\right)&=\int_0^\infty
\lambda_n(b)e^{-\lambda_n(b)s}\mathbb{P}_0(S_n(b)>(1-\epsilon)t_n-s)ds\notag\\
&\le (1-e^{-\lambda_n(b)t})\mathbb{P}_0(S_n(b)>(1-\epsilon)t_n-t)+e^{-\lambda_n(b)t},\notag
\end{align}
where the inequality is obtained by separating the region of integration into $(0,t)$ and $[t,\infty)$,
and
\begin{align}
 \mathbb{P}_0\left(\widetilde{\tau}^{(n)}_{M_n(b)}>(1+\epsilon)t_n\right)&=\int_0^\infty
\lambda_n(b)e^{-\lambda_n(b)s}\mathbb{P}_0(S_n(b)>(1-\epsilon)t_n-s)ds\notag\\
&\ge\mathbb{P}_0(S_n(b)>(1+\epsilon)t_n-r)e^{-\lambda_n(b)r}.\notag
\end{align}
By (\ref{eq-lim1}) and (\ref{eq-lim2}), the replacement of $t=C/\lambda_n(b)$ and $r=2C/\lambda_n(b)$ with $C=\frac{1}{4}\log\frac{1}{\alpha_b}$ in the above inequalities yields that, for all $\epsilon\in(0,1)$,
\[
 \lim_{n\ra\infty}\mathbb{P}_0(S_n(b)>(1-\epsilon)t_n-C/\lambda_n(b))=1
\]
and
\[
 \limsup_{n\ra\infty}\mathbb{P}_0(S_n(b)>(1+\epsilon)t_n-2C/\lambda_n(b))\le\sqrt{\alpha_b}<1.
\]
As a consequence, for $\epsilon\in(0,1)$, if $n$ is large enough, one has
\[
 (1+\epsilon)t_n-2C/\lambda_n(b)\ge (1-\epsilon)t_n-C/\lambda_n(b),
\]
which implies $t_n\lambda_n(b)\ge C/(2\epsilon)$. This proves $t_n\lambda_n(b)\ra \infty$.

To finish the proof of those equivalences, it remains to show (1)$\Ra$(4). Assume that $\mathcal{F}_c$ has a cutoff with cutoff time $t_n$. The replacement of $i=M_n(a)$ in (\ref{eq-tv}) implies that, for all $\epsilon\in(0,1)$,
\begin{equation}\label{eq-liminf}
 \liminf_{n\ra\infty}\mathbb{P}_0\left(\widetilde{\tau}^{(n)}_{M_n(a)}>(1-\epsilon)t_n\right)\ge a
\end{equation}
and
\begin{equation}\label{eq-limsup}
\limsup_{n\ra\infty}
\mathbb{P}_0\left(\widetilde{\tau}^{(n)}_{M_n(a)}>(1+\epsilon)t_n\right)\le a.
\end{equation}
By the Markov inequality, (\ref{eq-liminf}) implies that $t_n=O(u_n(a))$ for all $a\in(0,1)$. As a result of Lemma \ref{l-hit1}, (\ref{eq-limsup}) implies that $u_n(a)=O(t_n)$ for all $a\in(0,1)$, which leads to $t_n\asymp u_n(a)$ for all $a\in(0,1)$.

To fulfill the requirement in (4), one has to prove that there is $a\in(0,1)$ such that
\begin{equation}\label{eq-convto1}
 \lim_{n\ra\infty}\mathbb{P}_0\left(\widetilde{\tau}^{(n)}_{M_n(a)}>(1-\epsilon)t_n\right)=1,\quad\forall \epsilon\in(0,1).
\end{equation}
To see the above limit, we fix $\epsilon\in(0,1)$ and show that, for any subsequence of positive integers, there is a further subsequence satisfying (\ref{eq-convto1}). Let $k_n$ be a subsequence of positive integers and set
\[
 R(a):=\lim_{b\ra 1}\liminf_{n\ra\infty}\frac{\mathbb{E}_{M_{k_n}(a)}\widetilde{\tau}^{(k_n)}_{M_{k_n}(b)}}{t_{k_n}}.
\]
Clearly, $R(a)$ is nonnegative and non-increasing in $a$.

We consider the following two cases of $R(a)$. First, assume that $R(a)=0$ for some $a\in(0,1)$ and let $b_n$ be a sequence in $(a,1)$ that converges to $1$. Since $R(b_1)=0$, we may choose $\ell_1\in\{k_1,k_2,...\}$ such that $\mathbb{E}_{M_{\ell_1}(a)}\widetilde{\tau}^{(\ell_1)}_{M_{\ell_1}(b_1)}<t_{\ell_1}/2$. Inductively, for $n\ge 1$, we may select, according to the fact $R(b_{n+1})=0$, a constant $\ell_{n+1}\in\{k_1,k_2,...\}$ satisfying $\ell_{n+1}>\ell_n$ and
\[
 \mathbb{E}_{M_{\ell_{n+1}}(a)}\widetilde{\tau}^{(\ell_{n+1})}_{M_{\ell_{n+1}}(b_{n+1})}<t_{\ell_{n+1}}/2^{n+1}.
\]
This implies
\[
 \mathbb{E}_{M_{\ell_n}(a)}\widetilde{\tau}^{(\ell_n)}_{M_{\ell_n}(b)}=o(t_{\ell_n}),\quad\forall b\in(a,1).
\]
By Lemma \ref{l-meanvar}, $u_n(a)\asymp t_n$ implies $1/\lambda_n(a)=O(t_n)$ and, by Lemma \ref{l-hit2}, this yields $\text{Var}_{M_{\ell_n}(a)}\widetilde{\tau}^{(\ell_n)}_{M_{\ell_n}(b)}=o(t_{\ell_n}^2)$ for all $b\in(a,1)$. As a consequence of the one-sided Chebyshev inequality, we obtain
\[
 \lim_{n\ra\infty}\mathbb{P}_{M_{\ell_n}(a)}\left(\widetilde{\tau}^{(\ell_n)}_{M_{\ell_n}(b)}\le\eta t_{\ell_n}\right)=1,\quad\forall b\in(a,1),\,\eta>0.
\]
This leads to
\begin{align}
 &\liminf_{n\ra\infty}\mathbb{P}_0\left(\widetilde{\tau}^{(\ell_n)}_{M_{\ell_n}(a)}
 >(1-\epsilon)t_{\ell_n}\right)\notag\\
 \ge&\liminf_{n\ra\infty}\mathbb{P}_0\left(\widetilde{\tau}^{(\ell_n)}_{M_{\ell_n}(b)}
 >(1-\epsilon/2)t_{\ell_n},
 \widetilde{\tau}^{(\ell_n)}_{M_{\ell_n}(b)}-\widetilde{\tau}^{(\ell_n)}_{M_{\ell_n}(a)}
 \le\epsilon t_{\ell_n}/2\right)\notag\\
 =&\liminf_{n\ra\infty}\mathbb{P}_0\left(\widetilde{\tau}^{(\ell_n)}_{M_{\ell_n}(b)}
 >(1-\epsilon/2)t_{\ell_n}\right)\ge b,\notag
\end{align}
for all $b\in(a,1)$, where the last inequality uses (\ref{eq-liminf}). Letting $b$ tend to $1$ gives the desired limit.

Next, we assume that $R(a)>0$ for all $a\in(0,1)$. Along with this fact $u_n(a)\asymp t_n$ for all $a\in(0,1)$, it is easy to see that, for any $a\in(0,1)$, there is $b\in(a,1)$ such that $\mathbb{E}_{M_{k_n}(a)}\widetilde{\tau}^{(k_n)}_{M_{k_n}(b)}\asymp t_{k_n}$. To prove (\ref{eq-convto1}) for the subsequence $k_n$, we need the following discussion. For $n\ge 1$, set $H_{n,t}=e^{-t(I-K_n)}$ and let $(X_{n,t})_{t\ge 0}$ be a realization of the semigroup $H_{n,t}$ and, for $\eta\in(0,1)$, let
\[
 N_n(\eta)=\max\{0\le i\le n|H_{n,(1-\eta)t_n}(0,i)>\pi_n(i)\}.
\]
By Lemma \ref{l-unimodal}, we have
\[
 d_{n,\text{\tiny TV}}^{(c)}(0,(1-\eta)t_n)=H_{n,(1-\eta)t_n}(0,[0,N_n(\eta)])-\pi_n([0,N_n(\eta)]).
\]
Since $\mathcal{F}_c$ has a cutoff with cutoff time $(t_n)_{n=1}^\infty$, this implies
\[
 \lim_{n\ra\infty}H_{n,(1-\eta)t_n}(0,[0,N_n(\eta)])=1,\quad \lim_{n\ra\infty}\pi_n([0,N_n(\eta)])=0.
\]
Obviously, this yields
\begin{equation}\label{eq-convto1-2}
 \lim_{n\ra\infty}\mathbb{P}_0\left(X_{n,(1-\eta)t_n}\le M_n(a)\right)=1,\quad\forall a,\eta\in(0,1).
\end{equation}

Back to the case that $R(a)>0$ for all $a\in(0,1)$, one may choose $0<b<a^-<a<a^+<c<1$ such that
\begin{equation}\label{eq-sub}
 \mathbb{E}_{M_{k_n}(b)}\widetilde{\tau}^{(k_n)}_{M_{k_n}(a^-)}\asymp t_{k_n}\asymp
 \mathbb{E}_{M_{k_n}(a^+)}\widetilde{\tau}^{(k_n)}_{M_{k_n}(c)}.
\end{equation}
This implies that $M_{k_n}(b)<M_{k_n}(a^-)$ and $M_{k_n}(a^+)<M_{k_n}(c)$ for $n$ large enough. Next, let $L$ be a positive integer and set
\[
 \Delta_n=\Delta_n(L):=\frac{(1-\epsilon)t_n}{L}.
\]
Note that, for $0\le j\le L-1$,
\begin{align}
 &\mathbb{P}_0\left(\widetilde{\tau}^{(k_n)}_{M_{k_n}(a)}\in(j\Delta_{k_n},(j+1)\Delta_{k_n}],
 X_{k_n,(j+1)\Delta_{k_n}}\le M_{k_n}(b)\right)\notag\\
 \le&\mathbb{P}_0\left(\widetilde{\tau}^{(k_n)}_{M_{k_n}(a)}\in(j\Delta_{k_n},(j+1)\Delta_{k_n}]\right)
 \mathbb{P}_{M_{k_n}(a)}\left(\widetilde{\tau}^{(k_n)}_{M_{k_n}(b)}\le\Delta_{k_n}\right).\notag
\end{align}
By (\ref{eq-convto1-2}), summing up the above inequalities over $j$ and then passing $n$ to the infinity yields
\begin{align}
 &\limsup_{n\ra\infty}\mathbb{P}_0\left(\widetilde{\tau}^{(k_n)}_{M_{k_n}(a)}\le(1-\epsilon)t_{k_n}\right)\notag\\
 \le&\limsup_{n\ra\infty}\mathbb{P}_0\left(\widetilde{\tau}^{(k_n)}_{M_{k_n}(a)}\le(1-\epsilon)t_{k_n}\right)
 \times\limsup_{n\ra\infty}\mathbb{P}_{M_{k_n}(a)}\left(\widetilde{\tau}^{(k_n)}_{M_{k_n}(b)}\le\Delta_{k_n}\right).\notag
\end{align}
Observe that if there is $L>0$ such that
\begin{equation}\label{eq-convto<1}
 \limsup_{n\ra\infty}\mathbb{P}_{M_{k_n}(a)}\left(\widetilde{\tau}^{(k_n)}_{M_{k_n}(b)}\le\Delta_{k_n}\right)<1,
\end{equation}
then
\[
 \limsup_{n\ra\infty}\mathbb{P}_0\left(\widetilde{\tau}^{(k_n)}_{M_{k_n}(a)}\le(1-\epsilon)t_{k_n}\right)=0,
\]
as desired.  To get the limit in (\ref{eq-convto<1}), it suffices to show that there is $L>0$ such that
\[
 \limsup_{n\ra\infty}\mathbb{P}_{M_{k_n}(a)}\left(T_{k_n}\le\Delta_{k_n}\right)<1,
\]
where $T_n=\min\{\widetilde{\tau}^{(n)}_{M_{n}(b)},\widetilde{\tau}^{(n)}_{M_{n}(c)}\}$. By Lemma \ref{l-hit3}, $\mathbb{E}_{M_{k_n}(a)}T_{k_n}=A_{k_n}/B_{k_n}$, where
\[
 A_n=\sum_{\begin{subarray}{c}M_n(b)+1\le \ell_1\le M_n(a)\\ M_n(a)\le\ell_2\le M_n(c)-1  \end{subarray}}
 \frac{\pi_n([\ell_1,\ell_2])}{\pi_n(\ell_1)q_{n,\ell_1}\pi_n(\ell_2)p_{n,\ell_2}},\quad B_n=\sum_{\ell=M_n(b)}^{M_n(c)-1}\frac{1}{\pi_n(\ell)p_{n,\ell}}.
\]
It is easy to see from the first identity in Lemma \ref{l-meanvar2} that
\[
 A_n
 \ge(a^+-a^-)\mathbb{E}_{M_n(b)}\widetilde{\tau}^{(n)}_{M_n(a^-)}
 \mathbb{E}_{M_n(a^+)}\widetilde{\tau}^{(n)}_{M_n(c)},\quad B_n\le\mathbb{E}_{M_n(b)}\widetilde{\tau}^{(n)}_{M_n(c)}/b.
\]
Along with the fact that $u_n(a)\asymp t_n$ for all $a\in(0,1)$, one may apply (\ref{eq-sub}) to the above inequalities  to get $\mathbb{E}_{M_{k_n}(a)}T_{k_n}\asymp t_{k_n}$. Now, we choose $L>0$ such that
\[
 0<\mathbb{E}_{M_{k_n}(a)}T_{k_n}-\Delta_{k_n}\asymp t_{k_n},
\]
where the first inequality holds for $n$ large enough. Since $T_n\le \widetilde{\tau}^{(n)}_{M_n(c)}$, one also has
\begin{align}
 \text{Var}_{M_n(a)}T_n&\le\mathbb{E}_{M_n(a)}T_n^2\le\mathbb{E}_{M_n(a)}(\widetilde{\tau}^{(n)}_{M_n(c)})^2
 =\text{Var}_{M_n(a)}\widetilde{\tau}^{(n)}_{M_n(c)}+(\mathbb{E}_{M_n(a)}(\widetilde{\tau}^{(n)}_{M_n(c)})^2\notag\\
 &\le\text{Var}_{0}\widetilde{\tau}^{(n)}_{M_n(c)}+(\mathbb{E}_{0}\widetilde{\tau}^{(n)}_{M_n(c)})^2
 \le 2(\mathbb{E}_{0}\widetilde{\tau}^{(n)}_{M_n(c)})^2=2u_n(c)^2\asymp t_n^2.\notag
\end{align}
As a result of the one-sided Chebyshev inequality, this implies
\[
 \limsup_{n\ra\infty}\mathbb{P}_{M_{k_n}(a)}\left(T_{k_n}\le\Delta_{k_n}\right)\le\limsup_{n\ra\infty}
 \left(1+\frac{(\mathbb{E}_{M_{k_n}(a)}T_{k_n}-\Delta_{k_n})^2}{\text{Var}_{M_{k_n}(a)}T_{k_n}}\right)^{-1}<1.
\]

In the assumption of (2) and (3), the proof for choosing $(u_n(a))_{n=1}^\infty$ as a cutoff time is given in the proof for (2)$\Ra$(1). In the assumption of (4), the equivalence of cutoffs implies that $v_n(a)=o(u_n(a))$ for all $a\in(0,1)$. As a consequence of the Chebyshev inequality, this yields that, for all $a\in(0,1)$ and $c>0$,
\[
 \limsup_{n\ra\infty}\mathbb{P}_0\left(\left|\widetilde{\tau}^{(n)}_{M_n(a)}-u_n(a)\right|
 >cu_n(a)\right)\le\limsup_{n\ra\infty}\frac{v_n(a)^2}{c^2u_n(a)^2}=0.
\]
Along with the assumption of (4), one has $t_n\sim\mathbb{E}_0\widetilde{\tau}^{(n)}_{M_n(a)}$ for all $a\in(0,1)$, as desired.
\end{proof}

\begin{proof}[Proof of Theorem \ref{t-tv4}]
Set
\[
 \delta=\inf_{i,n}K_n(i,i),\quad K_n^{(\delta)}=(K_n-\delta I)/(1-\delta),\quad H^{(\delta)}_{n,t}=e^{t(K_n^{(\delta)}-I)}.
\]
It is easy to see that $(\mathcal{X}_n,K_n,\pi_n)$ and $(\mathcal{X}_n,H_{n,t}^{(\delta)},\pi_n)$ are respectively the $\delta$-lazy walk and the continuous time chain associated with $(\mathcal{X}_n,K_n^{(\delta)},\pi_n)$. Let $d_{n,\text{\tiny TV}},d_{n,\text{\tiny TV}}^{(c,\delta)}$ and $T_{n,\text{\tiny TV}},T_{n,\text{\tiny TV}}^{(c,\delta)}$ and $\tau_i^{(n)},\tau_i^{(n,\delta)}$ be respectively the total variation distances, the total variation mixing times and the first hitting times to state $i$ of chains $(\mathcal{X}_n,K_n,\pi_n)$ and $(\mathcal{X}_n,H_{n,t}^{(\delta)},\pi_n)$. As a result of the following observation
\begin{equation}\label{eq-kernelcomp}
 H_{n,t}^{(\delta)}=e^{t(K_n^{(\delta)}-I)}=e^{t(K_n-I)/(1-\delta)},
\end{equation}
it is easy to see that the ratio of the spectral gaps of $(\mathcal{X}_n,K_n,\pi_n)$ and $(\mathcal{X}_n,H_{n,t}^{(\delta)},\pi_n)$ is constant in $n$ and, further,
\begin{equation}\label{eq-comp3}
 \mathbb{E}_0\widetilde{\tau}^{(n,\delta)}_{M_n(a)}=(1-\delta)u_n(a),\quad
 \text{Var}_0\widetilde{\tau}^{(n,\delta)}_{M_n(a)}\asymp w_n(a),
\end{equation}
where the latter also uses Remark \ref{r-wcomp}. This is consistent with (\ref{eq-cut3}).

Set $\mathcal{F}^{(\delta)}_c=(\mathcal{X}_n,H_{n,t}^{(\delta)},\pi_n)_{n=1}^\infty$ and let $\mathcal{F}_c^{(\delta,L)}$ denote the family of chains in $\mathcal{F}_c^{(\delta)}$ started at the left boundary points. The remaining proof for the equivalence of (1), (2) and (3) is very similar to the proof of the discrete time case in Theorem \ref{t-tv2} if (\ref{eq-cut1}) and (\ref{eq-cut2}) hold under the replacement of $\mathcal{F},\mathcal{F}_c^{(\delta)}$ by $\mathcal{F}^L,\mathcal{F}_c^{(\delta,L)}$. These two equivalences are given by Theorem 3.4 in \cite{CSal13-1} but the prerequisite of this theorem asks the existence of some $\epsilon\in(0,1)$ such that $T_{n,\text{\tiny TV}}(0,\epsilon)\ra\infty$ and $T_{n,\text{\tiny TV}}^{(c,\delta)}(0,\epsilon)\ra\infty$. (The authors of \cite{CSal13-1} point out the observation that such a requirement is missed in their article.)
First, consider the requirement $T_{n,\text{\tiny TV}}^{(c,\delta)}(0,\epsilon)\ra\infty$. Recall the second inequality in Lemma \ref{l-tv} in the following
\[
 d_{n,\text{\tiny TV}}^{(c,\delta)}(0,t)\ge \mathbb{P}_0\left(\widetilde{\tau}_{M_n(a)}^{(n,\delta)}>t\right)-a.
\]
By Lemma \ref{l-hit1}, (\ref{eq-comp3}) and the fact $\text{Var}_0\widetilde{\tau}^{(n,\delta)}_{M_n(a)}
\le(\mathbb{E}_0\widetilde{\tau}^{(n,\delta)}_{M_n(a)})^2$, the above inequality implies
\[
 d_{n,\text{\tiny TV}}^{(c,\delta)}(0,\alpha(1-\delta)u_n(a))
 \ge\min\left\{e^{-\sqrt{\alpha}},\frac{(1-\alpha)^2}{\sqrt{\alpha}+(1-\alpha)^2}\right\}
 -a,\quad\forall \alpha\in(0,1).
\]
This yields that
\begin{equation}\label{eq-liminf2}
 \liminf_{n\ra\infty}\frac{T_{n,\text{\tiny TV}}^{(c,\delta)}(0,\epsilon)}{u_n(a)}>0\quad\text{for $\epsilon$ small enough}.
\end{equation}
Since $u_n(a)\ra\infty$, we have $T_{n,\text{\tiny TV}}^{(c,\delta)}(0,\epsilon)\ra\infty$ for $\epsilon$ small enough.

Next, we prove $T_{n,\text{\tiny TV}}(0,\epsilon)\ra\infty$. Note that one may use (\ref{eq-kernelcomp}) and the triangle inequality to derive
\begin{equation}\label{eq-comp2}
 d_{n,\text{\tiny TV}}^{(c,\delta)}(0,t)\le\mathbb{P}(N_t\le m)+\mathbb{P}(N_t>m)d_{n,\text{\tiny TV}}(0,m),
\end{equation}
where $(N_t)_{t\ge 0}$ is a Poisson process with parameter $1/(1-\delta)$. A simple application of the weak law of large numbers says that $N_t/t$ converges to $1/(1-\delta)$ in probability as $t$ tends to infinity. By (\ref{eq-liminf2}) and the assumption $u_n(a)\ra\infty$, the replacement of $t=\beta u_n(a)$ and $m=\lceil \beta u_n(a)\rceil$ in (\ref{eq-comp2}) with small $\beta$ implies that
\begin{equation}\label{eq-liminf3}
 \liminf_{n\ra\infty}\frac{T_{n,\text{\tiny TV}}(0,\epsilon)}{u_n(a)}>0\quad\text{for $\epsilon$ small enough}.
\end{equation}
This yields that $T_{n,\text{\tiny TV}}(0,\epsilon)\ra\infty$ for $\epsilon$ small enough.

To show (1)$\Lra$(4), let $(N_t)_{t\ge 0}$ be the Poisson process as before. It is easy to see from (\ref{eq-kernelcomp}) that if $(X_m^{(n)})_{m=0}^\infty$ is a realization of $(\mathcal{X}_n,K_n,\pi_n)$, then $(X_{N_t}^{(n)})_{t\ge 0}$ is a realization of $(\mathcal{X}_n,H_{n,t}^{(\delta)},\pi_n)$. This implies
\begin{equation}\label{eq-discts}
\begin{aligned}
 \mathbb{P}_0\left(\widetilde{\tau}^{(n,\delta)}_i>s\right)&=\mathbb{P}_0\left(X_{N_r}^{(n)}<i,\,\forall
 0\le r\le s\right)\\
 &=\mathbb{P}_0(X_m^{(n)}<i,\,\forall m\le N_s)=\mathbb{P}_0\left(\tau^{(n)}_i>N_s\right).
\end{aligned}
\end{equation}
Since $u_n(a)\ra\infty$ for some $a\in(0,1)$, we obtain
\[
 \mathcal{F}^L\text{ has a cutoff}\quad\Lra\quad\mathcal{F}_c^{(\delta,L)}\text{ has a cutoff}.
\]
By Theorem \ref{t-tv3}, the latter is equivalent to the existence of a sequence $t_n>0$ and a constant $a\in(0,1)$ satisfying
\begin{equation}\label{eq-2-4-1}
 t_n=O\left(\mathbb{E}_0\widetilde{\tau}^{(n,\delta)}_{M_n(c)}\right),\quad\forall c\in(0,1)
\end{equation}
and
\begin{equation}\label{eq-2-4-2}
 \lim_{n\ra\infty}\mathbb{P}_0\left(\widetilde{\tau}^{(n,\delta)}_{M_n(a)}>(1-\epsilon)t_n\right)=1,\quad\forall
 \epsilon\in(0,1)
\end{equation}
and, for any $b\in(a,1)$, there is $\alpha_b\in(0,1)$ such that
\begin{equation}\label{eq-2-4-3}
 \limsup_{n\ra\infty}\mathbb{P}_0\left(\widetilde{\tau}^{(n,\delta)}_{M_n(b)}>(1+\epsilon)t_n\right)\le\alpha_b,
 \quad\epsilon\in(0,1).
\end{equation}
As a result of (\ref{eq-comp3}), one can see that (\ref{eq-2-4-1}) is equivalent to $t_n=O(u_n(a))$ and further, by (\ref{eq-discts}), (\ref{eq-2-4-2}) implies
\begin{align}
 \liminf_{n\ra\infty}\mathbb{P}_0\left(\tau^{(n)}_{M_n(a)}>\frac{(1-\epsilon)t_n}{1-\delta}\right)
 &\ge\liminf_{n\ra\infty}\mathbb{P}_0\left(\tau^{(n)}_{M_n(a)}>N_{(1-\epsilon/2)t_n}\right)\notag\\
 &=\liminf_{n\ra\infty}\mathbb{P}_0\left(\widetilde{\tau}^{(n,\delta)}_{M_n(a)}>(1-\epsilon/2)t_n\right)=1.\notag
\end{align}
and (\ref{eq-2-4-3}) implies
\begin{align}
 \limsup_{n\ra\infty}\mathbb{P}_0\left(\tau^{(n)}_{M_n(b)}>\frac{(1+\epsilon)t_n}{1-\delta}\right)&
 \le\limsup_{n\ra\infty}\mathbb{P}_0\left(\tau^{(n)}_{M_n(b)}>N_{(1+\epsilon/2)t_n}\right)\notag\\
 &=\limsup_{n\ra\infty}\mathbb{P}_0\left(\widetilde{\tau}^{(n,\delta)}_{M_n(b)}>(1+\epsilon/2)t_n\right)\le\alpha_b,
 \notag
\end{align}
for all $\epsilon\in(0,1)$. This gives the desired properties in (4). Conversely, one may use a similar statement to prove (\ref{eq-2-4-2}) and (\ref{eq-2-4-3}) based on the observation of (4) and this part is omitted.

For a choice of the cutoff time, if (2) or (3) holds, the proof for the selected cutoff time is given by (\ref{eq-comp3}) and Theorem 3.4 in \cite{CSal13-1}. If (4) holds, the proof is exactly the same as that of Theorem \ref{t-tv3} and we skip it here.
\end{proof}

\begin{proof}[Proof of Corollary \ref{c-tv}]
The $(u_n(a),b_n)$ cutoff of $\mathcal{F}_c^L$ is immediately from (\ref{eq-mixing}) and Lemma \ref{l-hit2}. For the $(u_n(a),b_n)$ cutoff of $\mathcal{F}^L$, the assumption $\inf_nb_n>0$ and $b_n=o(u_n(a))$ implies that $u_n(a)\ra\infty$ for all $a\in(0,1)$, which means that the cutoff time tends to infinity. The remaining proof also uses Theorem 3.4 in \cite{CSal13-1} and is similar to the proof of the discrete time case in Theorem \ref{t-tv2}. We refer the reader to Section \ref{s-tv} for details.

\end{proof}

\section{Examples} \label{s-ex}

In this section, we consider some classical examples and use the developed theory to examine the existence of cutoff and, in particular, compute the cutoff time. First, we write $\mathcal{F}=(\mathcal{X}_n,K_n,\pi_n)_{n=1}^\infty$ for a family of irreducible birth and death chains with $\mathcal{X}_n=\{0,1,...,n\}$ and write $\mathcal{F}^L,\mathcal{F}^R$ for families of chains in $\mathcal{F}$ started at the left and right boundary states. For the continuous time case, those families are written as $\mathcal{F}_c,\mathcal{F}_c^L,\mathcal{F}_c^R$ instead. For $n\ge 1$, let $p_{n,i},q_{n,i},r_{n,i}$ be the birth, death and holding rates in $K_n$ and $\tau_i^{(n)},\widetilde{\tau}_i^{(n)}$ be the first hitting times to state $i$ of the $n$th chains in $\mathcal{F},\mathcal{F}_c$. For $a\in(0,1)$, $M_n(a)$ denotes a state in $\mathcal{X}_n$ satisfying $\pi_n([0,M_n(a)])\ge a$ and $\pi_n([M_n(a),n])\ge 1-a$.\vskip2mm

{\it (1) Biased random walk.} For $n\in\mathbb{N}$, let
\[
 p_{n,i}=r_{n,n}=p,\quad q_{n,i+1}=r_{n,0}=q,\quad\forall 0\le i<n,\,n\ge 1,
\]
with $q=1-p\in(0,1/2)$. Note that the stationary distribution satisfies
\[
 \pi_n(i)=\frac{p/q-1}{(p/q)^{n+1}-1}\left(\frac{p}{q}\right)^i,\quad\forall 0\le i\le n.
\]
This implies
\begin{equation}\label{eq-pq}
 \frac{\pi_n([0,i])}{\pi_n(i)}=\frac{p/q-(p/q)^{-i}}{p/q-1},\quad\forall 0\le i\le n.
\end{equation}
By Lemma \ref{l-meanvar2}, one has
\[
 \mathbb{E}_0\widetilde{\tau}^{(n)}_{n}=\sum_{i=0}^{n-1}\frac{\pi_n([0,i])}{p\pi_n(i)},\quad
 \zeta_{n,i}\le\text{Var}_i\widetilde{\tau}_{i+1}^{(n)}\le 2\zeta_{n,i},
\]
where
\[
 \zeta_{n,i}=\frac{1}{p^2\pi_n(i)}\sum_{\ell=0}^i\left(\frac{\pi_n([0,\ell])}{\pi_n(\ell)}\right)^2
 \pi_n(\ell).
\]
Applying (\ref{eq-pq}) to the computation of $\mathbb{E}_0\widetilde{\tau}^{(n)}_n$ and $\zeta_{n,i}$ yields
\[
 \mathbb{E}_0\widetilde{\tau}^{(n)}_{n}=\frac{n}{p-q}-\frac{p^2}{q(p-q)^2}
 \left(1-\left(\frac{q}{p}\right)^n\right),\quad
 \frac{1}{p^2}\le \zeta_{n,i}\le\frac{p}{(p-q)^3},\quad\forall i,
\]
where the bound of $\zeta_{n,i}$ leads to $\text{Var}_0\widetilde{\tau}^{(n)}_n\asymp n$.
Observe that $\pi_n([0,n])=1$ and $\pi_n(n)\ra 1-q/p$. As a consequence of Theorems \ref{t-sep2}, \ref{t-tv2} and \ref{t-tv5} with $M_n=n$, the families $\mathcal{F}_c,\mathcal{F}_c^L$ have a $(\frac{n}{p-q},\sqrt{n})$ cutoff in total variation and separation. To examine the existence of cutoff for $\mathcal{F}_c^R$, we fix $a\in(q/p^2-1,q/p)$. Based on the observation that $\pi_n(n-1)\ra(p-q)/p^2$, one has $M_n(a)=n-1$ for $n$ large enough and this implies $\text{Var}_n\widetilde{\tau}^{(n)}_{M_n(a)}=(\mathbb{E}_n\widetilde{\tau}^{(n)}_{M_n(a)})^2$. By Theorem \ref{t-tv3}, $\mathcal{F}_c^R$ has no cutoff in total variation.\vskip2mm

{\it (2) Metropolis chains for exponential distributions} Consider an increasing positive function $f$ on $(0,\infty)$. For $n\ge 1$, let $\pi_n(i)=\pi_n(0)f(i)$ and
\begin{equation}\label{eq-rates1}
 p_{n,i}=r_{n,0}=1/2,\quad q_{n,i+1}=\frac{f(i)}{2f(i+1)},\quad\forall 0\le i<n,
\end{equation}
and
\begin{equation}\label{eq-rates2}
 r_{n,i+1}=\frac{1}{2}-\frac{f(i)}{2f(i+1)},\quad\forall 0\le i<n-1,\quad r_{n,n}=1-\frac{f(n-1)}{f(n)}.
\end{equation}
One can check that the $n$th chain is the Metropolis chain for $\pi_n$ with base chain the simple random walk on $\mathcal{X}_n$ with holding probability $1/2$ at boundaries. We refer the reader to \cite{DS98} for details of Metropolis chains.

It is worthwhile to note that $K_n$ is monotonic, i.e. $p_{n,i}+q_{n,i+1}\le 1$ for all $0\le i<n$. By Corollary 4.2 in \cite{DLP10}, separation of the $n$th chain in $\mathcal{F},\mathcal{F}^L,\mathcal{F}^R$ (and respectively in $\mathcal{F}_c,\mathcal{F}_c^L,\mathcal{F}_c^R$) is the same. As a result of Theorem \ref{t-sep}, the existence of separation cutoff of $\mathcal{F}$ is equivalent to that of $\mathcal{F}_c$ and the cutoff time and window for $\mathcal{F}_c$ given by Theorem \ref{t-sep} is applicable to $\mathcal{F}$. For the total variation distance, if $\inf_{n,i}r_{n,i}>0$ is assumed, then Theorems \ref{t-tv2}, \ref{t-tv3} and \ref{t-tv4} and Remarks \ref{r-wcomp2} and \ref{r-wcomp} imply that the existence of cutoff of $\mathcal{F}$ (respectively $\mathcal{F}^L,\mathcal{F}^R$) is equivalent to that of $\mathcal{F}_c$ (respectively $\mathcal{F}_c^L,\mathcal{F}_c^R$). Furthermore, the cutoff times and windows for $\mathcal{F},\mathcal{F}_c$ given by Theorem \ref{t-tv2} (respectively for $\mathcal{F}^L,\mathcal{F}_c^L$ and for $\mathcal{F}^R,\mathcal{F}_c^R$ given by Theorems \ref{t-tv3} and \ref{t-tv4}) are consistent in the way that the cutoff times are equal and the cutoff windows are of the same order.

In this example, $f(x)=\exp\{\alpha x^\beta\}$ with $\alpha>0$ and $\beta>0$. Note that $\inf_{n,i}r_{n,i}>0$ if $\beta\ge 1$ and $\inf_{n,i}r_{n,i}=0$ if $\beta\in(0,1)$. In what follows, the cutoff phenomenon is discussed case by case according to $\beta$.

{\bf Case 1: $\beta>1$.} We first make some computations. Note that
\[
 \frac{d}{dx}f(x)=\alpha\beta x^{\beta-1}f(x)\ge\alpha\beta f(x)\quad\forall x\ge 1.
\]
This implies
\[
 \sum_{j=0}^if(j)\le 1+f(i)+\int_1^if(x)dx\le 1+f(i)+\frac{f(i)-f(1)}{\alpha\beta}\le\left(2+\frac{1}{\alpha\beta}\right)f(i).
\]
When $i$ tends to infinity, one has
\[
 \left(1-\frac{1}{i}\right)^\beta=1-\frac{\beta}{i}+O\left(\frac{1}{i^2}\right).
\]
This leads to
\[
 \frac{f(i-1)}{f(i)}=\exp\left\{-\alpha\beta i^{\beta-1}\left(1+O\left(\frac{1}{i}\right)\right)\right\}
 =O\left(\frac{1}{i^2}\right).
\]
As a result, we obtain
\[
 1\le \frac{\pi_n([0,i])}{\pi_n(i)}=1+\frac{\pi_n([0,i-1])}{\pi_n(i)}
 \le 1+\left(2+\frac{1}{\alpha\beta}\right)\frac{f(i-1)}{f(i)}
 =1+O\left(\frac{1}{i^2}\right).
\]
Replacing $i$ with $n$ gives $\pi_n(n)\ra 1$ and, by Lemma \ref{l-meanvar2}, one has
\[
 \mathbb{E}_0\widetilde{\tau}^{(n)}_n=2\sum_{i=0}^{n-1}\frac{\pi_n([0,i])}{\pi_n(i)}
 =2n+O(1)
\]
and
\[
 \text{Var}_i\widetilde{\tau}^{(n)}_{i+1}\asymp\frac{1}{\pi_n(i)}
 \sum_{\ell=0}^i\left(\frac{\pi_n([0,\ell])}{\pi_n(\ell)}\right)^2\pi_n(\ell)\asymp 1\quad\text{uniformly for }0\le i<n.
\]
The estimation of the variance implies $\text{Var}_0\widetilde{\tau}^{(n)}_n\asymp n$. By Theorem \ref{t-sep2}, \ref{t-tv2} and Theorem \ref{t-tv5}, both $\mathcal{F}_c$ and $\mathcal{F}_c^L$ have a $(2n,\sqrt{n})$ cutoff in total variation and separation. For the family $\mathcal{F}_c^R$, the observation, $\pi_n(n)\ra 1$, implies that the total variation mixing time of the $n$th chain is equal to $0$ when $n$ is large enough.

{\bf Case 2: $\beta=1$.} Set $\delta=(1-e^{-\alpha})/2$. Note that $(K_n-\delta I)/(1-\delta)$ is the biased random walk on $\mathcal{X}_n$ with $p=1/(1+e^{-\alpha})$. The result for biased random walks implies that $\mathcal{F}_c$ and $\mathcal{F}_c^L$ have a $(\frac{2n}{1-e^{-\alpha}},\sqrt{n})$ cutoff in total variation and separation but $\mathcal{F}_c^R$ has no total variation cutoff.

In Cases 1 and 2, one has $\inf_{n,i}r_{n,i}>0$. This implies that, in the total variation distance, the conclusion on the existence of cutoff, the cutoff time and the cutoff window also applies to $\mathcal{F}_c,\mathcal{F}_c^L,\mathcal{F}_c^R$.

{\bf Case 3: $0<\beta<1$.} First, observe that
\[
 \frac{d}{dx}(x^{1-\beta}f(x))=\alpha\beta f(x)+(1-\beta)x^{-\beta}f(x).
\]
This implies
\[
 \frac{i^{1-\beta}f(i)-j^{1-\beta}f(j)}{\alpha\beta+(1-\beta)j^{-\beta}}\le\int_j^if(x)dx\le\frac{i^{1-\beta}f(i)}{\alpha\beta},
 \quad\forall 1\le j<i.
\]
and then
\begin{equation}\label{eq-alphabeta1}
 \frac{1}{\alpha\beta+(1-\beta)j^{-\beta}}\left(1-\frac{j^{1-\beta} f(j)}{i^{1-\beta}f(i)}\right)\le\frac{f(1)+\cdots+f(i)}{i^{1-\beta}f(i)}
 \le \frac{1}{\alpha\beta}+i^{\beta-1}
\end{equation}
When $i\ge 2j$ and $j\ra \infty$, one has
\[
 \frac{f(j)}{f(i)}=\exp\left\{-\alpha i^\beta\left(1-\left(\frac{j}{i}\right)^\beta\right)\right\}=o\left(\frac{1}{i}\right).
\]
Consequently, we obtain, as $j\ra\infty$,
\begin{equation}\label{eq-alphabeta2}
 f(0)+\cdots+f(i)=i^{1-\beta}f(i)\left(\frac{1}{\alpha\beta}+O(j^{-\beta}+i^{\beta-1})\right)\quad\text{uniformly for $i\ge 2j$}.
\end{equation}
Replacing $i,j$ with $n,\lfloor n/2\rfloor$ in (\ref{eq-alphabeta2}) gives
\[
 \frac{1}{\pi_n(0)}=n^{1-\beta}f(n)\left(\frac{1}{\alpha\beta}+o(1)\right)\quad\text{as }n\ra\infty.
\]

Next, we fix $c>0$ and let $c_n$ be a sequence converging to $c$ such that $c_nn^{1-\beta}\in\mathcal{X}_n$. Set $M_n=n-c_nn^{1-\beta}$. Replacing $i,j$ with $M_n,\lfloor M_n/2\rfloor$ in (\ref{eq-alphabeta2}) yields
\[
 \lim_{n\ra\infty}\pi_n([0,M_n])=\lim_{n\ra\infty}\pi_n(0)\sum_{\ell=0}^{M_n}f(\ell)=e^{-c\alpha\beta}\in(0,1).
\]
By Lemma \ref{l-meanvar2}, one has, when $2j_n\le i_n\le M_n$ and $j_n\ra\infty$,
\[
 \mathbb{E}_0\widetilde{\tau}^{(n)}_{M_n}=2\sum_{\ell=0}^{M_n-1}\frac{f(0)+\cdots+f(\ell)}{f(\ell)}
 =\frac{2n^{2-\beta}}{\alpha\beta(2-\beta)}+O\left(n^{2-\beta} j_n^{-\beta}+i_n^{2-\beta}+n\right),
\]
where the second equality is given by separating $\sum_{\ell<M_n}$ into $\sum_{\ell<i_n}$ and $\sum_{i_n\le \ell<M_n}$ and then applying (\ref{eq-alphabeta1}) and (\ref{eq-alphabeta2}) respectively, and
\[
 \text{Var}_0\widetilde{\tau}^{(n)}_{M_n}\asymp\sum_{0\le\ell\le i<M_n}\frac{(f(0)+\cdots+f(\ell))^2}{f(i)f(\ell)}
 \asymp\sum_{i=0}^{M_n-1}\frac{1}{f(i)}\sum_{\ell=0}^i\ell^{2-2\beta}f(\ell),
\]
where the computation uses (\ref{eq-alphabeta1}). Observe that $4-3\beta>2-\beta$. Setting $j_n=\lfloor n^{1/2}\rfloor$ and $i_n=\lfloor n^{\frac{4-3\beta}{4-2\beta}}\rfloor$. Clearly, $i_n\ge 2j_n$ for $n$ large enough and, in the computation of expectation, this leads to
\[
 \mathbb{E}_0\widetilde{\tau}^{(n)}_{M_n}=\frac{2n^{2-\beta}}{\alpha\beta(2-\beta)}+O\left(n^{2-\frac{3}{2}\beta}+n\right).
\]
Applying the following fact
\[
 \frac{d}{dx}(x^{3-3\beta}f(x))=[(3-3\beta)x^{2-3\beta}+\alpha\beta x^{2-2\beta}]f(x)\asymp x^{2-2\beta}f(x),\quad\forall x\ge 1,
\]
to the computation of the variance yields
\[
 \text{Var}_0\widetilde{\tau}^{(n)}_{M_n}\asymp\sum_{i=0}^{M_n-1}i^{3-3\beta}\asymp n^{4-3\beta}.
\]
Similarly, one may use the observation that $\frac{f(M_n)}{f(n)}=\ra e^{-c\alpha\beta}$ to derive
\[
 q_{n,i}\asymp 1\quad\text{uniformly for $M_n\le i\le n$}.
\]
By Lemma \ref{l-meanvar2}, this implies
\[
 \mathbb{E}_n\widetilde{\tau}^{(n)}_{M_n}\asymp\sum_{i=M_n+1}^n\frac{f(i)+\cdots+f(n)}{f(i)}\asymp n^{2-2\beta}
\]
and
\[
 \text{Var}_n\widetilde{\tau}^{(n)}_{M_n}\asymp\sum_{M_n<i\le \ell\le n}\frac{(f(\ell)+\cdots+f(n))^2}{f(i)f(\ell)}\asymp n^{4-4\beta}.
\]
As a consequence of Theorem \ref{t-sep2}, \ref{t-tv2} and \ref{t-tv5}, $\mathcal{F}_c$ and $\mathcal{F}_c^L$ have a $(\frac{2n^{2-\beta}}{\alpha\beta(2-\beta)},n^{2-\frac{3}{2}\beta}+n)$ cutoff in total variation and separation but, by Theorem \ref{t-tv3}, $\mathcal{F}_c^R$ has no total variation cutoff. Note that, when $\beta\in(2/3,1)$, a better choice of the cutoff window is $n^{2-\frac{3}{2}\beta}$. To have this cutoff window, a more subtle estimation of the cutoff time is required.

We summarize the above results in the following theorem.

\begin{thm}\label{t-metro}
Let $f(x)=\exp\{\alpha x^\beta\}$ with $\alpha>0,\beta>0$. Consider the family $\mathcal{F}=(\mathcal{X}_n,K_n,\pi_n)_{n=1}^\infty$, where $\mathcal{X}_n=\{0,1,...,n\}$, $\pi_n(i)=\pi(0)f(i)$ and $K_n$ is a birth and death chain with transition rates
\[
 p_{n,i}=r_{n,0}=1/2,\quad q_{n,i+1}=\frac{f(i)}{2f(i+1)},\quad r_{n,i+1}=\frac{1}{2}-\frac{f(i)}{2f(i+1)},\quad\forall 0\le i<n.
\]
Then, $\mathcal{F}_c$ and $\mathcal{F}_c^L$ have a $(t_n,b_n)$ cutoff in total variation and separation but $\mathcal{F}_c^R$ has no total variation cutoff, where
\[
 t_n=\begin{cases}2n&\text{for }\beta>1\\\frac{2n}{1-e^{-\alpha}}&\text{for }\beta=1\\\frac{2n^{2-\beta}}{\alpha\beta(2-\beta)}&\text{for }0<\beta<1\end{cases},\quad
 b_n=\begin{cases}\sqrt{n}&\text{for }\beta\ge 1\\n^{2-\frac{3}{2}\beta}+n&\text{for }0<\beta<1\end{cases}.
\]
\end{thm}

\vskip2mm

{\it (3) Metropolis chains for polynomial distributions} In this example, we consider the family of Metropolis chains given by (\ref{eq-rates1}) and (\ref{eq-rates2}) with the replacement of $f(x)$ by $g(x)=\exp\{\alpha(\log(x+1))^\beta\}$, where $\alpha,\beta$ are positive. It has been shown in \cite{CSal13-2} that $\mathcal{F}_c$ has a cutoff in total variation and separation when $\beta>1$ but has no cutoff when $0<\beta\le 1$. The following theorem provides a cutoff time and a cutoff window when $\beta>1$.

\begin{thm}\label{t-metro2}
Let $g(x)=\exp\{\alpha(\log(x+1))^\beta\}$ with $\alpha>0$ and $\beta>1$. Consider the family $\mathcal{F}=(\mathcal{X}_n,K_n,\pi_n)_{n=1}^\infty$, where $\mathcal{X}_n=\{0,1,...,n\}$, $\pi_n(i)=\pi(0)g(i)$ and $K_n$ is a birth and death chain with transition rates
\[
 p_{n,i}=r_{n,0}=1/2,\quad q_{n,i+1}=\frac{g(i)}{2g(i+1)},\quad r_{n,i+1}=\frac{1}{2}-\frac{g(i)}{2g(i+1)},\quad\forall 0\le i<n.
\]
Then, $\mathcal{F}_c$ and $\mathcal{F}_c^L$ have a $(t_n,b_n)$ cutoff in total variation and separation but $\mathcal{F}_c^R$ has no total variation cutoff, where
\[
 t_n=\sum_{\ell=0}^N\frac{n^2}{\alpha\beta B_\ell(\log n)^{\beta+\ell-1}},\quad
 b_n=\frac{n^2}{(\log n)^{\frac{3}{2}(\beta-1)}}
\]
and $B_0=1$, $B_\ell=2^\ell(\beta-1)\beta\cdots(\beta+\ell-2)$, $N=\lceil\frac{\beta-3}{2}\rceil\ge 0$.
\end{thm}

\begin{rem}
Note that, in Theorem \ref{t-metro2}, $\beta+N-1<\frac{3}{2}(\beta-1)\le\beta+N$.
\end{rem}

The proof of Theorem \ref{t-metro2} is similar to the proof of the case $\beta\in(0,1)$ in Theorem \ref{t-metro} and is placed in the appendix.\vskip2mm

{\it (4) Metropolis chains for binomial distributions} For $n\ge 1$, let $\pi_n(i)=2^{-n}\binom{n}{i}$ and
\[
 p_{n,i}=q_{n,n-i}=\frac{1}{2},\quad q_{n,i+1}=p_{n,n-i-1}=\frac{i+1}{2(n-i)},\quad\forall 0\le i<n/2,
\]
and $r_{n,i}=1-p_{n,i}-q_{n,i}$ for $0\le i\le n$. It is easy to check that $K_n$ is the Metropolis chain for $\pi_n$ with base chain the simple random walk on $\mathcal{X}_n$ with holding probability $1/2$ at the boundary states. The separation cutoff of this family is proved in \cite{DS06} and we will discuss the cutoff time and the cutoff window in this example. First, one may use Lemma \ref{l-meanvar2} and (\ref{eq-ehren}) to derive
\begin{equation}\label{eq-ehrenmeanvar}
 \sum_{i=0}^{M_n-1}\frac{\pi_n([0,i])}{\pi_n(i)(1-\frac{i}{n})}=\frac{n\log n }{4}+O(n),\quad
 \sum_{0\le \ell\le i<M_n}\frac{\pi_n([0,\ell])^2}{\pi_n(\ell)\pi_n(i)}\asymp n^2,
\end{equation}
for any sequence $M_n\in\mathcal{X}_n$ satisfying $|M_n-\frac{n}{2}|=O(\sqrt{n})$. Note that
\[
 \pi_n(i)\left(1-\frac{i}{n}\right)=\frac{\pi_{n-1}(i)}{2},\quad \pi_n([0,i])=\pi_{n-1}([0,i])-\frac{\pi_{n-1}(i)}{2}.
\]
This implies
\begin{equation}\label{eq-pin3}
 \frac{\pi_n([0,i])}{\pi_n(i)(1-\frac{i}{n})}=\frac{2\pi_{n-1}([0,i])}{\pi_{n-1}(i)}-1.
\end{equation}
Set $M_n=\lfloor n/2\rfloor$. By Lemma \ref{l-meanvar2}, (\ref{eq-ehrenmeanvar}) and (\ref{eq-pin3}), we obtain
\[
 \mathbb{E}_0\widetilde{\tau}^{(n)}_{M_n}=\sum_{i=0}^{M_n-1}\frac{2\pi_n([0,i])}{\pi_n(i)}=
 \sum_{i=0}^{M_n-1}\left(\frac{\pi_{n+1}([0,i])}{\pi_{n+1}(i)(1-\frac{i}{n+1})}+1\right)=\frac{n\log n}{4}+O(n)
\]
and
\[
 \text{Var}_0\widetilde{\tau}^{(n)}_{M_n}\asymp \sum_{0\le \ell\le i<M_n}\frac{\pi_n([0,\ell])^2}{\pi_n(\ell)\pi_n(i)}\asymp n^2.
\]
In a similar way, one has
\[
 \mathbb{E}_n\widetilde{\tau}^{(n)}_{M_n}=\frac{n\log n}{4}+O(n),\quad \text{Var}_n\widetilde{\tau}^{(n)}_{M_n}\asymp n^2.
\]
As a consequence of Theorems \ref{t-sep2}, \ref{t-tv2} and \ref{t-tv5}, $\mathcal{F}_c$ has a $(\frac{1}{2}n\log n,n)$ separation cutoff and $\mathcal{F}_c,\mathcal{F}_c^L$ have a $(\frac{1}{4}n\log n,n)$ total variation cutoff.


\appendix

\section{Auxiliary results and proofs}

\begin{lem}\label{l-meanvar2}
Consider an irreducible birth and death chain on $\{0,1,...,n\}$ with transition rates $p_i,q_i,r_i$ and stationary distribution $\pi$. Let $\tau_i,\widetilde{\tau}_i$ be the hitting times in \textnormal{(\ref{eq-taui})}. Then, one has
\[
 \mathbb{E}_i\tau_{i+1}=\mathbb{E}_i\widetilde{\tau}_{i+1}=\frac{\pi([0,i])}{\pi(i)p_i},
\]
and
\[
 \textnormal{Var}_i(\tau_{i+1})=\frac{1}{p_i\pi(i)}\sum_{\ell=0}^i\pi(\ell)
 [\mathbb{E}_\ell\tau_{i+1}+\mathbb{E}_\ell\tau_i-1],
\]
and
\[
 \textnormal{Var}_i(\widetilde{\tau}_{i+1})=\frac{1}{p_i\pi(i)}\sum_{\ell=0}^i\pi(\ell)
 [\mathbb{E}_\ell\widetilde{\tau}_{i+1}+\mathbb{E}_\ell\widetilde{\tau}_i].
\]
\end{lem}
\begin{proof}
See \cite{BBF09} for a proof of the discrete time case. The continuous time case is a simple corollary of the discrete time case.
\end{proof}

\begin{proof}[Proof of Remark \ref{r-t-sep}]
Let $t_n,\lambda_{n,i},\lambda_n,\sigma_n,\rho_n$ be the notations in Theorem \ref{t-sep}. It has been proved in \cite{DS06} that
\begin{equation}\label{eq-sepcut1}
 \mathcal{F} \text{ has a separation cutoff} \quad \Lra\quad t_n\lambda_n\ra \infty
\end{equation}
and
\[
 \lfloor t_n-(1/\epsilon-1)^{1/2}\rho_n\rfloor\le T_{n,\text{\tiny sep}}(0,\epsilon)\le\lceil t_n+(1/\epsilon-1)^{1/2}\rho_n\rceil,\quad\forall\epsilon\in(0,1).
\]
These inequalities imply
\[
 |T_{n,\text{\tiny sep}}(0,\epsilon)-t_n|\le (1/\epsilon-1)^{1/2}\rho_n+1,\quad\forall\epsilon\in(0,1).
\]
Note that $\lambda_{n,i}\le 2$ for $1\le i\le n$. Clearly, this yields $t_n\ge n/2$.
As a consequence, if $\rho_n=o(t_n)$ or equivalently $\max\{\rho_n,1\}=o(t_n)$, then $\mathcal{F}$ has a $(t_n,\max\{\rho_n,1\})$ separation cutoff.

To see the inverse direction, note that
\[
 \max\{\rho_n^2,1/\lambda_n^2\}\le \frac{t_n}{\lambda_n}.
\]
This implies
\[
 \sqrt{t_n\lambda_n}\le\frac{t_n}{\max\{\rho_n,1/\lambda_n\}}\le t_n\lambda_n,
\]
and, as a result, we have
\begin{equation}\label{eq-sepcut2}
 t_n\lambda_n\ra\infty\quad\Lra\quad \max\{\rho_n,1/\lambda_n\}=o(t_n).
\end{equation}
By (\ref{eq-sepcut1}) and (\ref{eq-sepcut2}), $\mathcal{F}^L$ has a separation cutoff if and only if $\max\{\rho_n,1/\lambda_n\}=o(t_n)$. Further, if $\mathcal{F}$ has a separation cutoff, then $\rho_n=o(t_n)$.
\end{proof}

\begin{proof}[Proof of Lemma \ref{l-gap}]
Let $\pi$ be the stationary distribution of $K$. Since $\pi$ is a reversible measure for $K$, the spectra of $K,L_i$ are real. The interlacing property of $\lambda_j,\lambda^{(i)}_j$ is given by Theorem 4.3.8 of \cite{HJ90}. Clearly, this gives the first inequality $1/\lambda_1\le 1/\lambda^{(i)}_1$. Note that
\[
 \lambda_1^{(i)}=\min\left\{\frac{\langle (I-K)f,f\rangle_\pi}{\pi(f^2)}\bigg|f(i)=0\right\},
\]
where $\langle g,h\rangle_\pi=\sum_{j=0}^ng(j)h(j)\pi(j)$. By Proposition A.2 and Theorem 3.8 of \cite{CSal13-2}, one has
\[
 \frac{1}{4C(i)}\le\lambda_1^{(i)}\le\frac{1}{C(i)},\quad \frac{1}{4C(i)}\le\lambda_1\le\frac{1}{\min\{\pi([0,i]),\pi([i,n])\}C(i)},
\]
where
\[
 C(i)=\max\left\{\max_{0\le j<i}\sum_{\ell=j}^{i-1}\frac{\pi([0,j])}{\pi(\ell)K(\ell,\ell+1)},\max_{i<j\le n}\sum_{\ell=i+1}^j\frac{\pi([j,n])}{\pi(\ell)K(\ell,\ell-1)}\right\}.
\]
This gives the second inequality $1/\lambda^{(i)}\le(4/\min\{\pi([0,i]),\pi([i,n])\})/\lambda_1$.
\end{proof}

\begin{proof}[Proof of Lemma \ref{l-hit2}]
Let $p_k,q_k,r_k$ be the transition rates of $K$. We first consider the continuous time case. By Lemma \ref{l-meanvar2}, we have, for $0\le i<j\le n$,
\begin{align}
 \textnormal{Var}_i\widetilde{\tau}_j&\ge\sum_{k=i}^{j-1}\frac{1}{p_k\pi(k)}
 \sum_{\ell=0}^k\pi(\ell)\sum_{m=\ell}^k\mathbb{E}_m\widetilde{\tau}_{m+1}
 =\sum_{k=i}^{j-1}\frac{1}{p_k\pi(k)}\sum_{m=0}^k\pi([0,m])\mathbb{E}_m\widetilde{\tau}_{m+1}\notag\\
 &\ge\frac{\pi([0,i])}{\pi([0,j-1])}\sum_{k=i}^{j-1}\frac{\pi([0,k])}{p_k\pi(k)}
 \sum_{m=i}^k\mathbb{E}_m\widetilde{\tau}_{m+1}\notag\\
 &=\frac{\pi([0,i])}{\pi([0,j-1])}\sum_{k=i}^{j-1}
 \sum_{m=i}^k\mathbb{E}_k\widetilde{\tau}_{k+1}\mathbb{E}_m\widetilde{\tau}_{m+1}
 \ge\frac{\pi([0,i])}{2\pi([0,j-1])}(\mathbb{E}_i\widetilde{\tau}_j)^2.\notag
\end{align}
This proves the lower bound.

For the upper bound, let $a_1<\cdots<a_i$ and $b_1<\cdots<b_j$ be the eigenvalues of the submatrices of $I-K$ indexed respectively by $0,...,.i-1$ and $0,...,j-1$. By the strong Markov property, the first hitting time to state $i$ started at $0$ and the first hitting time to state $j$ started at $i$ are independent. By Lemma \ref{l-meanvar}, this implies
\[
 \mathbb{E}_i\widetilde{\tau}_j=\sum_{k=1}^j\frac{1}{b_k}-\sum_{k=1}^i\frac{1}{a_k},\quad
 \text{Var}_i\widetilde{\tau}_j=\text{Var}_0\widetilde{\tau}_j-\text{Var}_0\widetilde{\tau}_i=\sum_{k=1}^j\frac{1}{b_k^2}
 -\sum_{k=1}^i\frac{1}{a_k^2}.
\]
Inductively applying Theorem 4.3.8 of \cite{HJ90} yields the fact that
\[
 a_k>b_k,\quad\forall 1\le k\le i.
\]
As a result, we have
\begin{align}
 \text{Var}_i\widetilde{\tau}_j&=\sum_{k=1}^i\left(\frac{1}{b_k}-\frac{1}{a_k}\right)
 \left(\frac{1}{b_k}+\frac{1}{a_k}\right)+\sum_{k=i+1}^j\frac{1}{b_k^2}\notag\\
 &\le\frac{2}{b_1}\sum_{k=1}^i\left(\frac{1}{b_k}-\frac{1}{a_k}\right)+\frac{2}{b_1}\sum_{k=i+1}^j\frac{1}{b_k}
 =\frac{2}{b_1}\mathbb{E}_i\widetilde{\tau}_j.\notag
\end{align}

For the discrete time case, let $\delta=\min_iK(i,i)$ and set $K^{(\delta)}=(K-\delta I)/(1-\delta)$. Let $\tau^{(\delta)}_i,\widetilde{\tau}^{(\delta)}_i$ be the first hitting times to state $i$ of the discrete time and continuous time chains associated with $K^{(\delta)}$. Let $c_1<\cdots<c_i$ and $d_1<\cdots<d_j$ be the eigenvalues of the submatrices of $I-K^{(\delta)}$ indexed respectively by $0,...,.i-1$ and $0,...,j-1$. It is clear that $(1-\delta)c_1,...,(1-\delta)c_i$ and $(1-\delta)d_1,...,(1-\delta)d_j$ are the eigenvalues of the submatrices of $I-K$ indexed respectively by $0,...,.i-1$ and $0,...,j-1$. By Lemma \ref{l-meanvar}, we have
\[
 \mathbb{E}_i\tau_j=\sum_{k=1}^j\frac{1}{(1-\delta)d_k}-\sum_{k=1}^i\frac{1}{(1-\delta)c_k}
 =\frac{\mathbb{E}_i\widetilde{\tau}^{(\delta)}_j}{(1-\delta)}
\]
and
\[
 \text{Var}_i\tau_j=\sum_{k=1}^j\frac{1-(1-\delta)d_k}{(1-\delta)^2d_k^2}
 -\sum_{k=1}^i\frac{1-(1-\delta)c_k}{(1-\delta)^2c_k}
\]
The bounds for $\text{Var}_i\tau_j$ are immediately obtained by the result in the continuous time case and the following equalities.
\[
 \text{Var}_i\tau_j=\frac{\text{Var}_i\widetilde{\tau}^{(\delta)}_j}{(1-\delta)^2}
 -\frac{\mathbb{E}_i\widetilde{\tau}^{(\delta)}_j}{1-\delta}
 =\frac{\delta\text{Var}_i\widetilde{\tau}^{(\delta)}_j}{(1-\delta)^2}+\frac{\text{Var}_i\tau^{(\delta)}_j}{1-\delta}.
\]
\end{proof}

\begin{proof}[Proof of Lemma \ref{l-hit1}]
Let $\lambda_1,...,\lambda_i$ be the eigenvalues of the submatrix of $I-K$ indexed by $\{0,...,i-1\}$. By Lemma \ref{l-meanvar}, one has
\[
 \mathbb{E}_0\widetilde{\tau}_i=\sum_{k=1}^i\frac{1}{\lambda_k},
 \quad\textnormal{Var}_0\widetilde{\tau}_i=\sum_{k=1}^i\frac{1}{\lambda_k^2}.
\]
These identities imply $\textnormal{Var}_0\widetilde{\tau}_i\le\mathbb{E}_0\widetilde{\tau}_i/\lambda$, where $\lambda=\min\{\lambda_k|1\le k\le i\}$. As a result of the one-sided Chebyshev inequality, we have, for $a\in(0,1)$,
\[
 \mathbb{P}_0(\widetilde{\tau}_i>a\mathbb{E}_0\widetilde{\tau}_i)\ge 1-\frac{1}{1+(1-a)^2(\mathbb{E}_0\widetilde{\tau}_i)^2/\textnormal{Var}_0\widetilde{\tau}_i}\ge 1-\frac{1}{1+(1-a)^2\lambda\mathbb{E}_0\widetilde{\tau}_i}.
\]
Let $b$ be a positive constant. If $\lambda\mathbb{E}_0\widetilde{\tau}_i\ge b$, then
\[
 \mathbb{P}_0(\widetilde{\tau}_i>a\mathbb{E}_0\widetilde{\tau}_i)\ge 1-\frac{1}{1+(1-a)^2b}.
\]
Brown and Shao proved in \cite{BS87} that, under $\mathbb{P}_0$, $\widetilde{\tau}_i$ has the distribution as the sum of exponential random variables with parameters $\lambda_1,...,\lambda_i$. In the case of $\lambda\mathbb{E}_0\widetilde{\tau}_i\le b$, this leads to
\[
 \mathbb{P}_0(\widetilde{\tau}_i>a\mathbb{E}_0\widetilde{\tau}_i)\ge\exp\{-a\lambda\mathbb{E}_0\widetilde{\tau}_i\}\ge
 e^{-ab}.
\]
Summarizing both cases yields
\[
 \mathbb{P}_0(\widetilde{\tau}_i>a\mathbb{E}_0\widetilde{\tau}_i)\ge\min\left\{e^{-ab},1-\frac{1}{1+(1-a)^2b}\right\}.
\]
Taking $b=1/\sqrt{a}$ gives the desired inequality.
\end{proof}

\begin{proof}[Proof of Lemma \ref{l-hit3}]
For simplicity, we set $\tau=\min\{\tau_i,\tau_k\}$. The first equality is clear from the definition. To see the second equality, note that it follows immediately from the Markov property that
\[
 \mathbb{E}_j\tau=(\kappa_1+\cdots+\kappa_{k-i-1})
 \left(\frac{1+\gamma_1+\cdots+\gamma_{j-i-1}}{1+\gamma_1+\cdots+\gamma_{k-i-1}}\right)-(\kappa_1+\cdots+\kappa_{j-i-1}),
\]
where
\[
 \gamma_\ell=\frac{q_{i+1}q_{i+2}\cdots q_{i+\ell}}{p_{i+1}p_{i+2}\cdots p_{i+\ell}},\quad
 \kappa_\ell=\left(\frac{1}{q_{i+1}}+\frac{1}{q_{i+2}\gamma_1}+\cdots+\frac{1}{q_{i+\ell}\gamma_{\ell-1}}\right)\gamma_{\ell}.
\]
The proof of the above identity is somewhat complicated and we refer the reader to Equation (3.66) in \cite{PK11} for a proof. Observe that
\begin{equation}\label{eq-hittime4}
\begin{aligned}
 \mathbb{E}_j\tau(1+\gamma_1+\cdots+\gamma_{k-i-1})=&(\kappa_{j-i}+\cdots+\kappa_{k-i-1})\\
 &\qquad+\sum_{\begin{subarray}{c}1\le \ell_1\le j-i-1\\ j-i\le\ell_2\le k-i-1 \end{subarray}}(\gamma_{\ell_1}\kappa_{\ell_2}-\kappa_{\ell_1}\gamma_{\ell_2})
\end{aligned}
\end{equation}
and
\[
 \gamma_{\ell_1}\kappa_{\ell_2}-\kappa_{\ell_1}\gamma_{\ell_2}=\left(\frac{1}{q_{i+\ell_1+1}\gamma_{\ell_1}}
 +\cdots+\frac{1}{q_{i+\ell_2}\gamma_{\ell_2-1}}\right)\gamma_{\ell_1}\gamma_{\ell_2}
\]
In some computations, one can see that $\gamma_\ell=(\pi(i)p_i)/(\pi(i+\ell)p_{i+\ell})$. This implies
\[
 \kappa_{\ell}=\frac{\pi([i+1,i+\ell])}{\pi(i+\ell)p_{i+\ell}},\quad\sum_{\ell=\ell_1+1}^{\ell_2}
 \frac{1}{q_{i+\ell}\gamma_{\ell-1}}=\frac{\pi([i+\ell_1+1,i+\ell_2])}{\pi(i)p_i},
\]
and
\[
 \gamma_{\ell_1}\kappa_{\ell_2}-\kappa_{\ell_1}\gamma_{\ell_2}=\frac{\pi([i+\ell_1+1,i+\ell_2])\pi(i)p_i}
 {\pi(i+\ell_1)p_{i+\ell_1}\pi(i+\ell_2)p_{i+\ell_2}}.
\]
Putting the above identities back to (\ref{eq-hittime4}) gives
\begin{align}
 \mathbb{E}_j\tau\left(\pi(i)p_i\sum_{\ell=i}^{k-1}\frac{1}{\pi(\ell)p_\ell}\right)
 &=\sum_{\ell=j}^{k-1}\frac{\pi([i+1,\ell])}
 {\pi(\ell)p_\ell}+\sum_{\begin{subarray}{c}i+1\le \ell_1\le j-1\\ j\le\ell_2\le k-1 \end{subarray}}
 \frac{\pi([\ell_1+1,\ell_2])\pi(i)p_i}{\pi(\ell_1)p_{\ell_1}\pi(\ell_2)p_{\ell_2}}\notag\\
 &=\pi(i)p_i\sum_{\begin{subarray}{c}i\le \ell_1\le j-1\\ j\le\ell_2\le k-1 \end{subarray}}
 \frac{\pi([\ell_1+1,\ell_2])}{\pi(\ell_1)p_{\ell_1}\pi(\ell_2)p_{\ell_2}}\notag\\
 &=\pi(i)p_i\sum_{\begin{subarray}{c}i+1\le \ell_1\le j\\ j\le\ell_2\le k-1 \end{subarray}}
 \frac{\pi([\ell_1,\ell_2])}{\pi(\ell_1)q_{\ell_1}\pi(\ell_2)p_{\ell_2}},\notag
\end{align}
where the last equality uses the fact $\pi(i-1)p_{i-1}=\pi(i)q_i$.

\end{proof}

\begin{proof}[Proof of Theorem \ref{t-metro2}]
First, observe that
\[
 \frac{d}{dx}\left(\frac{(x+1)g(x)}{(\log(x+1))^{\beta-1}}\right)=\left(\alpha\beta+\frac{1-(\beta-1)/\log(x+1)}
 {(\log(x+1))^{\beta-1}}\right)g(x).
\]
Set $i_0=e^{\beta-1}$. For $i>j\ge i_0-1$, one has
\[
 \frac{A_{i,j}}{\alpha\beta+(\log(j+1))^{1-\beta}}\le\frac{(\log(i+1))^{\beta-1}}{(i+1)g(i)}\int_j^ig(x)dx\le\frac{1}{\alpha\beta},
\]
where
\[
 A_{i,j}=1-\frac{(j+1)(\log(j+1))^{1-\beta}g(j)}{(i+1)(\log(i+1))^{1-\beta}g(i)}>0.
\]
This implies, for $i>j\ge i_0-1$,
\begin{equation}\label{eq-alphabeta3}
 \frac{(\log(i+1))^{\beta-1}[g(0)+\cdots+g(i)]}{(i+1)g(i)}\begin{cases}
 \le\frac{1}{\alpha\beta}+\frac{i_0(\log(i+1))^{\beta-1}}{i+1},\\
 \ge\frac{A_{i,j}}{\alpha\beta+(\log(j+1))^{1-\beta}}.\end{cases}
\end{equation}
Note that, for $\log(i+1)\ge 2\log(j+1)$ and $i\ra\infty$,
\[
 \frac{g(j)}{g(i)}=\exp\left\{-\alpha(\log(i+1))^\beta
 \left(1-\frac{\log(j+1)}{\log(i+1)}\right)^\beta\right\}
 =o\left(\frac{1}{(\log(i+1))^{\beta-1}}\right).
\]
This implies that, when $\log(i+1)\ge 2\log(j+1)$ and $j\ra\infty$,
\[
 A_{i,j}=1+o\left((\log(j+1))^{1-\beta}\right).
\]
By (\ref{eq-alphabeta3}), one has, as $j\ra \infty$,
\begin{equation}\label{eq-alphabeta4}
 \frac{(\log(i+1))^{\beta-1}[g(0)+\cdots+g(i)]}{(i+1)g(i)}=\frac{1}{\alpha\beta}
 +O\left((\log(j+1))^{1-\beta}\right),
\end{equation}
uniformly for $\log(i+1)\ge 2\log(j+1)$.

Let $c_n$ be a sequence such that $c_nn(\log(n+1))^{1-\beta}\in\mathcal{X}_n$ and set $M_n=n[1-c_n(\log(n+1))^{1-\beta}]$. Suppose that $c_n$ converges to some positive constant $c$. Replacing $i,j$ with $n,\lfloor\sqrt{n}-1\rfloor$ and then with $M_n,\lfloor\sqrt{M_n}-1\rfloor$ in (\ref{eq-alphabeta4}) gives
\[
 \frac{1}{\pi_n(0)}\sim\frac{(n+1)g(n)}{\alpha\beta(\log(n+1))^{\beta-1}},\quad
 \lim_{n\ra\infty}\pi_n([0,M_n])=e^{-\alpha\beta c}.
\]
Next, we compute the expectation and variance of the first hitting time with initial state $0$. By Lemma \ref{l-meanvar2}, one has
\[
 \mathbb{E}_0\widetilde{\tau}^{(n)}_{M_n}=2\sum_{\ell=0}^{M_n-1}\frac{g(0)+\cdots+g(\ell)}{g(\ell)},\quad
 \text{Var}_0\widetilde{\tau}^{(n)}_{M_n}\asymp\sum_{0\le \ell\le i<M_n}\frac{(g(0)+\cdots+g(\ell))^2}{g(i)g(\ell)}.
\]
To sum up the right side of the above identities, we need the following computations. An application of the integration by parts gives that, for $k\in\{0,1,2,...\}$,
\begin{equation}\label{eq-int}
\begin{aligned}
 \int \frac{x+1}{(\log(x+1))^{\beta-1}}dx=&\sum_{\ell=0}^k\frac{(x+1)^2}{2B_\ell (\log(x+1))^{\beta+\ell-1}}\\
 &\qquad+\int\frac{x+1}{B_{k+1}(\log(x+1))^{\beta+k}}dx,
 \end{aligned}
\end{equation}
where $B_0=1$ and $B_\ell=2^\ell(\beta-1)\beta\cdots(\beta+\ell-2)$.
This implies, for $\ell_n\ra\infty$ and $\log(M_n)/\log(\ell_n)\ra\infty$,
\begin{equation}\label{eq-lnmn}
\begin{aligned}
 \int_{\ell_n}^{M_n} \frac{x+1}{(\log(x+1))^{\beta-1}}dx=&\sum_{\ell=0}^k\frac{(M_n+1)^2 }{2B_\ell (\log(M_n+1))^{\beta+\ell-1}}\\
 &\qquad+O\left(\frac{M_n^2}{(\log M_n)^{\beta+k}}\right).
\end{aligned}
\end{equation}
Using the following computations,
\[
 (M_n+1)^2=n^2\left(1+O\left(\frac{1}{(\log n)^{\beta-1}}\right)\right),
\]
and
\[
 \frac{1}{(\log(M_n+1))^{p}}
 =\frac{1}{(\log n)^{p}}\left(1+O\left(\frac{1}{(\log n)^\beta}\right)\right),\quad\forall p>0,
\]
one may rewrite (\ref{eq-lnmn}) as
\begin{equation}\label{eq-lnmn2}
 \int_{\ell_n}^{M_n} \frac{x+1}{(\log(x+1))^{\beta-1}}dx=\sum_{\ell=0}^k\frac{n^2}{2B_\ell(\log n)^{\beta+\ell-1}}
 +O\left(\frac{n^2}{(\log n)^{2\beta-2}}\right).
\end{equation}

Set $N=\lceil\frac{\beta-3}{2}\rceil\ge 0$ and let $i_n,j_n\in\mathcal{X}_n$ be states satisfying $\log(i_n+1)\ge 2\log(j_n+1)$, $j_n\ra\infty$ and $\log M_n/\log i_n\ra\infty$. By (\ref{eq-alphabeta3}) and (\ref{eq-int}) with $k=0$, one has
\[
 \sum_{\ell<i_n}\frac{g(0)+\cdots+g(\ell)}{g(\ell)}\asymp \int_1^{i_n}\frac{x+1}{(\log(x+1))^{\beta-1}}dx\asymp
 \frac{i_n^2}{(\log i_n)^{\beta-1}}=o(\log n),
\]
and, by (\ref{eq-alphabeta4}) and (\ref{eq-lnmn2}) with $k=N$, we get
\[
 \sum_{i_n\le\ell<M_n}\frac{g(0)+\cdots+g(\ell)}{g(\ell)}=\sum_{\ell=0}^N\frac{n^2}{2\alpha\beta B_\ell(\log n)^{\beta+\ell-1}}
 +O\left(\frac{n^2}{(\log n\log j_n)^{\beta-1}}\right).
\]
Putting both summations together and applying the setting, $j_n=\left\lfloor e^{\sqrt{\log n}}\}\right\rfloor-1$ and $i_n=\left\lceil e^{\sqrt{2(\log n)}}\right\rceil-1$, yields
\[
 \mathbb{E}_0\widetilde{\tau}^{(n)}_{M_n}=\sum_{\ell=0}^N\frac{n^2}{\alpha\beta B_\ell(\log n)^{\beta+\ell-1}}
 +O\left(\frac{n^2}{(\log n)^{\frac{3}{2}(\beta-1)}}\right).
\]
For the variance, note that
\begin{align}
 &\frac{d}{dx}\left(\frac{(x+1)^3}{(\log(x+1))^{3\beta-3}}g(x)\right)\notag\\
 =&\frac{(x+1)^2g(x)}{(\log(x+1))^{2\beta-2}}\left(\alpha\beta+\frac{3}
 {(\log(x+1))^{\beta-1}}-\frac{3(\beta-1)}{(\log(x+1))^\beta}\right)\notag
\end{align}
and
\[
 \frac{d}{dx}\left(\frac{(x+1)^4}{(\log(x+1))^{3\beta-3}}\right)=
 \frac{(x+1)^3}{(\log(x+1))^{3\beta-3}}\left(4-\frac{3(\beta-1)}{\log(x+1)}\right).
\]
By (\ref{eq-alphabeta3}), this implies
\[
 \sum_{0\le\ell\le i}\frac{(g(0)+\cdots+g(\ell))^2}{g(\ell)}\asymp\int_{i_0}^i
 \frac{(x+1)^2g(x)}{(\log(x+1))^{2\beta-2}}dx\asymp\frac{(i+1)^3g(i)}
 {(\log(i+1))^{3\beta-3}}
\]
and then
\begin{align}
 \text{Var}_0\widetilde{\tau}^{(n)}_{M_n}&\asymp\sum_{0\le i<M_n}\frac{(i+1)^3}{(\log(i+1))^{3\beta-3}}\asymp\int_{i_0}^{M_n}
 \frac{(x+1)^3}{(\log(x+1))^{3\beta-3}}dx\notag\\
 &\asymp\frac{(M_n+1)^4}{(\log(M_n+1))^{3\beta-3}}\asymp\frac{n^4}{(\log n)^{3\beta-3}}.\notag
\end{align}

Now, we compute the expectation and variance of the first hitting with initial state $n$. Note that
\[
 \lim_{n\ra\infty}\frac{g(M_n)}{g(n)}=e^{-\alpha\beta c}.
\]
This implies $\inf\{q_{n,i}|M_n<i<n,n\ge 1\}>0$ and, by Lemma \ref{l-meanvar2},
\[
 \mathbb{E}_n\widetilde{\tau}^{(n)}_{M_n}\asymp\sum_{M_n<i\le n}\frac{g(i)+\cdots+g(n)}{g(i)}\asymp (n-M_n)^2\asymp
 \frac{n^2}{(\log n)^{2\beta-2}}
\]
and
\[
 \text{Var}_n\widetilde{\tau}^{(n)}_{M_n}\asymp\sum_{M_n<i\le\ell\le n}\frac{(g(\ell)+\cdots+g(n))^2}{g(i)g(\ell)}\asymp
 (n-M_n)^4\asymp\frac{n^4}{(\log n)^{4\beta-4}}.
\]
The desired cutoff time and cutoff window are given by Theorems \ref{t-sep2}, \ref{t-tv2}, \ref{t-tv3} and \ref{t-tv5}.
\end{proof}


\end{document}